\documentclass[11pt,a4paper,reqno]{amsart}
\usepackage{setspace}
\usepackage{fullpage}
\usepackage{datetime}
\usepackage{indentfirst}
\usepackage{xcolor}
\usepackage{array}
\usepackage{enumitem}
\usepackage{amsmath}
\usepackage{amsfonts}
\usepackage{amssymb}
\usepackage{graphicx}
\usepackage{mathrsfs}
\usepackage{hyperref}
\usepackage[margin=2.5cm]{geometry}
\usepackage[normalem]{ulem}
\usepackage{tasks}
\usepackage{multicol}   
\usepackage{amsthm}
\usepackage{etoolbox}

\usepackage{subfig}
\usepackage[]{natbib} 

\usepackage{makecell}
\usepackage{float}
\usepackage{caption}
\usepackage{tikz}
\usetikzlibrary{patterns}
\usetikzlibrary{decorations.markings}

\usetikzlibrary{backgrounds}
\usetikzlibrary{arrows.meta}
\usetikzlibrary{shapes, arrows, calc, arrows.meta, fit, positioning}
\usetikzlibrary{%
	arrows,%
	arrows.meta,positioning,
	calc,%
	fit,%
	patterns,%
	plotmarks,%
	shapes.geometric,%
	shapes.misc,%
	shapes.symbols,%
	shapes.arrows,%
	shapes.callouts,%
	backgrounds,%
	chains,%
	topaths,%
	trees,%
	petri,%
	mindmap,%
	matrix,%
	folding,%
	fadings,%
	through,%
	positioning,%
	scopes,%
	decorations.fractals,%
	decorations.shapes,%
	decorations.text,%
	decorations.pathmorphing,%
	decorations.pathreplacing,%
	decorations.footprints,%
	decorations.markings,%
	shadows}

\usepackage{lipsum}
\makeatletter
\renewcommand\section{\@startsection{section}{1}%
  \z@{-1.0\linespacing\@plus-.7\linespacing}{1\linespacing}%
  {\normalfont\bfseries\large\centering}}
\renewcommand\subsection{\@startsection{subsection}{2}%
  \z@{-1\linespacing\@plus-.7\linespacing}{1\linespacing}%
  {\normalfont\large\bfseries\raggedright\hangindent=2em  \hangafter=1}}
\makeatother

\usepackage{hyperref}
\hypersetup{  
	colorlinks = true,    
	linkcolor = black,      
	urlcolor = black,        
	citecolor = black,     
	linkbordercolor = red,  
}

\tikzset{middlearrow/.style={
		decoration={markings,
			mark= at position 0.5 with {\arrow[very thick]{#1}},
		},
		postaction={decorate}
	}
}

\tikzset{midddlearrow/.style={
		decoration={markings,
			mark= at position 0.6 with {\arrow[very thick]{#1}},
		},
		postaction={decorate}
	}
}
\tikzset{lefttlearrow/.style={
		decoration={markings,
			mark= at position 0.7 with {\arrow[very thick]{#1}},
		},
		postaction={decorate}
	}
}
\tikzset{leftlearrow/.style={
		decoration={markings,
			mark= at position 0.55 with {\arrow[very thick]{#1}},
		},
		postaction={decorate}
	}
}
\newtheorem{theorem}{Theorem}[section]
\newtheorem{conjecture}{Conjecture}[section]

\newtheorem{definition}{Definition}
\newtheorem{lemma}[theorem]{Lemma}
\newtheorem{claim}{Claim}
\newtheorem{proposition}[theorem]{Proposition}

\newtheorem{question}{Question}

\theoremstyle{definition} 
\newtheorem{remark}{Remark}

\AtBeginEnvironment{theorem}{\setlength{\parskip}{0pt}}
\AtBeginEnvironment{conjecture}{\setlength{\parskip}{0pt}}
\AtBeginEnvironment{corollary}{\setlength{\parskip}{0pt}}
\AtBeginEnvironment{definition}{\setlength{\parskip}{0pt}}
\AtBeginEnvironment{lemma}{\setlength{\parskip}{0pt}}
\AtBeginEnvironment{claim}{\setlength{\parskip}{0pt}}
\AtBeginEnvironment{proposition}{\setlength{\parskip}{0pt}}
\AtBeginEnvironment{observation}{\setlength{\parskip}{0pt}}
\AtBeginEnvironment{case}{\setlength{\parskip}{0pt}}
\AtBeginEnvironment{problem}{\setlength{\parskip}{0pt}}
\AtBeginEnvironment{fact}{\setlength{\parskip}{0pt}}
\AtBeginEnvironment{question}{\setlength{\parskip}{0pt}}
\AtBeginEnvironment{remark}{\setlength{\parskip}{0pt}}

\title{Tur\'an-type and tiling problems in oriented graphs}

	\author{Ming Chen}
	\thanks{MC: School of Mathematics and Statistics, Jiangsu Normal University, Xuzhou 221000, China. Funded by Basic Research Program of Jiangsu (No.BK20251044), National Key Research and Development Program of China (No.2024YFA1013900), National Natural Science Foundation of China (No.12501483), and Natural Science Foundation of the Jiangsu Higher Education Institutions of China (No.25KJB110003).   \texttt{chenming314@jsnu.edu.cn}.}

    \author{Wenxu Lu}
    \thanks{WXL: School of Mathematics,  Shandong University, Jinan 250100, China. \texttt{wenxulu@mail.sdu.edu.cn}. }
	
	\author{Yun Wang}
	\thanks{YW: Data Science Institute, Shandong University, Jinan 250100, China. Funded by National Natural Science Foundation of China (No.12501489) and China Postdoctoral Science Foundation (No.2025M773101). \texttt{yunwang@sdu.edu.cn}.}
	
	\author{Zhiwei Zhang}
	\thanks{ZWZ: Interdisciplinary Center, Shandong University, Jinan 250100, China.  \texttt{zhiweizh@mail.sdu.edu.cn}.}

\begin{document} 

	\begin{abstract}

Given $a,b,c\in\mathbb N$, let $D_{a,b,c}$ be the tournament on $a+b+c$ vertices obtained by replacing the vertices of the directed triangle $C_3$ with transitive tournaments $TT_a$, $TT_b$, and $TT_c$, respectively. Keevash and Sudakov (2009) showed that every sufficiently large oriented graph $G$ on $n$ vertices with $\delta^{0}(G)\geqslant (1/2-o(1))n$ contains a $C_3$-tiling, equivalently a $D_{1,1,1}$-tiling, covering all but at most three vertices. We generalize this result to arbitrary blow-ups $D_{a,b,c}$. Specifically, for any fixed $a,b,c$, every sufficiently large oriented graph $G$ on $n$ vertices with $\delta^{0}(G)\geqslant (1/2-o(1))n$ contains a $D_{a,b,c}$-tiling covering all but at most $2(a+b+c)-3$ vertices. Moreover, this bound is essentially sharp. We also establish a stronger stability result: if $(a+b+c)\mid n$, then either $G$ contains a $D_{a,b,c}$-factor, or $G$ is close to an extremal graph.

Our interest in $D_{a,b,c}$ is also motivated by oriented Tur\'an theory: a seminal theorem of Bollob\'as and H\"aggkvist (1990) shows that a tournament $T$ is Tur\'anable (i.e., contained in every sufficiently large regular tournament) if and only if $T\subseteq D_{s,s,s}$ for some $s$.  Complementing our tiling results, we also investigate related semi-degree thresholds for powers of directed cycles and paths.  In particular, we present two $n$-vertex constructions that give lower bounds, showing that the minimum semi-degree thresholds for $C^2_l$ with $l\not\equiv 0\pmod 6$ and for $P^2_l$ with $l\geqslant 7$ are at least $4n/9$ and $3n/8$, respectively.
	\end{abstract}
	\maketitle
	 \noindent {2010 Mathematics Subject Classification:}  \texttt{05C20, 05C35, 05C38}

\section{Introduction}\label{SEC-introduction}
 Notation follows \cite{bang2009}, so we only repeat a few definitions here (also see Section \ref{SEC:notation-sketch}). A \emph{digraph} is not allowed to have parallel edges or loops and an \emph{oriented graph} is a digraph with no 2-cycle. A \emph{tournament} is an oriented graph in which every two vertices are adjacent. The \emph{minimum semi-degree} $\delta^0(G)$ of a digraph $G$ is the minimum of all the in-degrees and out-degrees of the vertices in $G$.  The \emph{minimum total degree} $\delta(G)$ of $G$ is the minimum number of edges incident to a  vertex in $G$.  A tournament $G$ on $n$ vertices with $\delta^0(G) = \lfloor \frac{n-1}{2} \rfloor$ is called a \emph{semi-regular tournament}. 
When $n$ is odd, an $n$-vertex semi-regular tournament is precisely a \emph{regular tournament}. Let $TT_k$ be the transitive tournament on $k$ vertices.  Given integers $a,b,c$ with  $1\leqslant a\leqslant b\leqslant  c$, we use $D_{a,b,c}$ to denote the tournament on $a+b+c$ vertices obtained from the directed triangle $C_3$ by replacing its vertices with $TT_a, TT_b$ and $TT_c$, respectively. 
For convenience, we use $D_s$ to denote $D_{s,s,s}$ for any $s\in \mathbb{N}$. Clearly, $D_1=D_{1,1,1}=C_3$.

For digraphs $H$ and $G$, an \emph{$H$-tiling} in $G$ is a family of pairwise disjoint copies of $H$ in $G$. 
An \emph{$H$-factor} is an $H$-tiling that covers all vertices of $G$. 
Two fundamental problems in graph theory are to find sufficient conditions on $G$ guaranteeing that $G$ contains a copy of $H$ (\textbf{Tur\'an-type problem}) or an $H$-factor (\textbf{tiling problem}). 
In this paper, we consider these two problems for oriented graphs.
We begin by reviewing some related results.

\medskip

\noindent\textbf{1. Tur\'an-type problem} 
\medskip

We say that an oriented graph $H$ is \emph{Tur\'anable} if there exists $n_0 \in \mathbb{N}$ such that every \emph{regular} tournament on $n \geqslant n_0$ vertices contains a copy of $H$.  Bollob\'as and  H\"{a}ggkvist \cite{bollobasJCTB50} considered the Tur\'anability of tournaments more than three decades ago by proving the following result.

\begin{theorem}[Bollob\'as, H\"{a}ggkvist~\cite{bollobasJCTB50}] \label{THM:bollobasJCTB50}
    A tournament $T$ is Tur\'anable if and only if  $T \subseteq D_s$ for some $s \in \mathbb{N}$.
\end{theorem}

Let $C_l$ and $P_l$ be the directed cycle and path of length $l$. For a directed path or cycle $L$, the \emph{$k$th power} of $L$, denoted by $L^{k}$, is the digraph obtained from $L$ by adding a
directed edge from $x$ to $y$ if there is a directed path of length at most $k$ from $x$ to $y$ on $L$. We usually call the 2nd power of $L$ the \emph{square} of $L$. 

DeBiasio et al. \cite{debiasioCPC35} showed that Theorem \ref{THM:bollobasJCTB50} also holds for powers of directed cycles.  
More precisely, they proved that given $l,k\in \mathbb{N}$ with $l\geqslant 2k+1$, $C_l^k$ is Tur\'anable if and only if $C_l^k \subseteq D_s$ for some $s\in \mathbb{N}$, or equivalently,  $C_l^k$ is Tur\'anable if and only if $l\geqslant3k$. 
DeBiasio et al. \cite{debiasioCPC35} also conjectured that  Theorem \ref{THM:bollobasJCTB50} can be extended to all oriented graphs. 
Very recently, Araujo and Xiang \cite{araujoARXIV2025} disproved this conjecture by building a $5$-vertex counterexample. Our first result shows that this conjecture has many counterexamples.

\begin{proposition}\label{PROP:manyF} 
There exists an infinite family of Tur\'anable oriented graphs, none of which is a subgraph of $D_s$ for any $s\in\mathbb{N}$.
\end{proposition}

As mentioned above, an oriented graph is  Tur\'anable if it is contained in every sufficiently large regular tournament $G$. 
It is then natural to ask  for the smallest number $g(n,H)$ such that $\delta^{0}(G)\geqslant g(n, H)$ guarantees the existence of a copy of $H$ in an $n$-vertex oriented graph $G$. We will consider the asymptotic normalized version of the minimum semi-degree threshold. We define
$$\kappa^0(H)=\lim_{n\rightarrow\infty}\frac{g(n,H)}{n}.$$
When $H$ is an oriented cycle, $\kappa^0(H)$ is a special case of the minimum semi-degree version of  the famous Caccetta--H\"aggkvist Conjecture   \cite{Caccetta1978}, and was studied extensively by K\"{u}hn, Osthus and Piguet \cite{kuhnEJC34}. Determining $\kappa^0(H)$ is highly non-trivial even when $H$ is of constant order. DeBiasio et al. \cite{debiasioCPC35} asked the following question.

\begin{question}[\cite{debiasioCPC35}]
Is it true that $\kappa^0(C_6^2)=2/5$?
\end{question}
In \cite{araujoARXIV2025}, Araujo and Xiang presented a construction which shows that $\kappa^0(C_6^2)\geqslant 3/7$.  In this paper, we give two lower bound constructions for $C_l^2$ and $P_l^2$, respectively, whose minimum semi-degrees are surprisingly large.

\begin{proposition}\label{Prop:Cl2-lowerbound}
$\kappa^0(C_l^2)\geqslant4/9$ for all $l\not\equiv 0 \pmod 6$, and $\kappa^0(P_l^2)\geqslant3/8$ for all $l\geqslant 7$.

Moreover, there exists a $4n/9$-regular $($resp., $3n/8$-regular$)$ oriented graph $G$ containing no $C_l^2$ $($resp., $P_l^2$$)$.
\end{proposition}

\medskip
\noindent\textbf{2. Tiling problem} 
\medskip

The tiling problem for oriented graphs has attracted considerable attention in recent years. In \cite{keevashJCTB99}, Keevash and Sudakov proved that, for sufficiently large $n$, every $n$-vertex oriented graph $G$ with $\delta^{0}(G)\geqslant (1/2-o(1))n$ contains a $C_3$-tiling covering all but at most three vertices. Subsequently, Li and Molla \cite{liEJC26} confirmed a conjecture of Cuckler \cite{cuckler2008} and Yuster \cite{yusterCSR1} by proving that every sufficiently large semi-regular tournament on $n\in 3\mathbb{N}$ vertices contains a $C_3$-factor. In \cite{wangJGT106}, Wang, Yan, and Zhang extended the result of Keevash and Sudakov \cite{keevashJCTB99} by showing that, for every integer $l\geqslant 4$ and sufficiently large $n\in l\mathbb{N}$, every $n$-vertex oriented graph with $\delta^{0}(G)\geqslant (1/2-o(1))n$ contains a $C_l$-factor.

Transitive tournament factors have also been studied extensively in recent years. Yuster \cite{yusterorder20} showed that every sufficiently large oriented graph $G$ on $n\in 3\mathbb{N}$ vertices with $\delta(G)\geqslant 5n/6$ contains a $TT_3$-factor. Treglown \cite{treglownJGT69} conjectured that the condition $\delta^0(G)\geqslant 7n/18$ guarantees a $TT_3$-factor, and he also provided an extremal example showing that this semi-degree threshold is best possible. This conjecture was later resolved for sufficiently large oriented graphs by Balogh, Lo, and Molla \cite{baloghJCTB124}. Furthermore, DeBiasio, Lo, Molla, and Treglown \cite{debiasioSJDM35} proved that $\delta(G)\geqslant (11/12+o(1))n$ suffices to guarantee a $TT_4$-factor in an oriented graph on $n\in 4\mathbb{N}$ vertices, and that this condition is asymptotically tight. In the same paper \cite{debiasioSJDM35}, they also established the best currently known general upper bound on the minimum total degree that guarantees a $TT_k$-factor for all $k\geqslant 5$. We remark that the oriented Ramsey number of $TT_k$ plays an important role in the study of the $TT_k$-factor problem, see \cite{debiasioSJDM35}.

The definition of $D_{a, b, c}$ suggests that $D_{a,b,c}$ (with $1\leqslant a\leqslant b\leqslant c$) contains a directed cycle of length $a+b+c$ and contains a $TT_a$-factor when $a=b=c$.
Hence, $D_{a, b, c}$ serves as a natural bridge between the tiling problems for cycles and those for transitive tournaments. In \cite{araujoARXIV2025}, Araujo and Xiang proved that every sufficiently large semi-regular tournament of order $n$, with $n\in 4\mathbb{N}$, contains a $D_{1,1,2}$-factor. They further asked for which integers $1\leqslant a \leqslant b\leqslant c$ every sufficiently large semi-regular tournament satisfying the necessary divisibility conditions admits a $D_{a,b,c}$-factor. In this paper, we answer this question completely.

Given $p, q\in \mathbb{N}$, let $\gcd(p,q)$ be the greatest common divisor between $p$ and $q$. In particular, $\gcd(p,0)=p$ for any positive integer $p$ and $\gcd(0,0)$ is undefined.

\begin{theorem}\label{THM-abc}
Given $a,b,c\in \mathbb{N}$ with $2\leqslant a\leqslant b\leqslant c$, there exists \(\eta>0\) such that the following holds for all sufficiently large \(n\in (a+b+c)\mathbb{N}\). If \(\gcd(a+b+c,c^2-ab)=1\), then every \(n\)-vertex oriented graph \(G\) with \(\delta^0(G)\geqslant (1/2-\eta)n\) contains a \(D_{a,b,c}\)-factor.
\end{theorem}

The remaining case \(a=1\) is settled by the following theorem.

\begin{theorem}\label{THM-1bc}
Given \(b,c\in \mathbb{N}\) with \(1\leqslant b\leqslant c\). If either \(3\mid (1+b+c)\) or \(\gcd(1+b+c,c^2-b)=1\), then every sufficiently large semi-regular tournament on \(n\in (1+b+c)\mathbb{N}\) vertices contains a \(D_{1,b,c}\)-factor.
\end{theorem}

Recall that Li and Molla \cite{liEJC26} showed that every sufficiently large semi-regular tournament on $n\in 3\mathbb{N}$ vertices contains a $C_3$-factor, i.e., $D_{1,1,1}$-factor. Theorem \ref{THM-1bc} clearly contains this result as a special case. Moreover, Theorem \ref{THM-1bc} generalizes the result of Araujo and Xiang on $D_{1,1,2}$-factors. Indeed, when $a=b=1$ and $c=2$, it holds that $\gcd(a+b+c,c^2-ab)=1$.

As mentioned in \cite{araujoARXIV2025}, for every $s\geqslant 2$, there exists a semi-regular tournament $T$ on $n\in s\mathbb{N}$ vertices without $D_s$-factors.  
The next two propositions extend this result to all $a,b,c\in \mathbb{N}$ and their proofs  will be presented in Appendix \ref{APP:B}.

\begin{proposition}\label{PRO-abcexamplegraph}
Let $a,b,c\in \mathbb{N}$ with $2\leqslant a\leqslant b\leqslant c$. If $\gcd(a+b+c,c^2-ab)>1$, then there exists a semi-regular tournament on \(n\in (a+b+c)\mathbb{N}\) vertices that does not contain a \(D_{a,b,c}\)-factor.
\end{proposition}

\begin{proposition}\label{PRO-1bcexamplegraph}
Let $b,c\in \mathbb{N}$ with $1\leqslant b\leqslant c$. If \(3\nmid (1+b+c)\) and \(\gcd(1+b+c,c^2-b)>1\), then there exists a semi-regular tournament on \(n\in (1+b+c)\mathbb{N}\) vertices that does not contain a \(D_{1,b,c}\)-factor.
\end{proposition}

Motivated by Propositions \ref{PRO-abcexamplegraph} and \ref{PRO-1bcexamplegraph}, it is natural to ask how large a \(D_{a,b,c}\)-tiling one can guarantee in oriented graphs that are close to regular tournaments. 
Our next result shows that, for any \(a,b,c\in \mathbb{N}\), every \(n\)-vertex oriented graph \(G\) with \(\delta^0(G)\geqslant (1/2-o(1))n\) contains an almost \(D_{a,b,c}\)-factor. 
Moreover, under the same minimum semi-degree condition, if \(G\) does not contain a \(D_{a,b,c}\)-factor, then \(G\) must be very close to a specific extremal structure, which will be defined later.

\begin{theorem}\label{THM:factor-extremal}
Given \(a,b,c\in \mathbb{N}\) with \(1\leqslant a \leqslant b \leqslant c\), there exist \(\eta,\gamma>0\) such that the following holds for all sufficiently large \(n\). Let \(G\) be an \(n\)-vertex oriented graph with \(\delta^0(G)\geqslant (1/2-\eta)n\). Then \(G\) contains a \(D_{a,b,c}\)-tiling covering all but at most \(2(a+b+c)-3\) vertices. In particular, if \(\gcd(a+b+c,c^2-ab)=1\), then \(G\) contains a \(D_{a,b,c}\)-tiling covering all but at most \(a+b+c-1\) vertices.

Moreover, if $n\in (a+b+c)\mathbb{N}$ but $G$ has no $D_{a,b,c}$-factor, then $V(G)$ has a partition $(V_1,V_2,V_3)$ satisfying  $|V_1|,|V_2|,|V_3| = n/3 \pm O(\gamma n)$ and $e(V_1,V_3),e(V_3,V_2),e(V_2,V_1) = O(\gamma n^2)$.
\end{theorem}

We remark that the constants $2(a+b+c)-3$ and $a+b+c-1$ cannot be improved; see Remark \ref{RE:2h-3}.
Meanwhile, the case $a=b=c=1$ of Theorem \ref{THM:factor-extremal} coincides with the main result in \cite{keevashJCTB99} by Keevash and Sudakov.

\medskip

\textbf{The rest of the paper is organized as follows.}   We begin with Section \ref{SEC:notation-sketch}, which contains some extra notation and terminology.  The important tools and some extra results are presented in Section \ref{SEC-pre}.  In Section \ref{SEC-main results}, we first present all necessary lemmas. Then we utilize them to finish the proofs of Theorems \ref{THM-abc}, \ref{THM-1bc} and \ref{THM:factor-extremal}.  The proofs of those necessary lemmas will be presented in the subsequent sections.  In the concluding remark section, we summarize our main results and discuss several related problems. Finally, proofs of the lower bound constructions in Propositions \ref{PROP:manyF}, \ref{Prop:Cl2-lowerbound}, \ref{PRO-abcexamplegraph}, and \ref{PRO-1bcexamplegraph} are given in the appendix.

\medskip

We remark that throughout the paper we may  assume w.l.o.g. that $1\leqslant a\leqslant b\leqslant c$. Indeed, by relabeling the three sets $TT_a,TT_b,TT_c$ of $D_{a,b,c}$ if necessary, we may  assume w.l.o.g. that $c\geqslant a,b$. Moreover, we may assume $b\geqslant a$ by reversing all edges of $D_{a,b,c}$ and $G$ if necessary.

\section{Notation}\label{SEC:notation-sketch}

 We now introduce some notation throughout the paper.
For integers $a<b$, $[a]$ will denote the set $\{1,2,\ldots, a\}$ and we will often write $t= b\pm a$ which means that $b-a\leqslant t\leqslant b+a$.  Given $x,h\in\mathbb{N}$, let $R_h(x)$ be the remainder of $x$ modulo $h$.   
We often denote a bipartite graph $G$ with partite sets $V_1,V_2$ by $G[V_1,V_2]$.

For a digraph $G$, we write $V(G)$ for its vertex set and $E(G)$ for its edge set. The order of $G$ is the number of vertices in $G$,  denoted by $|V(G)|$ or $|G|$ for simplicity.  We will denote an edge oriented from $u$ to $v$ by $uv$. We use $G[X]$ to denote the subdigraph induced by a vertex set $X$. Let $G-X = G[V(G)\backslash X]$.   For $v\in V(G)$, we denote the set of out-neighbors and in-neighbors of $v$ in $G$ by $N^+(v,G)$ and $N^-(v,G)$ respectively.  When $G$ is clear from the context, we abbreviate these to $N^+(v)$ and $N^-(v)$. Similarly, we write $d^+(v,G)=|N^+(v,G)|$ and $d^-(v,G)=|N^-(v,G)|$, and abbreviate them to $d^+(v)$ and $d^-(v)$ when no confusion can arise. For $S\subseteq V(G)$, we write $N^-(v,S)=N^-(v)\cap S$, $N^+(v,S)=N^+(v)\cap S$. Set $d^-(v,S)=|N^-(v,S)|$, $d^+(v,S)=|N^+(v, S)|$ and $d(v,S)=d^+(v,S)+d^-(v,S)$. 

Let each of $A$ and $B$ be a subset of $V(G)$ or a subdigraph of $G$. Let $N^+(A,B)$ denote the common out-neighbors of  vertices in $A$ that lie in $B$. $N^-(A,B)$ is defined analogously.   We use $E(A,B)$ to denote the set of edges from $A$ to $B$ and write $E(A)$ for $E(A,A)$. Set $e(A,B) = |E(A,B)|$ and $e(A) = |E(A)|$.  We write $N^{\pm}$ to indicate that a particular property holds for both in- and out-neighborhoods. We use $d^{\pm}$ similarly. For example, $d^{\pm}(v,V_{i\pm 1})\geqslant |V_{i\pm 1}|$ means that $d^{\ast}(v,V_{i\ast 1})\geqslant |V_{i\ast 1}|$ for each $\ast\in\{+,-\}$. 

Let $G$ be an oriented graph, and let $V_1,V_2,V_3$ be three pairwise disjoint subsets of $V(G)$. We say that a copy of $D_{a,b,c}$ in $G$ is of type-$(V_i,V_{i+1},V_{i+2})$ if the corresponding copies of $TT_a$, $TT_b$, and $TT_c$ lie in $G[V_i]$, $G[V_{i+1}]$ and $G[V_{i+2}]$, respectively. Moreover, we say that such a copy is of type-$(V_i)$ if all of its vertices lie in $V_i$. When no ambiguity arises, we abbreviate type-$(V_i,V_{i+1},V_{i+2})$ and type-$(V_i)$ to type-$(i,i+1,i+2)$ and type-$(i)$, respectively.

A cycle or path in digraphs always means a directed cycle or path.  It should be noted that we often identify $H$ with its vertex set $V(H)$, for example, we write $G-H$ and $v\in H$ instead of $G-V(H)$ and $v\in V(H)$.  We interchangeably think of a tiling $\mathcal{H}=\{H_1,H_2,\ldots,H_t\}$ as a set of disjoint  digraphs in $G$ and as a subgraph of $G$ with vertex set $V(\mathcal{H}) = \bigcup_{i} V(H_i)$ and edge set $E(\mathcal{H}) = \bigcup_{i} E(H_i)$. When an ordering is convenient, we also view $\mathcal{H}$ as the ordered $t$-tuple $\mathcal{H}=(H_1,H_2,\ldots,H_t)$.


\section{Preliminaries}\label{SEC-pre}

An important ingredient in this paper is Szemer\'{e}di's Regularity Lemma.  We first introduce some necessary notation. The \emph{density} of a bipartite graph $G[X,Y]$ with vertex classes $X,Y$ is defined to be $$\rho(X,Y)=\frac{e(X,Y)}{|X||Y|},$$ where $e(X,Y)$ is the number of edges between $X$ and $Y$ in the graph $G$. 
\begin{definition}
    Given $\varepsilon>0, d\in[0,1]$, we say that a bipartite graph $G[X,Y]$ is
	
	$\bullet$ \emph{$\varepsilon$-regular} if for all sets $A\subseteq X$ and $B\subseteq Y$ with $|A|\geqslant  \varepsilon|X|, |B|\geqslant  \varepsilon|Y|$ we have $$|\rho(A,B)-\rho(X,Y)|<\varepsilon;$$
	
	$\bullet$ \emph{$(\varepsilon,d)$-regular} if it is $\varepsilon$-regular with density $\rho(X,Y)\geqslant  d$;

	$\bullet$ \emph{$(\varepsilon,d)$-superregular} if it is $(\varepsilon,d)$-regular and additionally the degree $d_G(x,Y)\geqslant  d|Y|$ for every $x\in X$ and $d_G(y,X)\geqslant  d|X|$ for every $y\in Y$.
\end{definition}

 For a digraph $G$ and two disjoint sets $X,Y\subseteq V(G)$, let $G[X,Y]$ be a bipartite digraph where all edges are oriented from $X$ to $Y$. We define the pair $(X,Y)$ as (super)regular if it satisfies the (super)regularity condition in the underlying graph of $G$.

The next two lemmas are well-known and their proofs can be found  in \cite{taylor2013}.
	\begin{lemma}[\cite{taylor2013}]
		\label{LEM-fewvxswithsmalld}
		Let $(X,Y)$ be an $(\varepsilon,d)$-regular pair. Suppose that $A\subseteq X, B\subseteq Y$ with $|A|\geqslant  \varepsilon|X|, |B|\geqslant  \varepsilon|Y|$. 
        Then all but at most $\varepsilon|Y|$ vertices of $Y$ have at least $(d-\varepsilon)|A|$ neighbors in $A$, and all but at most $\varepsilon|X|$ vertices of $X$ have at least $(d-\varepsilon)|B|$ neighbors in $B$.
	\end{lemma}

\begin{lemma}[\cite{taylor2013}]
		\label{LEM-regulartosuper}
		Suppose  that $0<\varepsilon\leqslant  \alpha\leqslant  1/2$. Let $(X,Y)$ be an $(\varepsilon,d)$-regular pair. If $A\subseteq X, B\subseteq Y$ with $|A|\geqslant  \alpha|X|$ and $|B|\geqslant  \alpha|Y|$, then $(A,B)$ is $\varepsilon/\alpha$-regular and has density at least $d-\varepsilon$.
	\end{lemma}

The Diregularity Lemma is a version of the Regularity Lemma for digraphs due to Alon and
Shapira \cite{alon2003testing}.
Its proof is quite similar to the undirected version.
We will use the degree form of the Diregularity Lemma which can be easily derived (for example,  see \cite{young2007extremal}) from the standard version, in exactly the same manner as the undirected degree form.

\begin{lemma}[Diregularity Lemma, \cite{keevashJLMS79}]
		\label{LEM:reguler}
		For every $\varepsilon \in (0,1)$ and all numbers $M^{\prime}$, $M^{\prime\prime}$ there are numbers $M$ and $n_0$ such that if $G$ is a digraph on $n \geqslant  n_0$ vertices, $U_1,\ldots,U_{M^{\prime\prime}}$ is a partition of $V(G)$, and $d \in [0,1]$ is any real number, then there is a partition $V_0,V_1,\ldots,V_k$ 
		of $V(G)$ and a spanning subdigraph $G^\prime$ of $G$ such that the following statements hold:
		
		$\bullet$ $M^{\prime}\leqslant  k \leqslant  M$,
		
		$\bullet$ $|V_0| \leqslant  \varepsilon n$ and $|V_i|=m$ for each $i\in[k]$,
		
		$\bullet$ for each $V_i$ with $i\in[k]$ there exists some $U_j$ containing $V_i$,
		
		$\bullet$  for each $x \in V(G)$, $d^\pm(x,G^{\prime}) > d^\pm(x,G) - (d+\varepsilon)n$,
		
		$\bullet$ for all $i\in[k]$ the digraph $G^{\prime}[V_i]$ is empty,
		
		$\bullet$ for all $i\neq j\in[k]$ the bipartite graph $G^{\prime}[V_i,V_j]$ whose vertex classes are $V_i$ and $V_j$ and whose edges are all from $V_i$ to $V_j$ in $G^{\prime}$ is $\varepsilon$-regular and has density either $0$ or density at least $d$.
	\end{lemma}

The vertex sets $V_1,\ldots,V_k$ are called \emph{clusters}. 
The last condition of the lemma says that all pairs of clusters are $\varepsilon$-regular in both directions (but possibly with different densities).
Given clusters $V_1,\ldots,V_k$ and the  digraph $G^{\prime}$, the \emph{reduced digraph} $R^{\prime}$ with parameters $(\varepsilon, d)$ is the digraph
whose vertex set is $[k]$ and in which $ij\in E(R^{\prime})$ if and only if $(V_i, V_j)_{G^{\prime}}$ is $\varepsilon$-regular with density at least $d$.  Note that  $R^{\prime}$ is not necessarily oriented even if the original graph $G$ is. 
The following result of Kelly, K\"{u}hn, and Osthus shows that by discarding edges with appropriate probabilities one can obtain a spanning \emph{oriented} subgraph $R\subseteq R^{\prime}$ which still inherits the minimum (semi-)degree of $G$.

\begin{lemma}[\cite{keevashJLMS79, kellyCPC17}]\label{LEM-reducedori}
For every $\varepsilon \in (0,1)$, there exist integers $M^{\prime}$ and $n_0$
such that the following holds. Let $0\leqslant d\leqslant \gamma \leqslant 1$ and let $G$ be an   oriented graph on $n\geqslant n_0$ vertices and let $R^{\prime}$ be the reduced digraph 
obtained by applying the Diregularity Lemma to $G$ with parameters $\varepsilon$, $d$ and $M^{\prime}$.
Then $R^{\prime}$ has a spanning oriented subgraph $R$ satisfying 
\begin{itemize}
\item[\rm{(i)}]$\delta^0(R)\geqslant  (\delta^0(G)/|G|-(d+3\varepsilon))|R|$,
\item[\rm{(ii)}] $\delta(R)\geqslant (\delta(G)/|G|-(2d+3\varepsilon))|R|$.
\item [\rm{(iii)}] for all disjoint sets $S,T\subset V(R)$ with $e_G(S^\ast,T^\ast)\geqslant 3\gamma n^2$ we have $e_R(S,T)>\gamma|R|^2$, where $S^\ast= \bigcup_{i\in S}V_i$ and $T^\ast = \bigcup _{i\in T} V_i$.
\end{itemize}
\end{lemma}

 In fact, in \cite{keevashJLMS79}, condition~(iii) is stated with parameter $d$, but the same argument ensures that condition (iii) remains valid for $\gamma$. The oriented graph $R$ given by Lemma \ref{LEM-reducedori} is called the \emph{reduced oriented graph} with parameters $\varepsilon$ and $d$.

The rest of this section is devoted to introducing the ``dependent random choice'' which will be frequently used in the following embedding process.
In fact, we will apply a 
variant of the dependent random choice method, as mentioned in \cite{debiasioCPC35}.
    
\begin{lemma}[Dependent random choice variant, \cite{debiasioCPC35}]\label{LEM:dependentrandom} 
    Let $k\in \mathbb{N}$, $0<d\leqslant 1$, and define $r=d^k/\sqrt[k]{2}$.  Let $G[A,B]$ be a bipartite graph with $e(G) \geqslant d|A||B|$. For all $0 < \varepsilon < 1$, there exists $U \subseteq A$ with $|U| \geqslant r|A|$ such that all but at most $(\varepsilon|U|)^k$ of the $k$-tuples in $U$ have at least $\varepsilon r |B|$ common neighbors in $B$.
 \end{lemma}

\section{Proofs of Theorems \ref{THM-abc}, \ref{THM-1bc}, and \ref{THM:factor-extremal}}\label{SEC-main results}

In this section, we first state all the necessary lemmas and then use them to complete the proofs of our main theorems. The proofs of these lemmas will be presented in Sections \ref{SEC-5.1}-\ref{SEC-5.5}.

Let $G$ and $H$ be two (di)graphs with $|G|=n$ and $|H|=h$. We say that a set $S\subseteq V(G)$ is an \emph{$(H,x,y)$-linking $(ht-1)$-set} if $|S|=ht-1$ and both $G[S\cup\{x\}]$ and $G[S\cup\{y\}]$ contain an $H$-factor. Two vertices $x,y\in V(G)$ are \emph{$(H,\beta,t)$-reachable} if there exist at least $\beta n^{ht-1}$ $(H,x,y)$-linking $(ht-1)$-sets in $G$. A set $U\subseteq V(G)$ is \emph{$(H,\beta,t)$-closed} if every pair of distinct vertices in $U$ is $(H,\beta,t)$-reachable. A partition $\mathcal{P}=(U_1,U_2,\ldots,U_s)$ of $V(G)$ is $(H,\beta,t)$-closed if each $U_i$ is $(H,\beta,t)$-closed.

\begin{lemma}\label{LEM:closed-factor}
Given $a,b,c\in \mathbb{N}$ with $1\leqslant a\leqslant b\leqslant c$, let
$0<1/n \ll \eta \ll \beta \ll 1$, where $n\in (a+b+c)\mathbb{N}$.
Suppose that $G$ is an $n$-vertex oriented graph with
$\delta^0(G)\geqslant (1/2-\eta)n$.
If $V(G)$ is $(D_{a,b,c},\beta,t)$-closed for some $t\in\mathbb{N}$, then
$G$ contains a $D_{a,b,c}$-factor.
\end{lemma}

 \begin{definition}\label{DEF:extremal}
Let $G$ be an $n$-vertex oriented graph. A partition $(V_1,V_2,V_3)$ of $V(G)$ is called \emph{$\gamma$-extremal} for some $\gamma>0$, if the following condition holds for each $i\in[3]$, where the indices are taken modulo $3$.
\begin{enumerate}[label=\upshape \textbf{(EP\arabic{enumi})},ref=\upshape (EP\arabic{enumi}),series=mylist]
    \setlength{\itemindent}{1.5em}
    \item $|V_i|=n/3\pm O(\gamma n)$ and $e(V_{i+1},V_i)=O(\gamma n^2)$.\label{EP1}
\end{enumerate}
Moreover, the partition is called \emph{$\gamma$-superextremal} if it is $\gamma$-extremal and also satisfies
\begin{enumerate}[label=\upshape \textbf{(EP\arabic{enumi})},ref=\upshape (EP\arabic{enumi}),resume=mylist]
    \setlength{\itemindent}{1.5em}
    \item $d^+(v,V_{i+1}),\,d^-(v,V_{i-1})\geqslant n/6-O(\gamma n)$ for each $v\in V_i$.\label{EP2}
\end{enumerate}
\end{definition}

For convenience, when no ambiguity arises, we say that $G$ is \emph{$\gamma$-(super)extremal}.

\begin{lemma}\label{LEM:nonext}
Given $a,b,c \in \mathbb{N}$ with $1\leqslant a\leqslant b\leqslant c$, let
$0<1/n\ll \eta \ll \beta \ll \gamma \ll 1$.
Suppose that $G$ is an $n$-vertex oriented graph with
$\delta^0(G)\geqslant (1/2-\eta)n$.
If $V(G)$ is not $(D_{a,b,c},\beta,t)$-closed for any $t\in \mathbb{N}$, then
$G$ is $\gamma$-extremal and  neither $a+1=b=c$ nor $a+1=b+1=c$ holds.
\end{lemma}

The above lemma implies that if $a+1=b=c$ or $a+1=b+1=c$, then
$V(G)$ is $(D_{a,b,c},\beta,t)$-closed for some $t\in\mathbb{N}$, and hence
$G$ contains a $D_{a,b,c}$-factor by Lemma~\ref{LEM:closed-factor}.

Our next lemmas explore global structural properties of $\gamma$-extremal oriented graphs.

\begin{lemma}\label{LEM:abcbalance}
Given $a,b,c\in \mathbb{N}$ with  $1\leqslant a\leqslant b\leqslant c$, let
$0<1/n \ll \eta \ll \gamma\ll 1$.
Suppose that $G$ is an $n$-vertex oriented graph with
$\delta^0(G)\geqslant(1/2-\eta)n$.
If $G$ is $\gamma$-extremal, then there exist a set $S\subseteq V(G)$ with
$|S|\leqslant 2(a+b+c)-3$ and a constant-sized $D_{a,b,c}$-tiling $\mathcal{H}$ in $G-S$
such that $V(G)\backslash (S\cup V(\mathcal{H}))$ admits a $\gamma^{1/3}$-superextremal partition
$(V_1,V_2,V_3)$ with $|V_i|\equiv 0\pmod{a+b+c}$ for all $i\in[3]$.

In particular, if $\gcd(a+b+c,c^2-ab)=1$, then $|S|\leqslant a+b+c-1$.
\end{lemma}

\begin{lemma}\label{LEM:1bcbalance}
Given $b,c\in \mathbb{N}$ with  $1\leqslant b\leqslant c$ and $3\mid(1+b+c)$, let
$0<1/n\ll \gamma\ll 1$, where $n\in(1+b+c)\mathbb{N}$.
Suppose that $G$ is an $n$-vertex semi-regular tournament.
If $G$ is $\gamma$-extremal, then $G$ contains a constant-sized $D_{1,b,c}$-tiling $\mathcal{H}$ such that
$V(G)\setminus V(\mathcal{H})$ admits a $\gamma^{1/3}$-superextremal partition $(V_1,V_2,V_3)$ with
$|V_i|\equiv 0\pmod{1+b+c}$ for each $i\in[3]$.
\end{lemma}

\begin{lemma}\label{LEM:balancefactor}
Given $a,b,c\in \mathbb{N}$ with  $1\leqslant a\leqslant b\leqslant c$, let
$0<1/n\ll \eta \ll \gamma\ll 1$.
Suppose that $G$ is an $n$-vertex oriented graph with
$\delta^0(G)\geqslant(1/2-\eta)n$.
If $V(G)$ admits a $\gamma$-superextremal partition $(V_1,V_2,V_3)$ with
$|V_i|\equiv 0\pmod{a+b+c}$ for each $i\in[3]$, then $G$ contains a $D_{a,b,c}$-factor.
\end{lemma}

Equipped with the above lemmas, we are now ready to prove Theorems \ref{THM-abc}, \ref{THM-1bc}, and \ref{THM:factor-extremal}.

\begin{proof}[\textbf{Proof of Theorem \ref{THM-abc}}]
Let $a,b,c\in \mathbb{N}$ with $2\leqslant a \leqslant b \leqslant c$ and $\gcd(a+b+c,c^2-ab)=1$. Set $h=a+b+c$. Choose
$$
0< 1/n\ll\eta\ll \beta\ll\gamma\ll 1
\text{ and }
n\in h\mathbb{N}.
$$
Let $G$ be an $n$-vertex oriented graph with $\delta^0(G)\geqslant (1/2-\eta)n$.

We claim that $G$ contains a $D_{a,b,c}$-factor. By Lemma \ref{LEM:closed-factor}, it suffices to consider the case where $V(G)$ is not $(D_{a,b,c},\beta,t)$-closed for any $t\in\mathbb{N}$. Then Lemma \ref{LEM:nonext} implies that $G$ is $\gamma$-extremal. By Lemma \ref{LEM:abcbalance}, there exist a set $S\subseteq V(G)$ with $|S|\leqslant h-1$ and a constant-sized $D_{a,b,c}$-tiling $\mathcal{H}$ in $G-S$ such that $V(G)\setminus (S\cup V(\mathcal{H}))$ admits a $\gamma^{1/3}$-superextremal partition $(V_1,V_2,V_3)$ with $|V_i|\equiv 0\pmod h$ for all $i\in[3]$. Since $n\in h\mathbb{N}$, it follows that $|S|\equiv 0\pmod h$, and hence $|S|=0$. Let $G^\prime=G-V(\mathcal{H})$. Since $\mathcal{H}$ is a constant-sized tiling, we have
$$
\delta^0(G^\prime)\geqslant (1/2-\eta)n-|V(\mathcal{H})|
\geqslant (1/2-2\eta)|G^\prime|.
$$
Applying Lemma \ref{LEM:balancefactor} to $G^\prime$ with parameters $2\eta$ and $\gamma^{1/3}$, we conclude that $G^\prime$ contains a $D_{a,b,c}$-factor $\mathcal{H}^\prime$. Then $\mathcal{H}\cup\mathcal{H}^\prime$ is a $D_{a,b,c}$-factor of $G$, as desired.
\end{proof}

The proof of Theorem \ref{THM-1bc} is essentially the same as that of Theorem \ref{THM-abc}. For completeness, we include it here. Note that, in Theorem \ref{THM-1bc}, the oriented graph $G$ is assumed to be a semi-regular tournament.

\begin{proof}[\textbf{Proof of Theorem \ref{THM-1bc}}]
Let $b,c\in \mathbb{N}$ with $1\leqslant b\leqslant c$, and set $h=1+b+c$. Choose
$$
0<1/n\ll \eta\ll \beta\ll\gamma\ll 1
\text{ and } 
n\in h\mathbb{N}.
$$
Let $G$ be an $n$-vertex semi-regular tournament. Assume that either $\gcd(1+b+c,c^2-b)=1$ or $3\mid (1+b+c)$. If the former holds, then $G$ contains a $D_{1,b,c}$-factor by the same argument as in the proof of Theorem \ref{THM-abc}. Thus we may assume that $3\mid (1+b+c)$. Moreover, by Lemma \ref{LEM:closed-factor}, it suffices to consider the case where $V(G)$ is not $(D_{1,b,c},\beta,t)$-closed for any $t\in \mathbb{N}$.

By Lemma \ref{LEM:nonext}, $G$ is $\gamma$-extremal. Then Lemma \ref{LEM:1bcbalance} implies that there exists a constant-sized $D_{1,b,c}$-tiling $\mathcal{H}$ such that $V(G)\setminus V(\mathcal{H})$ admits a $\gamma^{1/3}$-superextremal partition $(V_1,V_2,V_3)$ with $|V_i|\equiv 0\pmod{1+b+c}$ for each $i\in[3]$. Let $G^\prime=G-V(\mathcal{H})$. Since $\mathcal{H}$ is a constant-sized tiling, we have
$$
\delta^0(G^\prime)\geqslant \lfloor (n-1)/2\rfloor-|V(\mathcal{H})|
\geqslant (1/2-\eta)|G^\prime|.
$$
Applying Lemma \ref{LEM:balancefactor} to $G^\prime$ with parameters $\eta$ and $\gamma^{1/3}$, we conclude that $G^\prime$ contains a $D_{1,b,c}$-factor $\mathcal{H}^\prime$. Then $\mathcal{H}\cup\mathcal{H}^\prime$ is a $D_{1,b,c}$-factor of $G$, as desired.
\end{proof}

The rest of this section is devoted to the proof of Theorem \ref{THM:factor-extremal}.

\begin{proof}[\textbf{Proof of Theorem \ref{THM:factor-extremal}}]
Let $a,b,c\in \mathbb{N}$ with $1\leqslant a\leqslant b\leqslant c$. Set $h=a+b+c$. Choose
$$0< 1/n\ll\eta\ll \beta \ll \gamma \ll 1.$$ 
Suppose that $G$ is an $n$-vertex oriented graph with $\delta^0(G)\geqslant(1/2-\eta)n$. Choose $U\subseteq V(G)$ with $|U|\leqslant h-1$ such that $|V(G)\backslash U|\equiv 0\pmod h$. Set $G^{\ast}=G-U$.

If $V(G^{\ast})$ is $(D_{a,b,c},\beta,t)$-closed in $G^{\ast}$ for some $t\in\mathbb{N}$, then, by Lemma \ref{LEM:closed-factor}, $G^{\ast}$ has a $D_{a,b,c}$-factor, and hence $G$ has the desired tiling since $|U|\leqslant h-1$.  
Thus we may assume that $V(G^{\ast})$ is not $(D_{a,b,c},\beta,t)$-closed in $G^{\ast}$ for any $t\in \mathbb{N}$. Then $G^{\ast}$ is $\gamma$-extremal by Lemma \ref{LEM:nonext}. Since $G^{\ast}=G-U$ and $U$ is a constant-sized set, $G$ is clearly $\gamma$-extremal as well. By Lemma \ref{LEM:abcbalance}, there exist a set $S\subseteq V(G)$ with $|S|\leqslant 2h-3$ (in particular, $|S|\leqslant h-1$ when $\gcd(a+b+c,c^2-ab)=1$) and a constant-sized $D_{a,b,c}$-tiling $\mathcal{H}$ in $G-S$ such that $V(G)\backslash (S\cup V(\mathcal{H}))$ admits a $\gamma^{1/3}$-superextremal partition $(V_1,V_2,V_3)$ with $|V_i|\equiv 0\pmod h$ for all $i\in[3]$. Let $G^{\prime}=G-(S\cup V(\mathcal{H}))$.
Clearly,
$\delta^{0}(G^{\prime})\geqslant (1/2-\eta)n-|S\cup V(\mathcal{H})|\geqslant (1/2-2\eta)|G^\prime|$.
Applying Lemma \ref{LEM:balancefactor} to $G^\prime$ with parameters $2\eta$ and $\gamma^{1/3}$, we obtain a $D_{a,b,c}$-factor $\mathcal{H}^{\prime}$ in $G^{\prime}$, which together with $\mathcal{H}$ forms a $D_{a,b,c}$-factor of $G-S$. This proves the first statement. 

Suppose now that $G$ has no $D_{a,b,c}$-factor and that $n\in h\mathbb{N}$. It follows from Lemma \ref{LEM:closed-factor} that $V(G)$ is not $(D_{a,b,c},\beta,t)$-closed, and then Lemma \ref{LEM:nonext} implies that $G$ is $\gamma$-extremal. The ``moreover'' part follows immediately from Definition \ref{DEF:extremal} \ref{EP1}. 
\end{proof}

We next show that the constants $2(a+b+c)-3$ and $a+b+c-1$ in Theorem \ref{THM:factor-extremal} cannot be improved. From the above proof, it is easy to see that it suffices to show that the constants in Lemma \ref{LEM:abcbalance} are best possible for some choices of $a,b,c$.

\begin{remark}\label{RE:2h-3}

Let $G$ be a blow-up of directed triangle $C_3$ obtained by replacing its three vertices with semi-regular tournaments on $3st+s-1$, $3st+2s-1$, and $3st+3s-1$ vertices, respectively. Denote the three vertex sets by $V_1$, $V_2$, and $V_3$, where $|V_i|=3st+is-1$.

Clearly, $|G|=9st+6s-3$ and $|V_i|=|G|/3\pm s$. Thus, $(V_1,V_2,V_3)$ is a $\gamma$-superextremal partition of $V(G)$ for any sufficiently small $\gamma>0$. Moreover, any copy of $D_{s,s,s}$ in $G$ can only lie entirely inside some $G[V_j]$ or be of type-$(j,j+1,j+2)$ for some $j\in[3]$. Therefore, one must remove at least $6s-3$ vertices from $V(G)$ in order to make the sizes of the resulting parts congruent to $0$ modulo $3s$. This shows that the constant ``$2h-3$'' in Lemma \ref{LEM:abcbalance} is best possible when $a=b=c=s$. Moreover, the constant ``$h-1$'' in Lemma \ref{LEM:abcbalance} is best possible by considering an $n$-vertex oriented graph  where $n\equiv h-1 \pmod {h}$.
\end{remark}

\section{Proof of Lemma \ref{LEM:closed-factor}}\label{SEC-5.1}

In this section, we prove Lemma \ref{LEM:closed-factor}. One of the key ingredients is the following result, which yields an almost $D_{a,b,c}$-factor in the oriented graph $G$. Its proof will be given in the next subsection.

\begin{lemma}[Almost Covering Lemma]\label{LEM:almostDabc}
Given $a,b,c\in\mathbb{N}$ with $1\leqslant a\leqslant b\leqslant c$, there exists a constant $0<\eta_1\ll 1$ such that the following holds. Let $0<1/n\ll \eta_2\ll \delta\leqslant 1$. If $G$ is an $n$-vertex oriented graph with $\delta^0(G)\geqslant(1/2-\eta_1)n$ and $\delta(G)\geqslant(1-\eta_2)n$, then $G$ contains a $D_{a,b,c}$-tiling covering all but at most $\delta n$ vertices.
\end{lemma}

The following result of Nenadov and Pehova \cite{nenadovSJDM34} also plays an important role in the proof of Lemma \ref{LEM:closed-factor}, as it provides a sufficient condition for the existence of an absorbing set. Before stating their result, we introduce some necessary definitions. Let $G$ and $H$ be two digraphs. We say that a set $S \subseteq V(G)$ is an \emph{$(H,\xi)$-absorbing set} for some $\xi$ if, for every subset $R \subseteq V(G) \backslash S$ such that $|H|$ divides $|S| + |R|$ and $|R| \leqslant \xi n$, the induced digraph $G[S \cup R]$ contains an $H$-factor. For a subset $S\subseteq V(G)$ with $|H|\mid|S|$ and an integer $t>0$, a subset $A_S\subseteq V(G)$ is an \emph{$(H,t)$-absorber of $S$} if $|A_S|=t|H|$ and both $G[A_S]$ and $G[A_S\cup S]$ contain an $H$-factor.

The original statements of the following two lemmas in \cite{hanRSA64,nenadovSJDM34} are formulated for graphs. However, their proofs also work for digraphs.

\begin{lemma}[Absorbing Lemma, \cite{nenadovSJDM34}]\label{LEM:Absorb}
Let $t,h \in \mathbb{N}$ with $h \geqslant 3$, and let $0<1/n\ll \eta \ll  \xi\ll \beta \ll 1$. Suppose that $H$ and $G$ are two digraphs on $h$ and $n$ vertices, respectively. If every set $S\in \binom{V(G)}{h}$ has at least $\beta n$ disjoint $(H,t)$-absorbers, then $G$ contains an $(H,\xi)$-absorbing set of size at most $\beta n$.
\end{lemma}

Lemma \ref{LEM:Absorb} indicates that the key step in constructing an absorbing set is to find $\Omega(n)$ disjoint absorbers for every $S\in \binom{V(G)}{h}$. This can be achieved by the following result.

\begin{lemma}[\cite{hanRSA64}]\label{LEM-findvtxab}
Given $t,h\in \mathbb{N}$ with $h\geqslant 3$, let $0<1/n \ll \eta \ll \beta \ll 1$. Suppose that $H$ and $G$ are two digraphs on $h$ and $n$ vertices, respectively. If $V(G)$ is $(H,\beta,t)$-closed, then every $S\in \tbinom{V(G)}{h}$ has at least $\beta^2 n$ disjoint $(H,ht)$-absorbers.
\end{lemma}

Equipped with Lemmas \ref{LEM:almostDabc}, \ref{LEM:Absorb}, and \ref{LEM-findvtxab}, we are ready to prove Lemma \ref{LEM:closed-factor}.

\begin{proof}[\textbf{Proof of Lemma \ref{LEM:closed-factor}}]
Given $a,b,c,t\in \mathbb{N}$ with $1\leqslant a\leqslant b \leqslant c$, and set $h=a+b+c$. Choose
$$0<1/n\ll\eta\ll \delta\ll\beta\ll\eta_1\ll1,$$
where $3\eta$ plays the role of $\eta_2$ in Lemma \ref{LEM:almostDabc} with parameters $\eta_1$ and $\delta$.

Suppose that $G$ is an $n$-vertex oriented graph with $\delta^{0}(G)\geqslant (1/2-\eta)n$. Then every vertex of $G$ has at most $2\eta n$ missing edges. Since $V(G)$ is $(D_{a,b,c}, \beta, t)$-closed, Lemma \ref{LEM-findvtxab} implies that every $S\in \binom{V(G)}{h}$ has at least $\beta^2 n$ disjoint $(D_{a,b,c}, ht)$-absorbers in $G$. Applying Lemma \ref{LEM:Absorb} with parameters $\beta^2$, $ht$, and $h$, we obtain a $(D_{a,b,c}, \delta)$-absorbing set $A \subseteq V(G)$ with $|A| \leqslant \beta^2 n$.

Let $G^{\prime}=G-A$. Then
$\delta^0(G^{\prime}) \geqslant(1/2-\eta)n-|A|\geqslant(1/2-\eta_1)|G^\prime|$
and
$\delta(G^{\prime})\geqslant |G^{\prime}|-2\eta n\geqslant (1-3\eta)|G^{\prime}|$,
since $\eta\ll\beta\ll\eta_1\ll 1$.
Applying Lemma \ref{LEM:almostDabc} to $G^{\prime}$ with parameters $\eta_1$ and $\eta_2=3\eta$, we obtain a $D_{a,b,c}$-tiling $\mathcal{H}$ in $G^{\prime}$ covering all but at most $\delta n$ vertices.
Let $R=V(G^{\prime}) \backslash V(\mathcal{H})$ be the set of uncovered vertices. Clearly, $|R|\leqslant \delta n$.
Since $A$ is a $(D_{a,b,c}, \delta)$-absorbing set and $h \mid (|A|+|R|)$, $G[A \cup R]$ contains a $D_{a,b,c}$-factor $\mathcal{H}^\prime$, which together with $\mathcal{H}$ forms a $D_{a,b,c}$-factor of $G$.
\end{proof}

\subsection{Proof of the Almost Covering Lemma}

In this subsection, we present several auxiliary lemmas and explain how they combine to yield the proof of Lemma \ref{LEM:almostDabc}. The proofs of these lemmas are deferred to the next subsection.

\begin{theorem}[\cite{keevashJCTB99}]\label{THM:trianglefactor}
There exists a constant $\eta>0$ such that, for sufficiently large $n$, every $n$-vertex oriented graph $G$ with $\delta^0(G)\geqslant (1/2-\eta)n$ contains a $C_3$-tiling covering all but at most three vertices.
\end{theorem}

Another ingredient in the proof of Lemma \ref{LEM:almostDabc} is the following embedding lemma.

\begin{lemma}\label{LEM:findmanyDabc}
Given $a,b,c\in \mathbb{N}$ with $a\leqslant b \leqslant c$, let
$0< 1/n\ll \eta \ll \beta \ll \xi \ll \varepsilon \ll d<1$.
Suppose that $G$ is an $n$-vertex oriented graph with $\delta(G)\geqslant (1-\eta )n$.
Let $A,B,C$ be three disjoint subsets of $V(G)$ such that $|A|,|B|,|C| \geqslant \xi n$ and $(A,B)$, $(B,C)$, and $(C,A)$ are $(\varepsilon,d)$-regular.
Then $G[A\cup B\cup C]$ contains at least $\beta n^{a+b+c}$ copies of $D_{a,b,c}$ of type-$(A,B,C)$.
\end{lemma}

We now prove Lemma \ref{LEM:almostDabc} using Theorem \ref{THM:trianglefactor} and Lemma \ref{LEM:findmanyDabc}.

\begin{proof}[\textbf{Proof of Lemma \ref{LEM:almostDabc}}]
Given $a,b,c\in \mathbb{N}$ with $1\leqslant a\leqslant b \leqslant c$, let
$$0<1/n \ll \eta_2 \ll \xi\ll 1/M \ll \delta \ll \eta_1=\eta/3 \ll 1,$$
where $\eta$ is the constant from Theorem \ref{THM:trianglefactor}, and $M$ is the constant in Lemma \ref{LEM:reguler} with $\varepsilon=\delta^3$, $d=\eta_1$, $M^\prime =1/\delta^3$, and $M^{\prime\prime}=1$.

Suppose that $G$ is an $n$-vertex oriented graph with $\delta^0(G)\geqslant (1/2-\eta_1)n$ and $\delta(G)\geqslant(1-\eta_2)n$. Applying Lemma \ref{LEM:reguler} to $G$ with parameters $\delta^3,\eta_1, M^\prime =1/\delta^3,M^{\prime\prime}=1$, we obtain a partition $V_0,V_1,\ldots,V_k$ of $V(G)$ satisfying  $1/\delta^3\leqslant k\leqslant M$, $|V_0|\leqslant \delta ^3 n$ and $|V_i|=m$ for each $i\in[k]$, where $n/(2k) \leqslant m\leqslant n/k$.

Let $R$ be the reduced oriented graph. Then $|R|=k$ and
$\delta^0(R)\geqslant (1/2-2\eta_1-3\delta^3)k\geqslant (1/2-\eta)k$
by Lemma \ref{LEM-reducedori}. Applying Theorem \ref{THM:trianglefactor} to $R$, we obtain a $C_3$-tiling $\mathcal{T}$ in $R$ covering all but at most three vertices. Let $123$ be any fixed triangle in $R$. We next claim that $G[V_1\cup V_2\cup V_3]$ contains a $D_{a,b,c}$-tiling covering all but at most $3\delta^2m$ vertices. 

To see this, it suffices to show that whenever $V_i^{\ast}\subseteq V_i$ satisfies $|V_i^{\ast}|\geqslant \delta^2m$ for  $i\in[3]$, the graph $G[V_1^{\ast}\cup V_2^{\ast}\cup V_3^{\ast}]$ contains a $D_{a,b,c}$-tiling $\mathcal{F}$ such that $|V(\mathcal{F})\cap V_i^\ast|=a+b+c$ for each $i$. Let $(V_i^A,V_i^B,V_i^C)$ be an almost balanced partition of $V_i^{\ast}$, that is, the sizes of $V_i^A$, $V_i^B$, and $V_i^C$ differ by at most one. Clearly,
$|V_i^A|, |V_i^B|, |V_i^C|\geqslant |V_i^{\ast}|/4\geqslant \delta^2 m/4$.
Since $123$ is a triangle in $R$, each of $(V_1,V_2)$, $(V_2,V_3)$, and $(V_3,V_1)$ is $(\delta^3,\eta_1)$-regular. Then Lemma~\ref{LEM-regulartosuper} implies that each of $(V_i^A,V_{i+1}^B)$, $(V_{i+1}^B,V_{i+2}^C)$, and $(V_{i+2}^C,V_i^A)$ is $(4\delta,\eta_1/2)$-regular for each $i\in[3]$. By Lemma \ref{LEM:findmanyDabc}, for each $i\in[3]$ there exists a copy $H_i$ of $D_{a,b,c}$ of type-$(V_i^A, V_{i+1}^B,V_{i+2}^C)$ in $G[V_i^A\cup V_{i+1}^B\cup V_{i+2}^C]$. Then $\{H_1,H_2,H_3\}$ is the desired $D_{a,b,c}$-tiling of $G[V_1^{\ast}\cup V_2^{\ast}\cup V_3^{\ast}]$. Therefore, for every triangle $123$ of $R$, the graph $G[V_1\cup V_2\cup V_3]$ contains a $D_{a,b,c}$-tiling covering all but at most $3\delta^2m$ vertices.

Recall that $\mathcal{T}$ is a $C_3$-tiling of $R$ covering all but at most three vertices of $R$. Hence, by combining the tilings corresponding to the triangles in $\mathcal{T}$, we obtain the desired tiling of $G$, covering all but at most
$|V_0|+k\delta^2m + 3m \leqslant \delta^3n + k\delta^2 n/k + 3n/k\leqslant \delta n$
vertices, since $|V_0|\leqslant \delta^3 n$ and $m\leqslant n/k \leqslant \delta^3 n$. This completes the proof.
\end{proof}

\subsection{Proof of Lemma \ref{LEM:findmanyDabc}}

Before proving Lemma \ref{LEM:findmanyDabc}, we first introduce a lemma showing that every oriented graph $G$ with $\delta(G)\geqslant(1-\eta)n$ contains $\Omega(n^s)$ copies of $TT_s$.

\begin{lemma} \label{LEM:findmanyTTk}
Given $s \in \mathbb{N}$, let $0< 1/n\ll \eta\ll  \mu \ll 1$.
If $G$ is an $n$-vertex oriented graph with $\delta(G) \geqslant (1-\eta) n$, then $G$ contains at least $\mu n^s$ copies of $TT_s$.
\end{lemma}

The proof of Lemma \ref{LEM:findmanyTTk} will be given at the end of this subsection.
We now complete the proof of Lemma \ref{LEM:findmanyDabc}.

\begin{proof}[\textbf{Proof of Lemma \ref{LEM:findmanyDabc}}]
Given $a,b,c\in \mathbb{N}$ with $a\leqslant b \leqslant c$, let
$$0<1/n\ll \eta \ll \beta\ll \xi \ll \varepsilon\ll \varepsilon_0\ll d,\mu <1,$$
where $\sqrt \eta$ is the minimum of the constants given by Lemma \ref{LEM:findmanyTTk} for $s=a,b,$ and $c$ with $\mu$. Set $r_x=(d-\varepsilon)^x/\sqrt[x]{2}$ for each $x\in\{a,b,c\}$. If $b=0$, then there is nothing to prove by applying Lemma \ref{LEM:findmanyTTk} to $G[C]$. Thus we may  assume that $1\leqslant b\leqslant c$.
 
Since $(C,A)$ is $(\varepsilon,d)$-regular, we have $e(C,A) \geqslant d|A||C|$. Applying Lemma \ref{LEM:dependentrandom} to the underlying graph of $G[C,A]$ with parameters $d$ and $\varepsilon=\varepsilon_0$, there exists $C^{\prime}\subseteq C$ with $|C^{\prime}| \geqslant r_c |C|$ such that all but at most $(\varepsilon_0 |C^{\prime}|)^c$ of the $c$-tuples in $C^{\prime}$ have at least $\varepsilon_0 r_c |A|$ common out-neighbors in $A$. By Lemma \ref{LEM-regulartosuper} and the fact that $(B,C)$ is $(\varepsilon,d)$-regular, we have $e(B,C^{\prime})\geqslant (d-\varepsilon)|B||C^{\prime}|$. Applying Lemma \ref{LEM:dependentrandom} in the same way to the underlying graph of $G[B,C^\prime]$, there exists $C^{\prime\prime}\subseteq C^{\prime}$ with $|C^{\prime\prime}| \geqslant r_c |C^{\prime}|$ such that all but at most $(\varepsilon_0 |C^{\prime\prime}|)^c$ of the $c$-tuples in $C^{\prime\prime}$ have at least $\varepsilon_0 r_c |B|$ common in-neighbors in $B$.

Since $\delta(G)\geqslant (1-\eta)n$, every vertex has at most $\eta n$ non-neighbors in $G$. As $|C|\geqslant \xi n$ and $|C^{\prime\prime}|\geqslant r_c^2|C|$, we have $\delta(G[C^{\prime\prime}]) \geqslant |C^{\prime\prime}|-\eta n \geqslant |C^{\prime\prime}|-\eta|C|/\xi \geqslant (1-\eta/(\xi r_c^2))|C^{\prime\prime}| \geqslant (1-\sqrt \eta)|C^{\prime\prime}|.$ 
Then Lemma \ref{LEM:findmanyTTk} shows that $G[C^{\prime\prime}]$ contains at least $\mu|C^{\prime\prime}|^c$ copies of $TT_c$. Moreover, there are at least
$$\mu |C^{\prime\prime}|^c-\varepsilon_0^c |C^{\prime}|^c-\varepsilon_0^c |C^{\prime\prime}|^c  \geqslant \mu |C^{\prime\prime}|^c/2$$
``nice'' copies of $TT_c$ in $G[C^{\prime\prime}]$ with the following property: all vertices of each such $TT_c$ have many common out-neighbors in $A$ and many common in-neighbors in $B$.
    
For any fixed ``nice'' copy of $TT_c$ in $G[C^{\prime\prime}]$, let $A^\prime\subseteq A$ be the set of its common out-neighbors and let $B^\prime \subseteq B$ be the set of its common in-neighbors. Since $(A,B)$ is $(\varepsilon,d)$-regular and
$|A^\prime| \geqslant \varepsilon_0 r_c |A|\geqslant \varepsilon |A|$ and
$|B^\prime| \geqslant \varepsilon_0 r_c |B|\geqslant \varepsilon |B|$,
the definition of regular pair implies that
$e(A^\prime,B^\prime) \geqslant (d-\varepsilon) |A^\prime||B^\prime|$.
Applying Lemma \ref{LEM:dependentrandom} to the underlying graph of $G[A^\prime,B^\prime]$ with $\varepsilon =\varepsilon_0$, there exists $B^{\prime\prime}\subseteq B^\prime$ with $|B^{\prime\prime}| \geqslant r_b |B^\prime|$ such that all but at most $(\varepsilon_0 |B^{\prime\prime}|)^b$ of the $b$-tuples in $B^{\prime\prime}$ have at least $\varepsilon_0 r_b|A^\prime|$ common in-neighbors in $A^\prime$. Since $\delta(G)\geqslant (1-\eta)n$ and $|B|\geqslant\xi n$, we have
$$\delta(G[B^{\prime\prime}])\geqslant  |B^{\prime\prime}|- \eta n\geqslant |B^{\prime\prime}|- \eta|B|/\xi \geqslant (1-\eta /(\xi\varepsilon_0 r_br_c))|B^{\prime\prime}|\geqslant (1-\sqrt \eta)|B^{\prime\prime}|.$$
By Lemma \ref{LEM:findmanyTTk}, $G[B^{\prime\prime}]$ contains at least
$\mu |B^{\prime\prime}|^b- (\varepsilon_0|B^{\prime\prime}|)^b \geqslant \mu |B^{\prime\prime}|^b/2$
``nice'' copies of $TT_b$ such that all vertices of each such copy have at least $\varepsilon_0 r_b|A^\prime|$ common in-neighbors in $A^\prime$.
    
For any fixed ``nice'' copy of $TT_b$, let $A^{\prime\prime}\subseteq A^\prime$ be the set of common in-neighbors of its vertices. In the same way, we have
$\delta(G[A^{\prime\prime}])\geqslant |A^{\prime\prime}|-\eta n \geqslant (1-\eta /(\xi \varepsilon_0^2 r_br_c))|A^{\prime\prime}|\geqslant (1-\sqrt \eta)|A^{\prime\prime}|$.
By Lemma \ref{LEM:findmanyTTk} again, there are at least $\mu |A^{\prime\prime}|^a$ copies of $TT_a$ in $G[A^{\prime\prime}]$. Recall that there are at least $\mu |C^{\prime\prime}|^c/2$ choices for $TT_c$ and $\mu |B^{\prime\prime}|^b/2$ choices for $TT_b$. After a simple calculation, $G$ contains at least
$$\mu |C^{\prime\prime}|^c/2 \cdot\mu |B^{\prime\prime}|^b/2 \cdot\mu |A^{\prime\prime}|^a \geqslant \beta n^{a+b+c}$$
copies of $D_{a,b,c}$, which completes the proof.
\end{proof} 

The remainder of this subsection is devoted to the proof of Lemma \ref{LEM:findmanyTTk}. In that proof, we will use a classical result from oriented Ramsey theory due to Erd\H{o}s and Moser \cite{erdosKMIHAS9}. For an oriented graph $H$, the oriented Ramsey number $\overrightarrow{r}(H)$ is the smallest integer $n$ such that every $n$-vertex tournament contains a copy of $H$. For simplicity, we write $\overrightarrow{r}(TT_k)$ as $\overrightarrow{r}(k)$. As noted in \cite{debiasioSJDM35,treglownJGT69}, the factorization problem for oriented graphs is closely related to oriented Ramsey numbers. Moreover, Erd\H{o}s and Moser \cite{erdosKMIHAS9} proved that
$2^{k/2-1}\leqslant \overrightarrow{r}(k)\leqslant 2^{k-1}$ for all $k\geqslant 2$.

The proof of Lemma \ref{LEM:findmanyTTk} also requires the following Counting Lemma.

\begin{lemma}[Counting Lemma, \cite{nagleRSA28}]\label{LEM-count}
Given $s\in \mathbb{N}$, let $0<1/m\ll \varepsilon\ll \delta \ll d<1$. Suppose that $G$ is a graph with a vertex partition $(V_1,V_2,\ldots,V_s)$, where $|V_i|=m$ for  $i\in[s]$. If $(V_i,V_j)$ is $(\varepsilon,d)$-regular for every $1\leqslant i < j \leqslant s$, then $G$ contains $(1\pm \delta)d^{s(s-1)/2}m^s$ copies of $K_s$.
\end{lemma}

\begin{proof}[\textbf{Proof of Lemma \ref{LEM:findmanyTTk}}]
Observe that we may assume that $s \geqslant 2$. Let
$$0<1/n\ll\eta\ll\mu\ll 1/M\ll\varepsilon\ll \delta\ll d\ll 2^{1-s},$$
where $M$ is the constant obtained from Lemma \ref{LEM:reguler} with parameters $\varepsilon,d,M^\prime=1/\varepsilon,M^{\prime\prime}=1$. 

Suppose that $G$ is an $n$-vertex oriented graph with $\delta(G)\geqslant (1-\eta)n$. Applying Lemma \ref{LEM:reguler} with parameters $\varepsilon,d,M^\prime=1/\varepsilon,M^{\prime\prime}=1$, there is a partition $V_0,V_1,\ldots,V_{k}$ of $V(G)$ with $k\geqslant 1/\varepsilon$ and $|V_i|=m>n/(2M)$ for each $i\in [k]$. Let $R$ be the reduced oriented graph of this partition. It follows from Lemma \ref{LEM-reducedori} that
$\delta(R)\geqslant (1-\eta-2d-3\varepsilon)k\geqslant (1-3d)k$.
Hence, we can greedily find a tournament on $2^{s-1}$ vertices as $k\geqslant 1/\varepsilon\gg 2^{s-1}$. 
Since $2^{s/2-1}\leqslant \overrightarrow{r}(s)\leqslant 2^{s-1}$, it follows that $R$ contains a copy of $TT_s$, say $H$.
We may assume w.l.o.g. that $V(H)=\{1,2,\ldots, s\}$ and that $ij\in E(H)$ for all $i<j$. Applying Lemma \ref{LEM-count} to the underlying graph of $G[\cup_{i\in[s]}V_i]$ with parameters $d$ and $\delta$, we obtain at least $d^{s(s-1)/2}m^s/2$ copies of $TT_s$ in $G$. Since $m\geqslant n/(2M)$, it follows that $G$ contains at least $\mu n^s$ copies of $TT_s$, completing the proof.
\end{proof}


\section{Proof of Lemma \ref{LEM:nonext}}\label{SEC-5.2}

In this section, we show that if the minimum semi-degree of $G$ is sufficiently large, then either $G$ is $(D_{a,b,c},\beta,t)$-closed or $G$ has an extremal structure.

Let $\mathcal{P}=(U_1,U_2,\ldots,U_p)$ be a partition of $V(G)$. The index vector of a set $S\subseteq V(G)$ with respect to $\mathcal{P}$, denoted by $i_{\mathcal{P}}(S)$, is the vector in $\mathbb{Z}^p$ whose $i$th coordinate is $|S\cap U_i|$ for each $i\in [p]$. For $j\in [p]$, let $\mathbf{e}_j\in \mathbb{Z}^p$ be the $j$th unit vector, that is, the vector whose $j$th coordinate is $1$ and whose other coordinates are $0$. A \emph{transferral} is a vector of the form $\mathbf{e}_i-\mathbf{e}_j$ for some distinct $i,j\in [p]$. A vector $\mathbf{v}\in \mathbb{Z}^p$ is called an $h$-vector if all its coordinates are non-negative and their sum is $h$. An $h$-vector $\mathbf{v}$ is called \emph{$(H,\beta)$-robust} if there are more than $\beta n^h$ copies of $H$ in $G$ whose vertex sets have index vector $\mathbf{v}$. Let $I^\beta_{\mathcal{P}}(H)$ be the set of all $(H,\beta)$-robust $h$-vectors, and let $L^\beta_{\mathcal{P}}(H)$ be the lattice generated by $I^\beta_{\mathcal{P}}(H)$, that is, the additive subgroup of $\mathbb{Z}^p$ generated by $I^\beta_{\mathcal{P}}(H)$. A \emph{$2$-transferral} is a transferral $\mathbf{v}\in L^\beta_{\mathcal{P}}(H)$ such that $\mathbf{v}=\mathbf{v}_1-\mathbf{v}_2$ for some $\mathbf{v}_1,\mathbf{v}_2\in I^\beta_{\mathcal{P}}(H)$. The lattice $L^\beta_{\mathcal{P}}(H)$ is called \emph{$2$-transferral-free} if it contains no $2$-transferral.

The following lemma can be used to construct a partition in which each part is closed.

\begin{lemma} \label{LEM-degreeclosepart}
    Given $a,b,c \in \mathbb{N}$, let $0<1/n\ll\eta\ll  \beta \ll \alpha \ll 1$. Suppose that $G$ is an $n$-vertex oriented graph with $\delta^0(G)\geqslant (1/2-\eta)n$. Then there exists an integer $t>0$ such that $V(G)$ admits a $(D_{a,b,c},\beta,t)$-closed partition $\mathcal{P}=(U_1,U_2,\ldots,U_p)$ with $p \leqslant \lceil 1/\alpha \rceil$ and $|U_i|\geqslant \alpha n/2$ for each $i\in [p]$.
\end{lemma}

The original statements of the following lemma and of Lemma \ref{LEM-closed-partition} later in \cite{hanRSA64} are phrased for graphs. However, their proofs also work for digraphs.  Lemma \ref{LEM-transclosed} gives a sufficient condition for merging two closed parts into a single part that remains closed.

\begin{lemma} [\cite{hanRSA64}] \label{LEM-transclosed}
Given $t,h\in \mathbb{N}$ with $h\geqslant 3$, let $0<1/n\ll\beta\ll 1$. Suppose that $H$ and $G$ are digraphs on $h$ and $n$ vertices, respectively. Let $\mathcal{P}=(U_1,U_2,\ldots,U_p)$ be an $(H,\beta,t)$-closed partition of $V(G)$. If there exists a $2$-transferral $\mathbf{v}=\mathbf{e}_i-\mathbf{e}_j\in L^\beta_{\mathcal{P}}(H)$, then $U_i\cup U_j$ is $(H,\beta',t')$-closed for some $\beta'\leqslant \beta$ and $t'\geqslant t$.
\end{lemma}

\begin{lemma}\label{LEM:transferral}
Given $a,b,c \in \mathbb{N}$ with $1\leqslant a\leqslant b\leqslant c$, let $0<1/n\ll \eta\ll \beta \ll\gamma \ll \alpha\ll 1$. Suppose that $G$ is an $n$-vertex oriented graph with $\delta^0(G)\geqslant(1/2-\eta)n$ and that $\mathcal{P}=(U_1,U_2,\ldots,U_p)$ is a partition of $V(G)$ with $p\geqslant 2$ and $|U_i|\geqslant \alpha n$ for each $i\in [p]$.
    
    \rm{(i)} If $a+1=b=c$ or $a+1=b+1=c$, then there exists a $2$-transferral $\mathbf{v} \in L_\mathcal{P}^{\beta}(D_{a,b,c})$.
    
    \rm{(ii)} If $L_\mathcal{P}^{\beta}(D_{a,b,c})$ is $2$-transferral-free, then $G$ is $\gamma$-extremal.
\end{lemma}
We are now ready to prove Lemma \ref{LEM:nonext}.
\begin{proof}[\textbf{Proof of Lemma \ref{LEM:nonext}}]
Given $a,b,c\in\mathbb{N}$ with $1\leqslant a\leqslant b\leqslant c$, let
$$
0<1/n\ll\eta\ll\beta\ll\beta^\prime\ll\gamma\ll \alpha\ll1.
$$

Suppose that $G$ is an $n$-vertex oriented graph with $\delta^0(G)\geqslant(1/2-\eta)n$. By Lemma \ref{LEM-degreeclosepart}, there exists an integer $t^\prime>0$ such that $V(G)$ admits a $(D_{a,b,c},\beta^\prime,t^\prime)$-closed partition $\mathcal{P}^\prime=(U_1,U_2,\ldots,U_p)$ with $p\leqslant \lceil 1/\alpha\rceil$ and $|U_i|\geqslant \alpha n/2$ for each $i\in [p]$. 

If $L^{\beta^\prime}_{\mathcal{P}^\prime}(D_{a,b,c})$ contains a $2$-transferral $\mathbf{v}=\mathbf{e}_i-\mathbf{e}_j$ for some distinct $i,j\in [p]$, then by Lemma \ref{LEM-transclosed} we may merge the corresponding parts $U_i$ and $U_j$, that is, replace $\mathcal{P}^\prime$ with the partition
$$
\mathcal{P}^{\prime\prime}=(\mathcal{P}^\prime\setminus \{U_i,U_j\})\cup \{U_i\cup U_j\}.
$$
Repeating this procedure for all $2$-transferrals, we eventually obtain a $(D_{a,b,c},\beta,t)$-closed partition $\mathcal{P}$ with $t\geqslant t^\prime$ such that $L^\beta_{\mathcal{P}}(D_{a,b,c})$ is $2$-transferral-free. Since $V(G)$ is not $(D_{a,b,c},\beta,t)$-closed, the partition $\mathcal{P}$ must contain at least two parts. The lemma now follows immediately from Lemma \ref{LEM:transferral}, applied with $\alpha/2$ in place of $\alpha$. 
\end{proof}

\subsection{Proof of Lemma \ref{LEM-degreeclosepart}}

Note that Lemma \ref{LEM-degreeclosepart} follows directly from the following lemmas.

\begin{lemma}[\cite{hanRSA64}]\label{LEM-closed-partition}
Given an integer $h\geqslant 3$, let $0<1/n \ll \beta^\prime \ll \beta \ll \alpha \ll 1$. Suppose that $H$ and $G$ are digraphs on $h$ and $n$ vertices, respectively. If every vertex in $V(G)$ is $(H,\beta,1)$-reachable to at least $\alpha n$ vertices, then there exists $t\in\mathbb{N}$ such that $V(G)$ admits an $(H,\beta^\prime,t)$-closed partition $\mathcal{P}=(U_1,U_2,\ldots,U_p)$ with $p\leqslant \lceil 1/\alpha\rceil$ and $|U_i|\geqslant \alpha n/2$ for each $i\in [p]$.
\end{lemma}

\begin{lemma} \label{LEM:Dabcs-contain-v and reachable}
Given $a,b,c\in \mathbb{N}$ with $1\leqslant a\leqslant b\leqslant c$, let $0<1/n\ll \eta\ll \beta \ll \alpha \ll 1$. 
If $G$ is an $n$-vertex oriented graph with $\delta^0(G)\geqslant (1/2-\eta)n$, then every vertex in $G$ is $(D_{a,b,c},\beta,1)$-reachable to at least $\alpha n$ other vertices of $G$.
\end{lemma}
We remark that the proof of Lemma \ref{LEM-closed-partition} is the same as that of Lemma 4.1 in \cite{hanRSA64}, whose statement is phrased for graphs. Before proving Lemma \ref{LEM:Dabcs-contain-v and reachable}, we introduce a useful auxiliary lemma. Its proof will be given at the end of this subsection.

\begin{lemma}\label{LEM:cross-edge-triangle}
    Let $\eta$ be a constant with $0<\eta\leqslant 1/100$. Then the following statements hold for every $n$-vertex oriented graph $G$ with $\delta^0(G)\geqslant (1/2-\eta)n$. 

   (i)  For any two disjoint subsets $A,B\subseteq V(G)$, if $|A|\geqslant (1/4+2\eta)n$ and $|B|\geqslant (1/2-\eta)n$, then there is an edge from $A$ to $B$ and an edge from $B$ to $A$ in $G$.
   
   (ii) For any partition $A,B$ of $V(G)$, there exists a directed triangle  $C_3$ such that $V(C_3)\cap A\neq \emptyset$ and $V(C_3)\cap B\neq \emptyset$.
\end{lemma}
Equipped with Lemma \ref{LEM:cross-edge-triangle}, we are ready to prove Lemma \ref{LEM:Dabcs-contain-v and reachable}.

\begin{proof}[\textbf{Proof of Lemma \ref{LEM:Dabcs-contain-v and reachable}}] 
Given $a,b,c\in \mathbb{N}$ with $1\leqslant a\leqslant b\leqslant c$, choose
$$
0<1/n\ll \eta \ll\beta \ll 1/M\ll \varepsilon \ll d\leqslant 1/400,
$$
where $M$ is given by Lemma \ref{LEM:reguler} with parameters $\varepsilon,d$, $M^\prime=1/\varepsilon$, and $M^{\prime\prime}=2$. Set $\alpha=1/(4M)$.

Fix a vertex $v\in V(G)$, and set $A=N^{+}(v)$, $B=N^{-}(v)$, and $C=V(G)\setminus (A\cup B)$. Since $\delta^{0}(G)\geqslant (1/2-\eta)n$, we have $|C|\leqslant 2\eta n$. Hence
$$
\delta^{0}(G[A\cup B])\geqslant (1/2-\eta)n-|C|
\geqslant (1/2-3\eta)|A\cup B|.
$$
Applying Lemma \ref{LEM:reguler} to $G[A\cup B]$ with parameters $\varepsilon,d$, $M^\prime=1/\varepsilon$, and $M^{\prime\prime}=2$, we obtain a partition $V_0,V_1,\ldots,V_k$ 
of $A\cup B$, where  $1/\varepsilon \leqslant k\leqslant M$, $|V_j|=m\geqslant n/(2M)$ for each $j\in[k]$, $V_j\subseteq A$ for each $j\in[i]$, and $V_j\subseteq B$ for each $i+1\leqslant j\leqslant k$. Let $R$ be the reduced oriented graph. By Lemma \ref{LEM-reducedori}, we have 
$$
\delta^0(R)\geqslant (1/2-3\eta-d-3\varepsilon)k\geqslant (1/2-2d)k.
$$

We may  assume w.l.o.g. that $R_A=\{1,2,\ldots,i\}$ and $R_B=\{i+1,i+2,\ldots,k\}$. Then $(R_A,R_B)$ is a partition of $V(R)$. Applying Lemma \ref{LEM:cross-edge-triangle} (ii) to this partition, we find a triangle $T$ in $R$ such that $V(T)\cap R_A\neq \emptyset$ and $V(T)\cap R_B\neq \emptyset$. We  may assume that $|V(T)\cap R_A|=2$ and write $T=12k$. In particular, $1,2\in R_A$ and $k\in R_B$. Let $S$ be the set of vertices $u\in V_1$ such that $
d^{+}(u,V_2)\geqslant dm/2
 \text{ and } 
d^{-}(u,V_k)\geqslant dm/2$. 
We claim that $v$ is $(D_{a,b,c},\beta,1)$-reachable to every vertex in $S$.

Fix any $u\in S$, and set $N_2=N^{+}(u,V_2)$ and $N_k=N^{-}(u,V_k)$. By the definition of $S$, we have $|N_i|\geqslant dm/2\geqslant dn/(4M)$ for each $i\in\{2,k\}$. Since $\delta^0(G)\geqslant (1/2-\eta)n$, every vertex of $G$ misses at most $2\eta n$ edges. Hence, for some $\sigma\in\{+,-\}$, we have $
|N^\sigma(u,V_1)|\geqslant (m-2\eta n)/2\geqslant m/3$, 
where we used $m\geqslant n/(2M)$ and $\eta\ll 1/M$. Set $N_1=N^\sigma(u,V_1)$.

Since $12k$ is a triangle in $R$, each of $(V_1,V_2)$, $(V_2,V_k)$, and $(V_k,V_1)$ is $(\varepsilon,d)$-regular. By Lemma~\ref{LEM-regulartosuper} and the fact that $|N_1|,|N_2|,|N_k|\geqslant dm/2$ with $\varepsilon\ll d$, it follows that $(N_1,N_2)$, $(N_2,N_k)$, and $(N_k,N_1)$ are $(\varepsilon^{1/2},d/2)$-regular. Applying Lemma \ref{LEM:findmanyDabc} with $\xi=d/(4M)$, we conclude that $G[N_1\cup N_2\cup N_k]$ contains at least $\beta n^{a+b+c-1}$ copies of $D_{a-1,b,c}$ of type $(N_1,N_2,N_k)$ (where possibly $a-1=0$). Note that $
N_1\subseteq V_1\subseteq N^+(v), 
N_2\subseteq V_2\subseteq N^+(v), 
N_k\subseteq V_k\subseteq N^-(v)$.
Therefore, every such copy of $D_{a-1,b,c}$ together with $v$ and together with $u$, respectively, forms a copy of $D_{a,b,c}$. This proves the claim.

Finally, since both $(V_1,V_2)$ and $(V_k,V_1)$ are $(\varepsilon,d)$-regular, Lemma \ref{LEM-fewvxswithsmalld} together with $m\geqslant n/(2M)$ and $\alpha=1/(4M)$ implies that $
|S|\geqslant (1-2\varepsilon)m\geqslant \alpha n$.  
Hence, $v$ is $(D_{a,b,c},\beta,1)$-reachable to every vertex in $S$, and this  completes the proof.
\end{proof}
At the end of this subsection, we give a proof of Lemma \ref{LEM:cross-edge-triangle}.
\begin{proof}[\textbf{Proof of Lemma \ref{LEM:cross-edge-triangle}}] 
Let $0<\eta\leqslant 1/100$, and let $G$ be an $n$-vertex oriented graph with $\delta^{0}(G)\geqslant (1/2-\eta)n$. Suppose that $A$ and $B$ are two disjoint subsets of $V(G)$.

To prove (i), we may assume that $|A|=(1/4+2\eta)n$ and $|B|=(1/2-\eta)n$. Since $G[B]$ is oriented, there exists a vertex $b\in B$ with $d^-(b,B)\leqslant |B|/2$. It follows from $\delta^0(G)\geqslant (1/2-\eta)n$ that $d^-(b,A)\geqslant (1/2-\eta)n-(n-|A|-|B|)-|B|/2>0$. 
Thus, there is an edge from $A$ to $b$, and hence from $A$ to $B$. Similarly, there exists  $b^\prime\in B$ with $d^+(b^\prime,B)\leqslant |B|/2$, and therefore $G$ contains an edge from $B$ to $A$.

Now we prove (ii). Let $A,B$ be a partition of $V(G)$, and assume w.l.o.g. that $|A|\leqslant |B|$. If there exists  $a\in A$ with $d^+(a,B)\geqslant (1/4+2\eta)n$, then Lemma \ref{LEM:cross-edge-triangle} (i) implies that there is an edge from $N^+(a,B)$ to $N^-(a)$, which together with $a$ forms the desired triangle. Thus, we may assume that $d^+(a,B)<(1/4+2\eta)n$ for every $a\in A$. 
By symmetry, $d^-(a,B)<(1/4+2\eta)n$ for every $a\in A$. 
Since $V(G)=A\cup B$ and $\delta^0(G)\geqslant (1/2-\eta)n$, it follows that $d^\pm(a,A)>(1/4-3\eta)n$ for every $a\in A$, and hence
$$
(1/2-6\eta)n<|A|\leqslant |B|<(1/2+6\eta)n.
$$

Fix any $u\in A$, and let $B_1=N^+(u,B), B_2=N^-(u,B), B_3=B\setminus(B_1\cup B_2)$. 
Since $d^\pm(u,B)<(1/4+2\eta)n$, we have $|B_1|,|B_2|<(1/4+2\eta)n$. Moreover, $\delta^0(G)\geqslant (1/2-\eta)n$ implies that $|B_3|\leqslant 2\eta n$. We may assume that $e(B_1,B_2)=0$, since otherwise any edge in $E(B_1,B_2)$ together with $u$ forms the desired $C_3$. Now $B_2$ contains a vertex, say $w$, with $d^{-}(w,B_2)\leqslant |B_2|/2$. Therefore,
\begin{align*}
d^{-}(w,A)
&\geqslant (1/2-\eta)n-|B_2|/2-|B_3|\\
&\geqslant (1/2-\eta)n-(n/8+\eta n)-2\eta n\\
&>(1/4+2\eta)n.
\end{align*}
By Lemma \ref{LEM:cross-edge-triangle} (i), there is an edge from $N^+(w)$ to $N^-(w,A)$, which together with $w$ forms the desired triangle. 
\end{proof}

\subsection{Proof of Lemma \ref{LEM:transferral}}

Let $G$ be an $n$-vertex oriented graph. For  $V_1,V_2,V_3\subseteq V(G)$, let $\operatorname{cyc}(V_1,V_2,V_3)$ denote the  number of disjoint directed triangles $v_1v_2v_3$ such that $v_i\in V_i$ for each $i\in[3]$. For any set $A\subseteq V(G)$, let $\overline{A}=V(G)\setminus A$, and define the strong $\beta$-out-neighborhood of $A$ in $G$ by $\operatorname{SN}_{\beta}^+(A)=\{x\in \overline{A}: d_G^-(x,A)\geqslant |A|-\beta n\}$. 
The strong $\beta$-in-neighborhood $\operatorname{SN}_{\beta}^-(A)$ is defined analogously.

\begin{lemma}[\cite{liEJC26}]\label{LEM:divide lemma}
Let $1/n \ll \eta \ll \varepsilon \ll \beta \ll \gamma \ll \alpha < 1,$
and let $G$ be an $n$-vertex oriented graph with  $\delta^0(G) \geqslant(1/2- \eta)n$. For every $A\subseteq V(G)$ such that $|A|\geqslant \alpha n$ and $\mathrm{cyc}(A,A,\overline{A}) \leqslant \varepsilon n^3$, we have 
\begin{itemize}
    \item $|\text{SN}^+_\beta(A)|,\, |\text{SN}^-_\beta(A)| = |\overline{A}|/2 \pm \gamma n$, and
    \item $|A| \leqslant \left(1/3 + \gamma\right)n$.
\end{itemize}
\end{lemma}

For any vertices $u,v\in V(G)$ and $\sigma,\tau\in\{+,-\}$, let $N^{\sigma,\tau}(uv)=N^\sigma(u)\cap N^\tau(v)$, and for each set $A\subseteq V(G)$, let $N^{\sigma,\tau}(uv,A)=N^{\sigma,\tau}(uv)\cap A$. We write $d^{\sigma,\tau}(uv)=|N^{\sigma,\tau}(uv)|$ and $d^{\sigma,\tau}(uv,A)=|N^{\sigma,\tau}(uv,A)|$.

\begin{lemma}\label{LEM:-+=+-}
    Let $\eta>0$, and let $G$ be an $n$-vertex oriented graph with $\delta^0(G)\geqslant(1/2-\eta)n$. Then, for any vertices $x,y\in V(G)$ and  $\sigma,\tau\in\{+,-\}$, we have
$$
\bigl|d^{\sigma,\tau}(xy)-d^{-\sigma,-\tau}(xy)\bigr|\leqslant 4\eta n.
$$
\end{lemma}
    \begin{proof} Observe that $d^{\sigma,\tau}(xy)=d^\sigma(x)-d^{\sigma,-\tau}(xy)-|N^\sigma(x)\setminus N(y)|$ and $d^{-\sigma,-\tau}(xy)=d^{-\tau}(y)-d^{\sigma,-\tau}(xy)-|N^{-\tau}(y)\setminus N(x)|$. The claim therefore follows immediately from the fact that $G$ is an oriented graph with $\delta^0(G)\geqslant(1/2-\eta)n$.
    \end{proof}

Before proving Lemma \ref{LEM:transferral}, we first introduce the following definition.

\begin{definition}
    Let $K^-_4$ be the oriented graph on vertex set $\{v_1,v_2,v_3,v_4\}$ such that both $v_1v_2v_3$ and $v_2v_3v_4$ are directed triangles. 
    Given four sets $A,B,C,D$, a copy of $K^-_4$ is said to be of type-$(A,B,C,D)$ if $v_1\in A$, $v_2\in B$, $v_3\in C$, and $v_4\in D$.
\end{definition}

Let $k_4^-(A,B,C,D)$ denote the number of copies of $K^-_4$ of type-$(A,B,C,D)$.

\begin{lemma}\label{LEM:triangle to extremal}
    Given $1/n\ll\eta\ll\varepsilon \ll \gamma \ll \alpha< 1$, let $G$ be an $n$-vertex oriented graph with $\delta^0(G)\geqslant(1/2-\eta)n$. For any subset $A\subseteq V(G)$ with $|A|\geqslant\alpha n$, if $\text{cyc}(A,A,\overline{A})\leqslant\varepsilon n^3$ and $k_4^-(A,V(G),V(G),\overline{A})\leqslant \varepsilon n^4$, then $G$ is $\gamma$-extremal.
\end{lemma}

\begin{proof} 
We choose $\beta$ such that $\varepsilon \ll \beta \ll \gamma$. For simplicity, we write $\text{SN}^\sigma_\beta(A)$ as $S^\sigma$ for  $\sigma\in\{+,-\}$. Let $S=V(G) \backslash(S^+\cup S^-\cup A)$. Applying Lemma \ref{LEM:divide lemma} with $\gamma=\beta$ and using the fact that $\text{cyc}(A,A,\overline{A})\leqslant \varepsilon n^3$, we obtain $|S^-|=|S^+|\pm 2\beta n$, $|A|\leqslant(1/3+\beta)n$, and $|S|\leqslant2\beta n$.

Next we claim that $|A\cup S|$, $|S^+|$, and $|S^-|$ are all approximately equal. Observe that $\text{cyc}(S^+, S^+, A)\leqslant \beta n \cdot \binom{|S^+|}{2} \leqslant \beta n^3$ and $\text{cyc}(S^+, S^+, S)\leqslant 2\beta n^3$ since $|S|\leqslant 2\beta n$. For any vertices $x\in S^+$ and $y\in S^-$, we have $d^-(x,A),d^+(y,A)\geqslant |A|-\beta n$ by the definition. Thus, at least $|A|-2\beta n$ vertices in $A$ form a directed triangle of type-$(A,S^+,S^-)$ together with $x$ and $y$. Therefore,
$$
(|A|-2\beta n)\cdot\text{cyc}(S^+,S^-,S^+)\leqslant k_4^-(A,S^+,S^-,S^+)\leqslant k_4^-(A,V(G),V(G),\overline{A})\leqslant \varepsilon n^4,
$$
which implies that $\text{cyc}(S^+,S^+,S^-)=\text{cyc}(S^+,S^-,S^+)\leqslant\varepsilon n^3$ since $|A|\geqslant \alpha n$. Thus, 
$$
\text{cyc}(S^+, S^+, \overline{S^+})=\text{cyc}(S^+,S^+,S^-)+\text{cyc}(S^+, S^+, A)+\text{cyc}(S^+, S^+, S)\leqslant 4\beta n^3.
$$
Applying Lemma \ref{LEM:divide lemma} with $(A,\varepsilon,\gamma)=(S^+,4\beta,\gamma)$, we obtain $|S^+|\leqslant (1/3+\gamma)n$. Moreover, since $|S^-|=|S^+|\pm 2\beta n$, we also have $|S^-|\leqslant (1/3+2\gamma)n$. Hence, $(1/3-\gamma)n\leqslant |A\cup S|\leqslant (1/3+3\beta)n$. 
It is then easy to check that $|S^+|,|S^-|,|A\cup S|=n/3\pm O(\gamma n)$.

We now claim that $G$ is $\gamma$-extremal by showing that $e(S^+,A\cup S),\ e(A\cup S,S^-),\ e(S^-, S^+) = O(\gamma n^2)$. 
Recall that for each $x\in S^+$ and $y\in S^-$, we have $d^-(x,A),d^+(y,A)\geqslant |A|-\beta n$. Thus, $e(S^+,A\cup S)\leqslant (\beta n+|S|)\cdot |S^+|\leqslant 3\beta n|S^+|=O(\gamma n^2)$. 
Similarly, we have $e(A\cup S,S^-)=O(\gamma n^2)$. The degree condition $\delta^0(G)\geqslant (1/2-\eta)n$ implies that every vertex has at most $2 \eta n$ non-neighbors in $G$. Hence,
\begin{align*}
    e(S^-,S^+)&\leqslant (n/2+\eta n) |S^+|- e(A\cup S,S^+)-e(S^+)\\
    &\leqslant (n/2+\eta n -|A\cup S|-|S^+|/2)|S^+|  + O(\gamma n^2)\\
    &= O(\gamma n^2),
\end{align*}
which completes the proof.
\end{proof}

Now we are ready to prove Lemma \ref{LEM:transferral}. 

\begin{proof}[\textbf{Proof of Lemma \ref{LEM:transferral}}]

Given $a,b,c\in \mathbb{N}$ with $1\leqslant a\leqslant b\leqslant c$, let $h=a+b+c$. Choose
$$0<1/n\ll \eta \ll\beta \ll 1/M \ll\varepsilon\ll d\ll \gamma 
\ll\alpha\ll  1/100. $$ 
    
Let $\mathcal{P}=(U_1,U_2,\ldots,U_p)$ be a partition of $V(G)$ with $p\geqslant 2$ and $|U_i|\geqslant \alpha n$ for each $i\in[p]$.
Applying Lemma \ref{LEM:reguler} to $G$ with parameters $\varepsilon$, $d$, $M^\prime=1/\varepsilon$, and $M^{\prime\prime}=p$, we obtain a refined partition $X_0,X_1,\ldots,X_k$ of $V(G)$, where $1/\varepsilon \leqslant k\leqslant M$ and $|X_0|\leqslant \varepsilon n$. Moreover, $|X_i|=m$ for each $i\in [k]$, and hence $n=km+|X_0|$ implies that
$k\leqslant n/m\leqslant 2k\leqslant 2M$.
Let $R^{\prime}$ be the reduced digraph of this partition, with vertex set $[k]$.
For each $i\in [p]$, set
$R_i=\{j\in [k]: X_j \subseteq U_i\}$.
Observe that $\mathcal{P}_R=(R_1,R_2,\ldots,R_p)$ is a partition of $[k]$. Furthermore, we have
$$|R_i|\geqslant|U_i\backslash X_0|/m\geqslant(\alpha n-\varepsilon n)/m\geqslant\alpha k/2.$$
By Lemma \ref{LEM-reducedori}, $R^{\prime}$ contains a spanning oriented graph $R$ with
$\delta ^0(R)\geqslant (1/2-\eta-d-3\varepsilon )k\geqslant (1/2-2d)k$.

We now complete the proof of part (i) of Lemma \ref{LEM:transferral}. Applying Lemma \ref{LEM:cross-edge-triangle} (ii) to $R$, there exists a directed triangle $T$ in $R$ with
$V(T)\cap R_1\neq\emptyset$ and $V(T)\cap \overline{R_1}\neq\emptyset$. We may  assume w.l.o.g. that $T=123$, with $1\in R_1$ and $2\in \overline{R_1}$.
Clearly, $(X_1,X_2)$, $(X_2,X_3)$, and $(X_3,X_1)$ are $(\varepsilon,d)$-regular.
Applying Lemma \ref{LEM:findmanyDabc} three times with parameters $\eta,\beta,\xi=1/(2M),\varepsilon,d$, we see that, for each $j\in[3]$, there are at least $\beta n^{h}$ copies of $D_{a,b,c}$ of type-$(j,j+1,j+2)$ in $G[X_1\cup X_2\cup X_3]$. Let $\mathcal{H}_j$ denote the family of copies of $D_{a,b,c}$ of type-$(j,j+1,j+2)$. Note that if $a+1=b=c$, then $\mathcal{H}_1$ and $\mathcal{H}_2$ yield two vectors $\mathbf{v}_1,\mathbf{v}_2\in I^{\beta}_\mathcal{P}(D_{a,b,c})$ such that
$\mathbf{v}_2-\mathbf{v}_1=\mathbf{e}_1-\mathbf{e}_2$.
If $a+1=b+1=c$, then $\mathcal{H}_2$ and $\mathcal{H}_3$ yield two vectors $\mathbf{v}_2,\mathbf{v}_3\in I^{\beta}_\mathcal{P}(D_{a,b,c})$ such that
$\mathbf{v}_2-\mathbf{v}_3=\mathbf{e}_1-\mathbf{e}_2$.
In either case, there exists a $2$-transferral in $L^{\beta}_{\mathcal{P}}(D_{a,b,c})$, as required.

\medskip

We now prove part (ii) of Lemma \ref{LEM:transferral}. Suppose that $L_\mathcal{P}^\beta(D_{a,b,c})$ is $2$-transferral-free. We first claim that $R$ contains no ``cross'' $K_4^-$.
     
\begin{claim}\label{CLM-K4-}
There is no $K_4^-$ of type-$(R_i,V(R),V(R),\overline{R_i})$ in $R$ for any $i\in [p]$.
\end{claim}

\begin{proof}[\textbf{Proof of Claim \ref{CLM-K4-}}]
Suppose to the contrary that $H$ is a copy of $K_4^-$ of type-$(R_i,V(R),V(R),\overline{R_i})$ in $R$ for some $i\in [p]$. For simplicity, we may assume w.l.o.g. that $i=1$ and that $V(H)=\{1,2,3,4\}$, where $1\in R_1$ and $4\in R_2$, and $E(H)=\{12,23,31,34,42\}$. Then $(X_1,X_2)$, $(X_2,X_3)$, $(X_3,X_1)$, $(X_3,X_4)$, and $(X_4,X_2)$ are $(\varepsilon,d)$-regular.

By Lemma \ref{LEM:findmanyDabc}, there are at least $\beta n^{h}$ copies of $D_{a,b,c}$ of type-$(1,2,3)$ in $G[X_1\cup X_2\cup X_3]$. These copies of $D_{a,b,c}$ give rise to an index vector, say $\mathbf{s}\in I_\mathcal{P}^\beta(D_{a,b,c})$. Moreover, since $(X_3,X_4)$ and $(X_4,X_2)$ are $(\varepsilon,d)$-regular, Lemma \ref{LEM-fewvxswithsmalld} implies that all but at most $2\varepsilon m$ vertices of $X_4$ have at least $(d-\varepsilon)|X_2|$ out-neighbors in $X_2$ and at least $(d-\varepsilon)|X_3|$ in-neighbors in $X_3$. Let $v$ be any such good vertex in $X_4$. Then there are at least $(1-2\varepsilon)m$ possible choices for $v$, and moreover $d^+(v,X_2),\ d^-(v,X_3)\geqslant (d-\varepsilon)m\geqslant dm/2$.

Since $\delta^0(G)\geqslant (1/2-\eta)n$ and $n/m\leqslant 2M$, we have $d(v,X_1)\geqslant m-2\eta n > 2m/3$, 
and hence $d^{\sigma}(v,X_1)\geqslant m/3$ for some $\sigma\in\{+,-\}$. Let $X^\ast_1=N^{\sigma}(v,X_1), 
X^\ast_2=N^+(v,X_2), 
X^\ast_3=N^-(v,X_3)$.
Clearly, each $X_i^\ast$ with $i\in[3]$ has size at least $dm/2$. By Lemma \ref{LEM-regulartosuper}, $(X^\ast_1,X^\ast_2)$, $(X^\ast_2,X^\ast_3)$, and $(X^\ast_3,X^\ast_1)$ are $(\sqrt\varepsilon,d/2)$-regular. Furthermore, since $\varepsilon\ll d$, Lemma \ref{LEM:findmanyDabc} shows that $G[X^\ast_1\cup X^\ast_2\cup X^\ast_3]$ contains at least $\sqrt{\beta}\, n^{h-1}$ copies of $D_{a,b,c-1}$ of type-$(2,3,1)$. By the choice of $X^\ast_1,X^\ast_2$, and $X^\ast_3$, each such copy together with $v$ forms a copy of $D_{a,b,c}$.

Recall that there are at least $(1-2\varepsilon)m$ choices for the vertex $v$. Therefore, there are at least
$$
(1-2\varepsilon)m\sqrt{\beta}\, n^{a+b+c-1}\geqslant \beta n^{a+b+c}
$$
copies of $D_{a,b,c}$, since $m\geqslant n/(2M)$ and $\beta\ll 1/M$. These copies yield an index vector $\mathbf{t}\in I_\mathcal{P}^\beta(D_{a,b,c})$. Then $\mathbf{s}-\mathbf{t}=\mathbf{e}_1-\mathbf{e}_2$ is a $2$-transferral in $L_\mathcal{P}^\beta(D_{a,b,c})$, contradicting our assumption. This completes the proof.
\end{proof}

For convenience, let $A=R_1$ and $B=V(R)\setminus R_1$. Then
$|A|,|B|\geqslant |R_i|\geqslant \alpha k/2$.
We first show that either $\text{cyc}(A,A,B)\leqslant \sqrt d\, k^3$ or $\text{cyc}(B,B,A)\leqslant \sqrt d\, k^3$ in $R$. Using this fact, we then prove that $G$ is $\gamma$-extremal. We may  assume w.l.o.g. that $\text{cyc}(A,A,B)\leqslant \sqrt d\, k^3$, since the argument in the case $\text{cyc}(B,B,A)\leqslant \sqrt d\, k^3$ is identical. Combining the fact that there is no $K^-_4$ of type-$(A,V(R),V(R),\overline{A})$ with Lemma \ref{LEM:triangle to extremal} applied with $\varepsilon=\sqrt d$ and $\alpha=\alpha/2$, we conclude that $R$ admits a $\gamma$-extremal partition $(V_1,V_2,V_3)$. Let $V^\ast_i=\bigcup_{j\in V_i}X_j$ for each $i\in[3]$, and put all vertices of $X_0$ into $V^\ast_1$. Clearly, $(V^\ast_1,V^\ast_2,V^\ast_3)$ is a partition of $V(G)$. A simple calculation shows that
$|V^\ast_i|=(k/3\pm O(\gamma k))m\pm |X_0|=n/3\pm O(\gamma n)$,
since $|V_i|=k/3\pm O(\gamma k)$. Combining
$e(V_{i+1},V_i)=O(\gamma k^2)$ with (iii) of Lemma \ref{LEM-reducedori}, we obtain
$e(V^\ast_{i+1},V^\ast_i)=O(\gamma n^2)$,
using the fact that $|X_0|\leqslant \varepsilon n\ll \gamma n$. Hence $(V^\ast_1,V^\ast_2,V^\ast_3)$ is a $\gamma$-extremal partition of $V(G)$, proving part (ii) of Lemma \ref{LEM:transferral}.

Thus, it remains to show that either $\text{cyc}(A,A,B)\leqslant \sqrt d\, k^3$ or $\text{cyc}(B,B,A)\leqslant \sqrt d\, k^3$ holds in $R$. If $\text{cyc}(A,A,B)\leqslant \sqrt d\, k^3$, then there is nothing to prove. Thus, we may assume w.l.o.g. that $\text{cyc}(A,A,B)> \sqrt d\, k^3$. Then there exists an edge $uv\in E_R(A,B)$ such that
$|N_R^{-,+}(uv,A)|\geqslant \sqrt d\, k^3/(|A||B|)\geqslant \sqrt d\, k$.
Moreover, we have $N_R^{-,+}(uv,B)=\emptyset$, since otherwise there would exist a cross $K_4^-$, contradicting Claim \ref{CLM-K4-}. For each $S\in\{A,B\}$, let
$S_1=N_R^{-,+}(uv,S)$, $S_2=N_R^{+,-}(uv,S)$, $S_3=N_R^{+,+}(uv,S)$, and $S_4=N_R^{-,-}(uv,S)$, and set
$S^r=S\setminus(S_1\cup S_2\cup S_3\cup S_4)$.
Note that $u\in A^r$ and $v\in B^r$. Clearly,
$|A_1|=|N_R^{-,+}(uv,A)|\geqslant \sqrt d\, k$,
$|B_1|=|N_R^{-,+}(uv,B)|=0$, and
$|A^r|+|B^r|\leqslant 8dk+2$ since $\delta^0(R)\geqslant(1/2-2d)k$.

For disjoint sets $X,Y\subseteq V(R)$, we write $X \Rightarrow Y$ to mean that all edges between $X$ and $Y$ are directed from $X$ to $Y$. We now determine the adjacency relations among several of these sets. It is easy to check that
$A_4\cup B_4\Rightarrow A_3\cup B_3$,
since otherwise $R$ would contain a cross $K_4^-$ of type-$(u,A_3\cup B_3,A_4\cup B_4,v)$. Similarly, we have
$A_1\Rightarrow A_2\cup A_3$
and
$B_2\cup B_4\Rightarrow A_1$,
since otherwise $R$ would contain a cross $K_4^-$ of type-$(A_2\cup A_3,A_1,u,v)$ or type-$(B_2\cup B_4,v,A_1,u)$. Note that for any pair of vertices $x,y\in V(R)$ (possibly $x=y$), all but at most $8dk$ vertices of $R$ are common neighbors of $x$ and $y$. Thus $x$ and $y$ have common neighbors in $A_1$, since $|A_1|\geqslant \sqrt d\, k>8dk$. This implies that $A_2\Rightarrow B_3$. Indeed, suppose that $b_3a_2$ is an edge from $B_3$ to $A_2$. Then there exists $a_1\in N_R^{-}(a_2,A_1)$, since $d(a_2,A)>0$ and $A_1 \Rightarrow A_2$. It follows that $R[\{a_1,a_2,v,b_3\}]$ is a cross $K_4^-$, contradicting Claim \ref{CLM-K4-}. By the same argument, we also have
$B_2\Rightarrow A_3$
and
$B_4\Rightarrow B_2$,
since otherwise there would exist a cross $K_4^-$ of type-$(A_1,A_3,B_2,v)$ or type-$(A_1,u,B_2,B_4)$. In summary, we have 

\smallskip
\begin{tasks}[label=\textbullet](3)
    \task $A_4\cup B_4\Rightarrow A_3\cup B_3$;
    \task $A_1\Rightarrow A_2\cup A_3$;
    \task $B_2\cup B_4\Rightarrow A_1$;
    \task $A_2\Rightarrow B_3$;
    \task $B_2\Rightarrow A_3$;
    \task $B_4\Rightarrow B_2$.
\end{tasks}
\smallskip

Next we claim that $\text{cyc}(B,B,A)\leqslant \sqrt d\, k^3$ in $R$. To prove this, we need the following result.

\begin{claim}\label{CLAIM:type of BBA}
\text{ }
\settasks{label=\upshape(\arabic*)}
\begin{tasks}(2)
    \task $\text{cyc}(B_2\cup B_3,B_4,A)=0$\rm{;}
    \task $\text{cyc}(B_2\cup B_4,B,A_1\cup A_3)=0$\rm{;}
    \task $\text{cyc}(B,B_3,A_2\cup A_4)=0$\rm{;}
    \task $\text{cyc}(B_3,B_2,A)=0$\rm{;}
    \task $\text{cyc}(B,B_2\cup B_3,A_1\cup A_4)=0$\rm{;}
    \task $\text{cyc}(B_2\cup B_4,B,A_2\cup A_3)=0$.
\end{tasks}
\end{claim}

\begin{proof}[\textbf{Proof of Claim \ref{CLAIM:type of BBA}}]
It is easy to verify (1)--(3), since
$B_4\Rightarrow B_2\cup B_3$,
$B_2\cup B_4\Rightarrow A_1\cup A_3$,
and
$A_2\cup A_4\Rightarrow B_3$.
If (4) fails, then there exists a cross $K_4^-$ of type-$(A,B_3,B_2,v)$, since $B_2\Rightarrow v$ and $v\Rightarrow B_3$.
Condition (5) can be verified similarly, for otherwise there exists a cross $K_4^-$ of type-$(u,B_2\cup B_3,A_1\cup A_4,B)$.
Moreover, since $A_1\Rightarrow A_2\cup A_3$, $B_2\cup B_4\Rightarrow A_1$, and $|A_1|>8dk$, condition (6) must also hold; otherwise, some vertex in $A_1$ forms a cross $K_4^-$ of type-$(A_1,A_2\cup A_3,B_2\cup B_4,B)$.
This completes the proof.
\end{proof}

Since $|A^r|+|B^r|\leqslant 8dk+2$, there are at most $9dk^3$ directed triangles containing at least one vertex from $A^r\cup B^r$. Then Claim \ref{CLAIM:type of BBA} implies that
$$
\text{cyc}(B,B,A)\leqslant 9dk^3+\text{cyc}(B_3,B_3,A_3)+\text{cyc}(B_4,B_4,A_4).
$$
It remains to show that
$\text{cyc}(B_3,B_3,A_3),\text{cyc}(B_4,B_4,A_4)\leqslant 17dk^3$.
Recall that
$A_1\cup A_4\cup B_4\Rightarrow A_3$
and
$N_R^{-,+}(uv)\cup N_R^{-,-}(uv)=A_1\cup A_4\cup B_4\subseteq \overline{B_3}$.
Hence, the degree condition $\delta^0(R)\geqslant(1/2-2d)k$ implies that $d^-(x,\overline{B_3})\geqslant d_R^{-,+}(uv)+d_R^{-,-}(uv)-2dk$ 
for each $x\in A_3$.
By Lemma \ref{LEM:-+=+-} and the degree condition $\delta^0(R)\geqslant(1/2-2d)k$, we have
$d_R^{-,+}(uv)=d_R^{+,-}(uv)\pm 8dk$
and
$d_R^{-,-}(uv)=d_R^{+,+}(uv)\pm 8dk$.
Therefore,
\begin{align*}
d^-(x,\overline{B_3})
&\geqslant \bigl(d_R^{-,+}(uv)+d_R^{+,-}(uv)+d_R^{-,-}(uv)+d_R^{+,+}(uv)-16dk\bigr)/2-2dk\\
&\geqslant (k-8dk-2)/2-10dk\\
&=(1/2-15d)k.
\end{align*}
By the degree condition  again, we obtain $d^-(x)\leqslant (1/2+2d)k$, and hence
$d^-(x,B_3)\leqslant 17dk$.
It follows that $e(B_3,A_3)\leqslant 17dk^2$, which in turn implies that
$\text{cyc}(B_3,B_3,A_3)\leqslant 17dk^3$.
By symmetry, considering the out-degree of any vertex $y\in B_4$, we similarly obtain
$\text{cyc}(B_4,B_4,A_4)\leqslant 17dk^3$,
since $B_4\Rightarrow A_1\cup A_3\cup B_3$.
In summary,
$$
\text{cyc}(B,B,A)\leqslant \sqrt d\, k^3,
$$
since $d\ll 1$.
This completes the proof.
\end{proof}


\section{Proof of Lemma \ref{LEM:abcbalance}}\label{SEC-5.3}

Before proving Lemma \ref{LEM:abcbalance}, we introduce some necessary definitions and tools. Their proofs will be given at the end of this section. Given a partition $(V_1,V_2,V_3)$ of $V(G)$, we say that a vertex $v\in V_i$ is \emph{$\gamma$-good} if
$$
d^\pm(v,V_{i\pm 1}) \geqslant |V_{i\pm 1}|-O(\gamma n),
$$
and \emph{$\gamma$-bad} otherwise.

The following lemma shows that every sufficiently large oriented graph $G$ with large minimum semi-degree contains only relatively few $\gamma$-bad vertices in each part of a $\gamma$-extremal partition.

\begin{lemma}\label{LEM:extremal-badvertex}
Let $0<1/n \ll \eta \ll \gamma \ll 1$. 
Suppose that $G$ is an $n$-vertex oriented graph with $\delta^0(G)\geqslant(1/2-\eta)n$.
Then for every $\gamma$-extremal partition of $V(G)$, the number of $\gamma^{1/2}$-bad vertices is at most $\gamma^{1/3}n$.
\end{lemma}

Since the minimum semi-degree in our problem is very large, we are able to reassign all $\gamma$-bad vertices while preserving certain desirable properties of the resulting partition.

\begin{lemma}\label{LEM:fexbad}
    Let $0<1/n\ll \eta \ll\gamma \ll 1$. Suppose that $G$ is an $n$-vertex oriented graph with $\delta^0(G)\geqslant(1/2-\eta)n$. If $G$ is $\gamma$-extremal, then $V(G)$ has a $\gamma^{1/3}$-superextremal partition. 
\end{lemma} 

To finish the proof of Lemma \ref{LEM:abcbalance}, we still need the following embedding lemma, which can be used to adjust the size of each part in a partition.

\begin{lemma}\label{LEM:greedily Dabc}
    Given $a,b,c\in\mathbb{N}$ with $1\leqslant a\leqslant b\leqslant c$, let $0<1/n \ll \eta \ll \beta \ll \gamma \ll 1$. Suppose that $G$ is an $n$-vertex oriented graph with $\delta^0(G)\geqslant(1/2-\eta)n$. If $V(G)$ has a $\gamma$-extremal partition $(V_1,V_2,V_3)$, then for each $i\in[3]$, $G$ contains at least $\beta n$ disjoint copies of $D_{a,b,c}$ of type-$(i,i+1,i+2)$.
\end{lemma}

The following lemma will be used to ``balance'' the size of the sets in the extremal graph. Its proof is given in Appendix \ref{APP:A}.

\begin{lemma}\label{LEM-abc system2}   Given $a,b,c\in \mathbb{N}$ with $1\leqslant a\leqslant b\leqslant c$, let $h=a+b+c$ and $\Delta=c^2-ab$. If $\gcd(h,\Delta)=1$, then the following system of congruences has an integer solution.
    \begin{align}\label{system of abc1}
\begin{cases}
ax + cy + bz \equiv r_1 \pmod h\\
bx + ay + cz \equiv r_2 \pmod h\\
cx + by + az \equiv r_3 \pmod h,
\end{cases}
\end{align} 
where every $r_i$ is an integer and $r_1+r_2+r_3 \equiv 0\pmod h$.
\end{lemma}

Indeed, the systems of congruences in \eqref{system of abc1} are closely related to the different types of $D_{a,b,c}$. The variables $x,y,z$ represent the numbers of copies of the corresponding types in a tiling, while the constants appearing in the systems are closely tied to the balance conditions of the tiling.

\smallskip

Equipped with the above lemmas, we are ready to give the proof of Lemma \ref{LEM:abcbalance}.
Here, we recall that $R_h(x)$ is the remainder of $x$ modulo $h$.

\begin{proof}[\textbf{Proof of Lemma \ref{LEM:abcbalance}}]
Given $a,b,c\in \mathbb{N}$ with $1\leqslant a\leqslant b\leqslant c$, set $h=a+b+c$ and choose
$$0<1/n \ll \eta \ll \gamma \ll 1.$$
Suppose that $G$ is an $n$-vertex oriented graph with $\delta^0(G)\geqslant (1/2-\eta)n$ and that $G$ is $\gamma$-extremal. By Lemma \ref{LEM:fexbad}, there exists a $\gamma^{1/3}$-superextremal partition of $V(G)$, say $(U_1,U_2,U_3)$. Moreover, by Lemma \ref{LEM:greedily Dabc}, $G$ contains linearly many disjoint copies of $D_{a,b,c}$ of type-$(i,i+1,i+2)$ for each $i\in[3]$.

Note that $R_h(|U_1|)\leqslant h-1$, and each copy of $D_{a,b,c}$ of type-$(1,2,3)$ uses exactly $a$ vertices of $U_1$. Thus, after removing a constant-sized family $\mathcal{H}_1$ of disjoint copies of $D_{a,b,c}$ of type-$(1,2,3)$, we have $R_h(|U_1\backslash V(\mathcal{H}_1)|)\leqslant a-1.$ 
For each $i\in[3]$, let $S_i$ be any subset of $U_i\backslash V(\mathcal{H}_1)$ such that
$$|S_i|=R_h(|U_i\backslash V(\mathcal{H}_1)|).$$
Clearly, $|S_1|\leqslant a-1$ and $|S_2|,|S_3|\leqslant h-1$.

If at least one of $|S_2|$ and $|S_3|$ is at most $b+c-1$, then set $\mathcal{H}=\mathcal{H}_1$, let $S=S_1\cup S_2\cup S_3$, and define $V_i=U_i\backslash (S\cup V(\mathcal{H}))$ 
for each $i\in[3]$. We claim that $\mathcal{H}$ and $S$ are the desired tiling and set, respectively. Indeed, $|S|\leqslant (a-1)+(b+c-1)+(h-1)=2h-3$. 
Since $S\cup V(\mathcal{H})$ has constant size, the partition $(V_1,V_2,V_3)$ of
$V(G)\backslash (S\cup V(\mathcal{H}))$ is clearly $\gamma^{1/3}$-superextremal. Moreover, $|V_i|\equiv 0\pmod h$ for each $i\in[3]$. This completes the proof in this case.

Thus we may assume that $|S_2|,|S_3|\geqslant b+c$. Applying Lemma \ref{LEM:greedily Dabc} again, we can find two disjoint copies of $D_{a,b,c}$ in $G-V(\mathcal{H}_1)$, one of type-$(1,2,3)$ and the other of type-$(2,3,1)$. Let $\mathcal{H}_2$ consist of these two copies, and set $\mathcal{H}=\mathcal{H}_1\cup \mathcal{H}_2$. For each $i\in[3]$, let $S_i^{\prime}$ be any subset of $U_i\backslash V(\mathcal{H})$ such that
$|S_i^{\prime}|=R_h(|U_i\backslash V(\mathcal{H})|)$. 
Since $h=a+b+c$ and $a\leqslant b\leqslant c$, we have $|S_1^{\prime}|=h+|S_1|-(a+c)\leqslant a+b-1,|S_2^{\prime}|=|S_2|-(a+b),|S_3^{\prime}|=|S_3|-(b+c)$. 
Moreover, as $|S_2|,|S_3|\leqslant h-1$, it follows that
$|S_2^{\prime}|+|S_3^{\prime}|\leqslant a+c-2$,
 and hence
$$\sum_{i\in[3]}|S_i^{\prime}|\leqslant (a+b-1)+(a+c-2)\leqslant 2h-3.$$
It is straightforward to check that $\mathcal{H}$ and $S_1^{\prime}\cup S_2^{\prime}\cup S_3^{\prime}$ are the desired tiling and set, respectively.

We now prove the ``in particular'' part. Let $S$ be any subset of $V(G)$ with $|S|\leqslant h-1$ such that $|V(G)\backslash S|$ is divisible by $h$, and set
$r_i=R_h(|U_i\backslash S|)$ 
for each $i\in[3]$. Since $\gcd(h,c^2-ab)=1$, the system of congruences \eqref{system of abc1} in Lemma \ref{LEM-abc system2} has an integer solution, say $(x_0,y_0,z_0)$. Note that $(R_h(x_0),R_h(y_0),R_h(z_0))$ is also an integer solution to \eqref{system of abc1}, so we may assume w.l.o.g. that
$0\leqslant x_0,y_0,z_0\leqslant h-1$. 
In particular, $x_0,y_0,$ and $z_0$ are all bounded by a constant. By Lemma \ref{LEM:greedily Dabc}, we can find a $D_{a,b,c}$-tiling $\mathcal{H}$ consisting of $x_0$ copies of $D_{a,b,c}$ of type-$(1,2,3)$, $y_0$ copies of type-$(2,3,1)$, and $z_0$ copies of type-$(3,1,2)$. Let
$V_i=U_i\backslash (S\cup V(\mathcal{H}))$ 
for each $i\in[3]$. It is not difficult to check that
$|U_i\cap V(\mathcal{H})|\equiv r_i\pmod{h}$ 
for each $i\in[3]$, and hence $|V_i|\equiv 0\pmod h$. Therefore, $\mathcal{H}$ and $S$ are the desired $D_{a,b,c}$-tiling and set, which completes the proof.
\end{proof}

As guaranteed, we now complete the proofs of the aforementioned tools in the remainder of this subsection.

\begin{proof}[\textbf{Proof of Lemma \ref{LEM:extremal-badvertex}}]
Let $0<1/n\ll\eta\ll\gamma\ll 1$, and let $G$ be an $n$-vertex oriented graph with $\delta^0(G)\geqslant (1/2-\eta)n$. Fix a $\gamma$-extremal partition $(V_1,V_2,V_3)$ of $V(G)$. Since $\delta^0(G)\geqslant (1/2-\eta)n$, every vertex $v\in V(G)$ misses at most $2\eta n$ edges in $G$. Therefore, for each $i\in[3]$, we have
$$
e(V_i,V_{i+1})\geqslant |V_i|(|V_{i+1}|-2\eta n)-e(V_{i+1},V_i)\overset{\scriptstyle\text{\ref{EP1}}}{\geqslant} |V_i||V_{i+1}|-O(\gamma n^2).
$$
We claim that for each $i\in[3]$, there are at most $\gamma^{1/3}n/6$ vertices $v\in V_i$ with $d^+(v,V_{i+1})\leqslant |V_{i+1}|-\gamma^{1/2}n$. Indeed, this follows immediately from the fact that every such $v$ contributes at least $\gamma^{1/2}n$ missing edges to $V_{i+1}$, together with $\gamma\ll 1$. By symmetry, for each $i\in[3]$, there are at most $\gamma^{1/3}n/6$ vertices in $V_{i+1}$ whose in-degree from $V_i$ is at most $|V_i|-\gamma^{1/2}n$. This completes the proof.
\end{proof}

\begin{proof}[\textbf{Proof of Lemma \ref{LEM:fexbad}}]
Let $0<1/n\ll\eta\ll\gamma\ll 1$, and let $G$ be an $n$-vertex oriented graph with $\delta^0(G)\geqslant (1/2-\eta)n$. Fix a $\gamma$-extremal partition $(V_1,V_2,V_3)$ of $V(G)$. By Lemma \ref{LEM:extremal-badvertex}, the number of $\gamma^{1/2}$-bad vertices with respect to this partition is at most $\gamma^{1/3}n$. Let $v$ be any such bad vertex. 

It follows from the degree condition that $d^+(v,V_i),d^-(v,V_j)\geqslant n/6-O(\gamma n)$ for some $i,j\in[3]$. Moreover, since $v$ has at most $2\eta n$ missing edges in $G$, \ref{EP1} implies that for each $i\in[3]$, either $d^+(v,V_i)$ or $d^-(v,V_i)$ is at least $n/6-O(\gamma n)$. Hence there exists $i\in[3]$ such that $d^-(v,V_{i-1})\geqslant n/6-O(\gamma n)$ and $d^+(v,V_{i+1})\geqslant n/6-O(\gamma n)$. We then reassign $v$ to the set $V_i$.
     
After reassigning all $\gamma^{1/3}n$ bad vertices, we obtain a $\gamma^{1/3}$-superextremal partition of $V(G)$, which proves the lemma.
\end{proof}

\begin{proof}[\textbf{Proof of Lemma \ref{LEM:greedily Dabc}}]
Given $a,b,c\in \mathbb{N}$ with $1\leqslant a\leqslant b\leqslant c$, let $0<1/n\ll\eta\ll \beta\ll\gamma\ll 1$. Suppose that $G$ is an $n$-vertex oriented graph with $\delta^0(G)\geqslant (1/2-\eta)n$ and that $(V_1,V_2,V_3)$ is a $\gamma$-extremal partition of $V(G)$. By \ref{EP1}, we have $|V_i|=n/3\pm O(\gamma n)$ for each $i\in[3]$. 

By Lemma \ref{LEM:extremal-badvertex}, all but at most $\gamma^{1/3}n$ vertices of $V(G)$ are $\gamma^{1/2}$-good with respect to the partition $(V_1,V_2,V_3)$. It is not difficult to check that $(V_i,V_{i+1})$ is $(\gamma^{1/4},1/2)$-regular for each $i\in[3]$. Applying Lemma \ref{LEM:findmanyDabc} to $(V_i,V_{i+1},V_{i+2})$ with parameters $\gamma^{1/4},1/2$ and $(a+b+c)\beta$, there are at least $(a+b+c)\beta n^{a+b+c}$ copies of $D_{a,b,c}$ of type-$(i,{i+1},{i+2})$.
Since each copy of $D_{a,b,c}$ intersects at most $(a+b+c)n^{a+b+c-1}$ other copies of $D_{a,b,c}$, there are at least $\beta n$ disjoint copies of $D_{a,b,c}$ in $G$. This completes the proof.  
\end{proof}

\section{Proof of Lemma \ref{LEM:1bcbalance}}\label{SEC-5.4}

\begin{definition}\label{DEF:robust}
Given $b,c\in\mathbb{N}$ with $1\leqslant b\leqslant c$, let $(V_1,V_2,V_3)$ be a partition of $V(G)$, and let $i,j\in[3]$ be distinct. 
A set $X\subseteq V(G)$ is called  \emph{$\gamma$-transfer for $(V_i,V_j)$} if, for every set $A\subseteq V(G)\setminus X$ with $|A|=O( \gamma n)$, there exists a set $S\subseteq V(G)\setminus A$ such that
\begin{itemize}
\item $|S|\leqslant 3(1+b+c)$, and $G[S]$ contains a $D_{1,b,c}$-factor;
\item $|V_i\setminus S| \equiv |V_i|-1 \pmod{1+b+c}$ and $|V_j\setminus S| \equiv |V_j|+1 \pmod{1+b+c}$.
\end{itemize}
\end{definition}

To prove Lemma \ref{LEM:1bcbalance}, we also need the following result. 

\begin{lemma}\label{LEM-edgegood}
    Given $b,c\in\mathbb{N}$ with $1\leqslant b\leqslant c$, let $0<1/n\ll\gamma\ll 1$. Suppose that $G$ is an $n$-vertex semi-regular tournament. If $(V_1,V_2,V_3)$ is a $\gamma$-superextremal partition of $V(G)$, then the vertex set of every edge in $E(V_{i+1},V_i)$ is $\gamma$-transfer for both $(V_i,V_{i-1})$ and $(V_{i+1},V_{i-1})$.
\end{lemma}

\begin{proof}[\textbf{Proof of Lemma \ref{LEM:1bcbalance}}] 
Given $b,c\in\mathbb{N}$ with $1\leqslant b\leqslant c$, set 
$h=1+b+c$. Let $G$ be a semi-regular tournament on $n\in h\mathbb{N}$ vertices that is $\gamma$-extremal. Since $3\mid h$ and $n\in h\mathbb{N}$, it follows that $3\mid n$. By Lemma \ref{LEM:fexbad}, $V(G)$ has a $\gamma^{1/3}$-superextremal partition, say $(V_1,V_2,V_3)$. Relabeling the parts if necessary, we may assume w.l.o.g. that $V_1$ is the largest part. Moreover, by reversing all edges of $G$ and swapping the labels of $V_2$ and $V_3$ if necessary, we may further assume that $|V_2|\geqslant |V_3|$. Set $|V_1|=n/3+g$, $|V_2|=n/3+f$, and $|V_3|=n/3-g-f$, where $g$ and $f$ are integers since $3\mid n$. It follows from $|V_1|\geqslant |V_2|\geqslant |V_3|$ that $g\geqslant 0$ and $g\geqslant f\geqslant -g-f$.

We denote by $M_{21}$ and $M_{13}$ maximum matchings in $E(V_2,V_1)$ and $E(V_1,V_3)$, respectively. Next we claim that $|M_{21}|\geqslant 2g+f-1$ and $|M_{13}|\geqslant g-f-1$. Considering the out-degrees of vertices in $V_2$, we have 
\begin{align*}
    e(V_2,V_1)&\geqslant \sum_{v\in V_2} d^+(v)-e(V_2)-e(V_2,V_3)\\
    &\geqslant |V_2|\lfloor (n-1)/2\rfloor-|V_2|(|V_2|-1)/2-|V_2||V_3|\\
    &\geqslant |V_2|(g+f/2-1/2).
\end{align*}

Let $H$ be the underlying bipartite graph corresponding to the set of edges from $V_2$ to $V_1$. Clearly, $e(H)\geqslant|V_2|(g+f/2-1/2)$. On the other hand, since $(V_1,V_2,V_3)$ is $\gamma^{1/3}$-superextremal, \ref{EP2} implies that every vertex of $V_2$ has degree at most $n/6 + O(\gamma^{1/3}n)$ in $H$. Since $H$ is bipartite, by K\"{o}nig Theorem, the size of a maximum matching equals the size of a minimum vertex cover in $H$. Thus, we obtain that
$$
e(H)\leqslant |M_{21}|(n/6 + O(\gamma^{1/3}n)).
$$
It follows from $\gamma\ll 1$ that $|M_{21}|>2g+f-2$, and hence $|M_{21}|\geqslant 2g+f-1$. Similarly, by considering the in-degrees of vertices in $V_3$, we obtain $|M_{13}|\geqslant g-f-1$.

Next we claim that there exists a $D_{1,b,c}$-tiling $\mathcal{H}_1$ such that
$$
|V_1\backslash V(\mathcal{H}_1)|\equiv|V_2\backslash V(\mathcal{H}_1)|\equiv|V_3\backslash V(\mathcal{H}_1)|\pmod h.
$$
If $g=0$, then $|V_1|=|V_2|=|V_3|$, and there is nothing to prove. Thus we may assume that $g\geqslant 1$. 

When $f\geqslant 0$, since
$$
|M_{21}|\geqslant 2g+f-1\geqslant g+f\geqslant R_h(g)+R_h(f),
$$
we may first choose $R_h(g)$ edges from $M_{21}$ and then choose another $R_h(f)$ edges from the remaining edges of $M_{21}$. By Lemma \ref{LEM-edgegood}, we obtain that the former $R_h(g)$ edges have vertex sets that are $\gamma$-transfer for $(V_1,V_3)$, while the latter $R_h(f)$ edges have vertex sets that are $\gamma$-transfer for $(V_2,V_3)$.

When $f<0$, we have
$$
|M_{13}|\geqslant g-f-1\geqslant -f\geqslant R_h(-f).
$$
By Lemma~\ref{LEM-edgegood}, we may choose $R_h(-f)$ edges from $M_{13}$ whose vertex sets are $\gamma$-transfer for $(V_1,V_2)$. Moreover,
$$
|M_{21}|\geqslant 2g+f-1\geqslant g=(g+f)+(-f)\geqslant R_h(g+f)+R_h(-f),
$$
so there exist $R_h(g+f)$ edges in $M_{21}$ that are disjoint from the previously chosen $R_h(-f)$ edges in $M_{13}$. The vertex sets of these edges are $\gamma$-transfer for $(V_1,V_3)$.

By Definition~\ref{DEF:robust}, in either case the edges chosen above can be extended to a constant-sized $D_{1,b,c}$-tiling $\mathcal H_1$ such that
$$
|V_1\setminus V(\mathcal H_1)|\equiv |V_2\setminus V(\mathcal H_1)|\equiv |V_3\setminus V(\mathcal H_1)|\equiv kh/3 \pmod h
$$
for some $k\in\{0,1,2\}$. Moreover, $
(V_1\setminus V(\mathcal H_1),\,V_2\setminus V(\mathcal H_1),\,V_3\setminus V(\mathcal H_1))
$
is still a $\gamma^{1/3}$-superextremal partition of $V(G)\setminus V(\mathcal H_1)$.

Recall that $h$ is divisible by $3$. Hence $kh/3$ is an integer. Let
\begin{equation*}
(x,y,z) = \begin{cases}
(kh/3, 0, 2kh/3), & \mbox{if } b\equiv0\pmod{3}; \\
(kh/3, 0, 0), & \mbox{if } b\equiv1\pmod{3}; \\
(kh/3, 2kh/3, 0), & \mbox{if } b\equiv2\pmod{3}.
\end{cases}
\end{equation*}
Note that $x,y$, and $z$ are small integers. By Lemma \ref{LEM:greedily Dabc}, there is a collection $\mathcal{H}_2$ of disjoint copies of $D_{1,b,c}$ in $G- V(\mathcal{H}_1)$, consisting of $x$ copies of $D_{1,b,c}$ of type-$(1,2,3)$, $y$ copies of $D_{1,b,c}$ of type-$(2,3,1)$, and $z$ copies of $D_{1,b,c}$ of type-$(3,1,2)$, respectively. Next we claim that
$$
|V_i\cap V(\mathcal{H}_2)| \equiv kh/3\pmod h
\quad\text{for each } i\in[3].
$$
Here we only consider the case $b\equiv0\pmod{3}$; the other two cases can be verified similarly. 

Since $3\mid(1+b+c)$ and $b\equiv0\pmod{3}$, it follows that $c\equiv2\pmod3$. Recall that in this case $\mathcal{H}_2$ consists of $kh/3$ copies of $D_{1,b,c}$ of type-$(1,2,3)$ and $2kh/3$ copies of $D_{1,b,c}$ of type-$(3,1,2)$. As $b\equiv0\pmod{3}$, we have $2bk/3\in\mathbb{Z}$, and hence $2bkh/3\equiv0\pmod h$. This implies that $
|V_1\cap V(\mathcal{H}_2)|=kh/3+2bkh/3\equiv kh/3\pmod h$. 
Similarly, $|V_2\cap V(\mathcal{H}_2)|=(b+2c)kh/3\equiv kh/3\pmod{h}$ 
and $|V_3\cap V(\mathcal{H}_2)|=(c+2)kh/3\equiv kh/3\pmod{h}$, 
since $b+2c\equiv c+2\equiv1\pmod3$. 

Therefore, $\mathcal{H}=\mathcal{H}_1\cup\mathcal{H}_2$ is the desired $D_{1,b,c}$-tiling. Let $V_i^\prime=V_i\backslash V(\mathcal{H})$. Clearly, $|V_i^\prime|\equiv0 \pmod h$. Moreover, since $|V(\mathcal{H})|$ is constant, it is not difficult to check that $(V_1^\prime,V_2^\prime,V_3^\prime)$ is a $\gamma^{1/3}$-superextremal partition of $V(G)\backslash V(\mathcal{H})$.
\end{proof}

We now turn to the proof of Lemma \ref{LEM-edgegood}.

 \begin{proof}[\textbf{Proof of Lemma \ref{LEM-edgegood}}]

 We may  assume w.l.o.g. that $i=1$. Let $B$ be the set of $\gamma^{1/2}$-bad vertices in $G$. Since $(V_1, V_2, V_3)$ is a $\gamma$-superextremal partition of $V(G)$, we have $|B|\leqslant \gamma^{1/3}n$ by Lemma \ref{LEM:extremal-badvertex}. Moreover, \ref{EP1} shows that $|V_i|=n/3\pm O(\gamma n)$ for each $i\in[3]$. We first show that the vertex set of every edge  in $E(V_2,V_1)$ is $\gamma$-transfer for $(V_1,V_3)$. To see this, we need the following two claims.

\begin{claim}\label{CLAIM:(i,j)-set}
If $w \in V_1$ satisfies $d^{\pm}(w, V_3) \geqslant 2\gamma^{1/4}n$, then every subset $U\subseteq V(G)$ containing $w$ is $\gamma$-transfer for $(V_1,V_3)$.
\end{claim}

\begin{proof}
Let $A$ be any subset of $V(G)\backslash U$ with $|A|=O(\gamma n)$.  Set $W=V_3\backslash(A\cup B)$ and $W^\pm=N^\pm(w,W)$. Clearly, $\{W^-,W^+\}$ is a partition of $W$ as $G$ is a tournament. Moreover, we have  $|W|= |V_3|-|A\cup B|= n/3\pm\gamma^{1/4}n$ due to \ref{EP1} and $|W^\pm|\geqslant d^{\pm}(w,V_3)-|A\cup B|\geqslant\gamma^{1/4}n$.  Note that every vertex $v\in W\subseteq V_3$ is $\gamma^{1/2}$-good for the partition $(V_1,V_2,V_3)$. Then $d^-(v,V_2)\geqslant|V_2|-O(\gamma^{1/2}n)$ and thus $d^+(v,V_2)=O(\gamma^{1/2}n)$.  By \ref{EP1}, we have 
\begin{align*}
    d^+(v,W)&\geqslant d^+(v)-d^+(v,V_1)-d^+(v,V_2)-|A\cup B|\\
    &\geqslant\lfloor\frac{n-1}{2}\rfloor-(n/3+O(\gamma n))-O(\gamma^{1/2}n)-O(\gamma n)-\gamma^{1/3}n\\
    &\geqslant(1/2-2\gamma^{1/3})|W|.
\end{align*}
Similarly, we have  $d^-(v,W)\geqslant(1/2-2\gamma^{1/3})|W|$ and thus $\delta^0(G[W])\geqslant(1/2-2\gamma^{1/3})|W|$. It follows from $|W^\pm|\geqslant \gamma^{1/4} n$ that 
\begin{align*}
e(W^+,W^-)&\geqslant\sum_{v\in W^+}d^+(v,W)-e(W^+)\\
&\geqslant (1/2-2\gamma^{1/3})|W||W^+|-|W^+|^2/2\\
&\geqslant (|W|-2\gamma^{1/3}n -|W^+|)|W^+|/2\\
&\geqslant |W^+||W^-|/4.
\end{align*}

Applying Lemma \ref{LEM:dependentrandom} with $k=2^{b-1}$ and $d=\varepsilon=1/4$,  there exists a set of $2^{b-1}$ vertices in $W^+$ having  at least  $2^{c-1}$ common out-neighbors in $W^-$.  As $\overrightarrow{r}(k)\leqslant 2^{k-1}$, there is a $TT_b$ copy in $W^+$ and a $TT_c$ copy in $W^-$ such that together with $w$, they form a copy of $D_{1,b,c}$. Let $S$ be the vertex set of this  $D_{1,b,c}$ copy. This completes the proof.
\end{proof}

\begin{claim}\label{Claim:three abc}
    For each $xy\in E(V_2,V_1)$, if $d^{-,+}(xy,V_1) \geqslant \gamma^{1/
    4}n$, then $\{x,y\}$ is $\gamma$-transfer for $(V_1,V_3)$.
\end{claim}

\begin{proof}
   Let $A$ be any subset of $V(G)$ with $|A|=O(\gamma n)$ that avoids $x$ and $y$. Set $U_i=V_i\backslash(A\cup B)$ for each $i\in[3]$. Clearly, $d^{-,+}(xy,U_1)\geqslant d^{-,+}(xy,V_1)-|A\cup B|\gg 2^{c-1}$. By $\overrightarrow{r}(c)\leqslant2^{c-1}$ again, there is a $TT_c$ copy in the subgraph of $G$ induced by $N^{-,+}(xy,U_1)$. Let $X$ be the vertex set of the $TT_c$ copy. 
   
   Note that  all vertices in $X$ miss at most $O(\gamma^{1/2} n)$ in-edges from $U_{3}$. Moreover, as $G$ is a tournament, one of $|N^{+,-}(xy,U_3)|$ and $|N^{+,+}(xy,U_3)|$ is at least $|N^{+}(x,U_3)|/2$. We may  assume w.l.o.g. that the former case holds. It follows from  \ref{EP2} that  $d^+(x,U_3)\geqslant n/7$ and thus  $$|N^{+,-}(xy,U_{3})\cap N^{-}(X,U_{3})| \geqslant n/14-O(\gamma^{1/2}n)\gg 2^{b-2}.$$
This implies that there is a copy of $TT_{b-1}$ in $N^{+,-}(xy,U_{3})\cap N^{-}(X,U_{3})$, which forms a $TT_{b}$ with $y$. Let $Y$ be the vertex set of the $TT_{b-1}$ copy. As $Y\subseteq  N^+(x,U_3)\cap N^{-}(X,U_{3})$ and $X\subseteq N^{-,+}(xy,U_1)$, there is a $D_{1,b,c}$ copy, say $T_1$, in the subgraph induced by $X\cup Y\cup \{x,y\}$.

As $|A\cup B|\leqslant 2\gamma^{1/3}n$, the partition $(U_1,U_2,U_3)$ is clearly a $\gamma^{1/3}$-extremal partition of $V(G)\backslash (A\cup B)$. By  Lemma \ref{LEM:greedily Dabc}, there are two disjoint copies $T_2,T_3$ of $D_{1,b,c}$ that avoid $V(T_1)$, where one copy is of type-$(1,2,3)$ and the other is of type-$(3,1,2)$. Set $S=V(T_1\cup T_2\cup T_3)$.  Notice that $x,y\in S$, $|S\cap V_1|=2+b+c$, $|S\cap V_2|=1+b+c$ and $|S\cap V_3|=b+c$. By Definition \ref{DEF:robust}, $\{x,y\}$ is $\gamma$-transfer for $(V_1,V_3)$.
\end{proof}

Consider a fixed edge $xy$ in $E(V_2,V_1)$. Now we claim that $\{x,y\}$ is $\gamma$-transfer for $(V_1,V_3)$.  We may assume that $d^{-,+}(xy,V_1)<\gamma^{1/4}n$ since otherwise there is nothing to prove by Claim~\ref{Claim:three abc}. Furthermore,  if $d^+(y,V_{3})\geqslant2\gamma^{1/4}n$, then Claim \ref{CLAIM:(i,j)-set} and \ref{EP2} show that $\{x,y\}$ is $\gamma$-transfer for $(V_1,V_3)$.  Therefore, we may further assume that $d^+(y,V_{3})<2\gamma^{1/4}n$. Clearly,  $d^+(y,V_1)\geqslant \lfloor(n-1)/2\rfloor-|V_2|-2\gamma^{1/4}n$ and $d^{-,+}(xy,V_1)\geqslant d^-(x,V_1)+d^+(y,V_1)-|V_1|$. Combining \ref{EP1} and $d^{-,+}(xy,V_1)<\gamma^{1/4}n$, we get that $d^-(x,V_1)<n/6+O(\gamma^{1/4}n)$. Thus by \ref{EP1} and \ref{EP2} we have $$d^-(x,V_1)=n/6\pm O(\gamma^{1/4}n) \mbox{ and then } d^+(x,V_1)= n/6\pm O(\gamma^{1/4}n).$$ 

 For any subset $A$ of $V(G)\backslash \{x,y\}$ with $|A|=O(\gamma n)$, set $J=N^+(x,V_1)\backslash (A\cup B)$. It follows from $|A\cup B|\leqslant 2\gamma^{1/3}n$ that $|J|= n/6\pm O(\gamma^{1/4}n)$. As $G[J]$ is oriented, there exists $y^{\ast}\in J$ such that 
$d^+(y^{\ast},J)\leqslant |J|/2\leqslant n/12+O(\gamma^{1/4}n)$. Moreover, since $y^{\ast}\notin B$, we have $d^+(y^{\ast},V_1)\geqslant \lfloor (n-1)/2\rfloor-|V_2|-O(\gamma n)\geqslant n/6-O(\gamma n)$. This  implies that 
\begin{align*}
    d^{-,+}(xy^\ast,V_1)
    &= d^+(y^{\ast},V_1)-d^{+,+}(xy^\ast,V_1)\\
    &\geqslant d^+(y^{\ast},V_1)-d^+(y^{\ast},J)-|A|-|B|\\
    &\geqslant n/6-O(\gamma n)-(n/12+O(\gamma^{1/4}n))-2\gamma^{1/3}n \\
    &\gg \gamma^{1/4}n.
\end{align*}
Applying Claim \ref{Claim:three abc} again, we obtain that $\{x,y^\ast\}$ is $\gamma$-transfer for $(V_1,V_3)$. From Definition \ref{DEF:robust} and the fact that $x, y^\ast \notin A$, there exists a set $S\subseteq V(G)\backslash A$ with the required property. Consequently, $\{x,y\}$ is $\gamma$-transfer for $(V_1,V_3)$.

 Similarly, by computing  $d^\pm(x,V_3)$ and $d^{-,+}(xy,V_2)$, one can prove that $\{x,y\}$ is $\gamma$-transfer for $(V_2,V_3)$, which completes the proof. 
 \end{proof}

\section{Proof of Lemma \ref{LEM:balancefactor}}\label{SEC-5.5}

To prove Lemma \ref{LEM:balancefactor}, we will need the following result.

\begin{theorem}[\cite{johanssonDM211}]\label{LEM:blowup}
Let $G$ be a graph with a vertex partition  $(V_1,V_2,V_3)$  such that $|V_1|=|V_2|=|V_3|=m$. If for every $i\in [3]$ and every $v\in V_i$,
$$
d(v, V_{i+1}),d(v,V_{i-1}) \geqslant \frac{2}{3}m + \sqrt{m},
$$
then $G$ contains a triangle-factor.
\end{theorem}

Now we are ready to prove Lemma \ref{LEM:balancefactor}. 

\begin{proof}[\textbf{Proof of Lemma \ref{LEM:balancefactor}}]
Set $h=a+b+c$. Choose $\varepsilon,d$ and $\eta$ such that
$$0<1/n\ll \varepsilon,d\ll \eta\ll\gamma \ll 1.$$  
Let $B$ be the set of $\gamma^{1/2}$-bad vertices in $G$. By Lemma \ref{LEM:extremal-badvertex}, we have $|B|\leqslant \gamma^{1/3}n$. We divide the proof into the following three steps.

\medskip
\textbf{Step 1. Cover bad vertices by a $D_{a,b,c}$-tiling.}
\medskip

Let $S$ be any subset of $V(G)\backslash B$ with $|S|\leqslant \gamma^{1/4}n$, and let $v$ be any vertex in $B$. In this step, we show that $G$ has a $D_{a,b,c}$-tiling $\mathcal{H}_v$ satisfying the following properties.

   (i) $v\in V(\mathcal{H}_v)$ and $V(\mathcal{H}_v)\backslash\{v\}\subseteq V(G)\backslash (S\cup B)$;
   
   (ii) $\mathcal{H}_v$ consists of three disjoint copies of $D_{a,b,c}$;

   (iii) $|V(\mathcal{H}_v)\cap V_1|=|V(\mathcal{H}_v)\cap V_2|=|V(\mathcal{H}_v)\cap V_3|$.
   
Set $\widetilde{V}_i= V_i\backslash (S\cup B)$ for each $i\in [3]$. It follows from \ref{EP1} and $|S\cup B|\leqslant 2 \gamma^{1/4}n$ that $|\widetilde{V}_i|=n/3\pm O(\gamma^{1/4}n)$. Suppose that $v$ belongs to $V_j$ for some $j\in [3]$. Since $(V_1,V_2,V_3)$ is $\gamma$-superextremal, by \ref{EP2} we have
$$
d^+(v,\widetilde{V}_{j+1}),d^-(v,\widetilde{V}_{j-1})\geqslant n/6-O(\gamma n)-|S\cup B|\geqslant n/7.
$$
Set $V^\ast_{j+1}=N^+(v,\widetilde{V}_{j+1})$ and $V^\ast_{j-1}=N^-(v,\widetilde{V}_{j-1})$. By the degree condition, $v$ misses at most $2\eta n$ edges in $G$. Then one of $|N^+(v,\widetilde{V}_{j})|$ and $|N^-(v,\widetilde{V}_{j})|$ is at least $(|\widetilde{V}_j|-2\eta n)/2\geqslant n/7$ by \ref{EP1}. Let $V^\ast_j$ be the larger of these two neighborhoods. 

Note that every vertex in $V^\ast_1\cup V^\ast_2\cup V^\ast_3$ is $\gamma^{1/2}$-good. Thus, for each $u\in V^\ast_i$ and $i\in[3]$, we have
$$
d^+(u,V^\ast_{i+1})\geqslant |V^\ast_{i+1}|-O(\gamma^{1/2} n)
 \text{ and } 
d^-(u,V^\ast_{i-1})\geqslant |V^\ast_{i-1}|-O(\gamma^{1/2} n).
$$
It is easy to verify that $(V_i^\ast,V_{i+1}^\ast)$ is $(\gamma,1/2)$-regular for  $i\in[3]$. Applying Lemma \ref{LEM:findmanyDabc} to $(V_i^\ast,V_{i+1}^\ast,V_{i+2}^\ast)$ for each $i$, there exists $\beta>0$ such that it contains at least $\beta n^{a+b+c}$ copies of $D_{a,b,c}$ of type-$(i,i+1,i+2)$. Since each copy of $D_{a,b,c}$ intersects at most $(a+b+c)n^{a+b+c-1}$ others, we can choose three pairwise disjoint copies $H_1,H_2,H_3$ in $G[V^\ast_1\cup V^\ast_2\cup V^\ast_3]$ such that $H_i$ is of type-$(i,i+1,i+2)$ for each $i\in[3]$.

Recall that $v\in V_j\cap B$ for some $j\in[3]$. Thus, $v\notin V_j^{\ast}$, and hence $v$ does not belong to $H_1\cup H_2\cup H_3$. By construction, $v$ either dominates all vertices in $V_j^{\ast}$ or is dominated by all of them. Hence, for any $x\in V(H_1)\cap V_j^{\ast}$, there is a copy $H_1^\ast$ of $D_{a,b,c}$ in $G[(H_1-x)\cup \{v\}]$. Then $\{H^\ast_1,H_{2},H_{3}\}$ is the desired $D_{a,b,c}$-tiling $\mathcal{H}_v$ satisfying (i)-(iii). 

We carry out the above process sequentially for each $v\in B$, defining $S$ to be the set of vertices in $V(G)\backslash B$ that are used in the previously constructed $D_{a,b,c}$-tilings. Clearly, $|S|=O(\gamma^{1/3}n)\ll\gamma^{1/4}n$, since $|B|\leqslant \gamma^{1/3}n$ and $|V(\mathcal{H}_v)|=3h$ for each $v\in B$. Therefore, we can find a $D_{a,b,c}$-tiling $\mathcal{H}_0$ of size $3h|B|$ consisting of the three tiles in $\mathcal{H}_v$ for each $v\in B$. Note that each $\mathcal{H}_v$ uses exactly $h$ vertices from each $V_i$. Thus $|V_i\backslash V(\mathcal{H}_0)|\equiv 0 \pmod {h}$ for each $i\in [3]$. 

\medskip
    
Set $V_i^0=V_i\backslash V(\mathcal{H}_0)$ for each $i\in [3]$. It follows from \ref{EP1} and $|V(\mathcal{H}_0)|=O(\gamma^{1/3}n)$ that $|V_i^0|=n/3\pm O(\gamma^{1/3}n)$. Moreover, every vertex in $V_i^0$ is $\gamma^{1/2}$-good for each $i\in[3]$.
   
\medskip
\textbf{Step 2. Balance $|V_1^0|,|V_2^0|$, and $|V_3^0|$ by a $D_{a,b,c}$-tiling.}
\medskip

By relabeling the sets if necessary, we may assume w.l.o.g. that $|V_3^0|\leqslant |V_1^0|,|V_2^0|$. Set $k_i=(|V_i^0|-|V_3^0|)/h$ for $i\in [2]$. Since $|V^0_{i}|=n/3\pm O(\gamma^{1/3} n)$ and $|V^0_{i}|\equiv 0\pmod h$ for each $i\in[3]$, we have $k_1,k_2 = O(\gamma^{1/3} n)$, and both $k_1$ and $k_2$ are integers. Next we embed a $D_{a,b,c}$-tiling $\mathcal{H}_i$ of size $k_i$ in $G[V_i^0]$ for each $i\in[2]$ such that
$$
|V^0 _1\backslash V(\mathcal{H}_1)|=|V^0_2\backslash V(\mathcal{H}_2)|=|V^0_3|\equiv 0\pmod h.
$$

For each $v\in V_1^0$, since it is $\gamma^{1/2}$-good, we have $d^-(v,V_2), d^+(v,V_3)=O(\gamma^{1/2} n)$. Combining \ref{EP1} with the degree condition, we obtain
$$
d^\pm(v,V_1)\geqslant (n/2-\eta n)- O(\gamma^{1/2}n)-(n/3+O(\gamma n))\geqslant n/6-O(\gamma^{1/2}n).
$$
Thus $d^\pm(v,V^0_1)\geqslant d^\pm(v,V_1)-|V(\mathcal{H}_0)|
\geqslant |V^0_1|/2-O(\gamma ^{1/3}n)$ as $|V(\mathcal{H}_0)|=O(\gamma^{1/3}n)$ and $|V_1^0|=n/3\pm O(\gamma^{1/3}n)$. We apply the Diregularity Lemma to $G[V^0_1]$ with $\varepsilon$, $M^\prime = 1/\varepsilon$, $M^{\prime\prime}=1$, and $d$. By Lemma \ref{LEM-reducedori}, there exists a reduced oriented graph $R$ with vertex set $[k]$ and
$
\delta^0(R)\geqslant (1/2-O(\gamma^{1/3})-d-3\varepsilon)k.
$
Applying Lemma \ref{LEM:cross-edge-triangle} (i) to $R$, there is an edge from $N^+_R(1)$ to $N^-_R(1)$, which together with $1$ forms a directed triangle in $R$. Note that Lemma \ref{LEM:cross-edge-triangle} holds for all sufficiently large oriented graphs with $\delta^0(G)\geqslant (1/2-1/100)|G|$. Thus we can greedily find at least $\gamma^{1/4}k$ disjoint triangles in $R$, since $\eta \ll \varepsilon \ll d\ll \gamma \ll 1/100$.

Let $123$ be one such fixed triangle in $R$, and let $X_1,X_2,X_3$ be the corresponding sets in $G[V_1^0]$. Then for each $i\in[3]$, $(X_i, X_{i+1})$ is $(\varepsilon, d)$-regular and $|X_i|=m$ for some integer $m$. Next we claim that there is a $D_{a,b,c}$-tiling $\mathcal{D}$ in $G[X_1\cup X_2\cup X_3]$ covering at least $2m$ vertices. Moreover, $\mathcal{D}$ uses an equal number of vertices from $X_1,X_2$, and $X_3$. Indeed, let $\mathcal{D}$ be initially empty and set $X_i^{\prime}=X_i\backslash V(\mathcal{D})$. By the same argument as in Step 1, as long as $|X_i^{\prime}|> m/3$, Lemma \ref{LEM:findmanyDabc} implies that there are three disjoint copies of $D_{a, b, c}$ in $G[X^\prime_1\cup X^\prime_2\cup X^\prime_3]$ such that they use exactly $h$ vertices from each $X_i^{\prime}$, where $i\in[3]$. Note that this is possible because $(X^\prime_i, X^\prime_{i+1})$ is $(3\varepsilon,d/2)$-regular by $|X_i^{\prime}|> m/3$ and Lemma \ref{LEM-regulartosuper}.

Therefore, $G[V_1^0]$ has a $D_{a,b,c}$-tiling covering at least $\gamma^{1/4}k\cdot 2m$ vertices. Recall that $km\geqslant |V_1^0|-\varepsilon |V_1^0|\geqslant n/4$ and that $k_1 = O(\gamma^{1/3} n)$. Then one can find the desired tiling $\mathcal{H}_1$ of size $k_1$ in $G[V_1^0]$. Similarly, there is a $D_{a,b,c}$-tiling $\mathcal{H}_2$ of size $k_2$ in $G[V_2^0]$.

Set $\mathcal{H}=\mathcal{H}_0\cup \mathcal{H}_1\cup \mathcal{H}_2$. Let $V_i^{\prime}=V_i\backslash V(\mathcal{H})$ for each $i\in[3]$. By the construction above, we have $|V_1^{\prime}|=|V_2^{\prime}|=|V_3^{\prime}|\equiv0\pmod h$ and $|V(\mathcal{H})|=O(\gamma^{1/3}n)$.

\medskip
\textbf{Step 3. Find a $D_{a,b,c}$-factor in the resulting graph $G[V_1^{\prime}\cup V_2^{\prime}\cup V_3^{\prime}]$.}
\medskip

Since $|V^\prime _1|=|V^\prime _2|=|V^\prime _3|\equiv0\pmod h$, there exists an integer $t$ such that $|V^\prime _i|=ht$ for each $i\in[3]$. Moreover, it follows from $|V(\mathcal{H})|=O(\gamma^{1/3}n)$ that
$$
t=(n-|V(\mathcal{H})|)/(3h)\geqslant n/(4h).
$$
For each $i\in[3]$, let $\{V^a_i,V^b_i,V^c_i\}$ be any partition of $V^\prime _i$ with $|V^j_i|=jt$ for each $j\in\{a,b,c\}$. Set $W_i=V^a_i\cup V_{i+1}^b\cup V^c_{i+2}$, where the indices are taken modulo $3$. Then $(W_1,W_2,W_3)$ is a partition of $V(G)\backslash V(\mathcal{H})$. By symmetry, it suffices to show how to find a $D_{a,b,c}$-factor in $G[W_1]$. 

Recall that Yuster \cite{yusterorder20} showed that every $jt$-vertex oriented graph $G$ with $\delta(G)\geqslant(1-4^{-j})jt+4^j$ contains a $TT_j$-factor. Since $\delta^0(G)\geqslant (1/2-\eta)n$ and $|V^a_1|=at\geqslant an/(4h)$, we have
$
\delta(G[V_1^a])\geqslant |V_1^a|-2\eta n\geqslant (1-8\eta h/a )|V_1^a|.
$
It follows from $\eta \ll 1$ that $G[V_1^a]$ contains a $TT_a$-factor $\mathcal{T}_a$. By a similar argument, there is a $TT_b$-factor $\mathcal{T}_b$ in $G[V_2^b]$ and a $TT_c$-factor $\mathcal{T}_c$ in $G[V_3^c]$. Now we construct an auxiliary graph $H$ as follows. Let
$
V(H)=X_a\cup X_b\cup X_c,
$
where each vertex of $X_i$ corresponds to a copy of $TT_i$ in $\mathcal{T}_i$ for each $i\in\{a,b,c\}$. Clearly, $|X_a|=|X_b|=|X_c|=t$. For simplicity, for each $u\in X_i$ with $i\in\{a,b,c\}$, let $T_u$ be the copy of $TT_i$ to which $u$ corresponds. For each $u\in X_i$ and $v\in X_j$ with $(i,j)\in\{(a,b),(b,c),(c,a)\}$, we put $uv\in E(H)$ if and only if every vertex of $T_u$ dominates every vertex of $T_v$ in $G$.

Recall that every vertex in $V(G)\backslash V(\mathcal{H})$ is $\gamma^{1/2}$-good. Then every copy of $TT_i$ with $i\in\{a,b,c\}$ has at most $O(\gamma^{1/2} n)$ missing edges in $G$. This shows that for each $u\in X_i$, both its in-degree and out-degree in $H$ are at least
$
t-O(\gamma^{1/2} n)\geqslant (1-\gamma^{1/3})t,
$
since $t\geqslant n/(4h)$. Applying Lemma \ref{LEM:blowup} to the underlying graph of $H$, we obtain a triangle-factor in $H$, which corresponds to a $D_{a,b,c}$-factor in $G[W_1]$ by the construction of $H$. By the same argument, each of $G[W_2]$ and $G[W_3]$ contains a $D_{a,b,c}$-factor. Together with the former tiling $\mathcal{H}$, this yields a $D_{a,b,c}$-factor of $G$. This completes the proof.
\end{proof}

\section{Concluding remarks}\label{SEC-remark}

In \cite{keevashJCTB99}, Keevash and Sudakov raised the problem of finding $C_3$-tilings in oriented graphs. In this paper, we extend their result by determining when an oriented graph contains a $D_{a,b,c}$-factor or a $D_{a,b,c}$-tiling covering all but at most a constant number of vertices. Our proof combines the Diregularity method, the lattice-based absorption method, and a tailored structural analysis for handling certain problematic vertices.

A natural question is whether the semi-degree conditions considered in this paper can also guarantee the existence of factors for other graphs. Among the most extensively studied problems in this direction is the existence of $TT_k$-factors. For $k=3$, Balogh, Lo, and Molla \cite{baloghJCTB124} showed that $\delta^{0}(G)\geqslant 7n/18$ suffices. The minimum total degree threshold for $TT_4$-factors was asymptotically determined by DeBiasio, Lo, Molla, and Treglown in \cite{debiasioSJDM35}. It would be interesting to determine the exact minimum semi-degree threshold and the exact minimum total degree threshold for $TT_k$-factors with $k\geqslant 5$.

One of the central problems in digraph theory, the Caccetta--H\"aggkvist Conjecture, states that every $n$-vertex oriented graph with minimum out-degree $d$ contains a directed cycle of length at most $n/d$. For the existence of a directed triangle, Shen \cite{shenJCTB74} proved that minimum out-degree $0.355n$ suffices, and Hamburger, Haxell, and Kostochka \cite{hamburgeEJC14} showed that, if one considers minimum semi-degree instead of minimum out-degree, then this constant can be improved slightly: they proved that $\delta^0(G)\geqslant 0.346n$ guarantees a directed triangle. The related natural question of what minimum semi-degree forces $C_l$ for $l\geqslant 4$ was raised in \cite{kellyJCTB100}.

\begin{conjecture}[\cite{kellyJCTB100}]
    Let $l$ be an integer with $l\geqslant 4$, and let $k$ be the smallest integer greater than $2$ that does not divide $l$.
    Then there exists $n_0$ such that every oriented graph $G$ on $n\geqslant n_0$ vertices with $\delta^{0}(G)\geqslant \lfloor n/k\rfloor+1$ contains a copy of $C_l$.
\end{conjecture}

In \cite{kellyJCTB100}, Kelly et al.\ proved the conjecture exactly for $k=3$, and asymptotically for $k=4$ and $l\geqslant 42$, as well as for $k=5$ and $l\geqslant 2550$. They also showed that a bound of $\lfloor n/3\rfloor+1$ suffices for any $l\geqslant 4$. In \cite{kuhnEJC34}, K\"uhn, Osthus, and Piguet proved the conjecture asymptotically in the case where $k\geqslant 7$ and $l$ is sufficiently large compared to $k$. Recently, Grzesik and Volec \cite{rzesikIMRN2023} determined the exact semi-degree threshold for $C_l$ with $l\geqslant 4$ in sufficiently large oriented graphs. In particular, they showed that the threshold in the conjecture is correct only when either $l$ is not divisible by $3$ or $l\equiv 3 \pmod{12}$.

Inspired by the celebrated P\'osa--Seymour Conjecture, we consider the existence of $C_l^k$ for each $l\geqslant 2k+1$\footnote{The definition of an oriented graph requires $l\geqslant 2k+1$.}. The smallest case is $k=2$. When $l\notin 6\mathbb{N}$, we give a construction of an $n$-vertex $4n/9$-regular oriented graph that is $C_l^2$-free.  The case $l=6$ was also raised by DeBiasio et al.\ in \cite{debiasioCPC35}. The currently best lower bound in this case was previously given by Araujo and Xiang. In our opinion, the lower bounds in both cases are correct, and proving the matching upper bounds should be a challenging problem. Before settling this problem completely, it may be easier to determine the minimum semi-degree threshold for $P_l^2$ for each $l\geqslant 4$. We note that $P_3^2=TT_3$. A classical result in oriented Ramsey theory states that every tournament on $4$ vertices contains a copy of $TT_3$. Hence, by a standard calculation, every oriented graph $G$ with $\delta(G)\geqslant 2n/3+1$ contains a copy of $P_3^2$. Meanwhile, the balanced blow-up of $C_3$ shows that this degree condition is tight. It is surprising that there exists an $n$-vertex $3n/8$-regular oriented graph containing no $P_l^2$ for any $l\geqslant 7$. This phenomenon further suggests that determining the minimum semi-degree threshold for $P_l^2$ for every $l\geqslant 4$ is an interesting problem.

\bibliographystyle{plain}

\newpage
\appendix
\section{Proof of Lemma \ref{LEM-abc system2}}
\label{APP:A}

Before proving Lemma \ref{LEM-abc system2}, we formulate and prove the characterization for the existence of integer solutions to the systems of congruences in \eqref{system of abc}.

\begin{lemma}\label{PRO-abc system}
    Given $a,b,c\in \mathbb{N}$ with $1\leqslant a\leqslant b\leqslant c$, let $h=a+b+c$ and $\Delta=c^2-ab$. Then the following system of congruences has an integer solution if and only if $\gcd(h,\Delta)=1$. 
    \begin{align}\label{system of abc}
\begin{cases}
ax + cy + bz \equiv 1 \pmod h\\
bx + ay + cz \equiv-1 \pmod h\\
cx + by + az \equiv 0 \pmod h
\end{cases}
\end{align}
\end{lemma}
\begin{proof}
\noindent Note that $(c^2-ab)-(a^2-bc)=(c-a)(a+b+c)=(c-a)h$. Hence
$h\mid\bigl((c^2-ab)-(a^2-bc)\bigr)$, and therefore
$c^2-ab\equiv a^2-bc \pmod h$. Similarly,
$c^2-ab\equiv b^2-ac\pmod h$.

First assume that $\gcd(h,\Delta)=1$. Then there exists a unique integer
$\Delta^{-1}\in [h-1]$ such that
$\Delta\cdot \Delta^{-1}\equiv 1\pmod h$.
We claim that
$(x_0,y_0,z_0)=\bigl((a+c)\Delta^{-1},-(a+b)\Delta^{-1},0\bigr)$
is an integer solution to \eqref{system of abc}. Indeed,
$ax_0+cy_0+bz_0=\Delta^{-1}(a^2-bc)\equiv \Delta^{-1}\Delta\equiv 1\pmod h$,
so the first congruence in \eqref{system of abc} holds. The second and third congruences can be verified similarly.

Now assume that \eqref{system of abc} has an integer solution $(x_1,y_1,z_1)$.
Suppose for a contradiction that $g=\gcd(h,\Delta)>1$.
Since the congruences in \eqref{system of abc} hold modulo $h$ and $g\mid h$,
reducing them modulo $g$ yields
\begin{align*}
\begin{cases}
    ax_1+cy_1+bz_1\equiv 1 \pmod g,\\
    bx_1+ay_1+cz_1\equiv -1 \pmod g,\\
    cx_1+by_1+az_1\equiv 0 \pmod g.
\end{cases}
\end{align*}

Using the first two congruences, we obtain
$(a^2-bc)x_1+(ab-c^2)z_1
= a(ax_1+cy_1+bz_1)-c(bx_1+ay_1+cz_1)
\equiv a+c\pmod g$.
Since $c^2-ab\equiv a^2-bc\pmod h$ and $g=\gcd(h,c^2-ab)$, it follows that
$c^2-ab\equiv a^2-bc \equiv 0\pmod g$.
Hence $a+c\equiv(c^2-ab)(x_1-z_1)\equiv 0\pmod g$.

Similarly,
$c(ax_1+cy_1+bz_1)-b(bx_1+ay_1+cz_1)\equiv b+c\pmod g$,
and together with $c^2-ab\equiv b^2-ac\pmod h$ \and $g=\gcd(h,c^2-ab)$, this implies that
$b+c\equiv 0\pmod g$.
Note that we also have
$h=a+b+c\equiv 0\pmod g$.
Therefore,
$a\equiv b\equiv c\equiv 0\pmod g$.
Substituting these congruences into the first congruence above yields
$0\equiv 1\pmod g$, a contradiction.
Thus $g=\gcd(h,\Delta)=1$, completing the proof.
\end{proof}

Now we are ready to prove Lemma \ref{LEM-abc system2}.

\medskip
\noindent\textbf{Lemma \ref{LEM-abc system2}.}\textit{ Given $a,b,c\in \mathbb{N}$ with $1\leqslant a\leqslant b\leqslant c$, let $h=a+b+c$ and $\Delta=c^2-ab$. If $\gcd(h,\Delta)=1$, then the following system of congruences has an integer solution.
\begin{equation*}\tag{\ref{system of abc1}}\label{eq:A}
\begin{cases}
ax + cy + bz \equiv r_1 \pmod h,\\
bx + ay + cz \equiv r_2 \pmod h,\\
cx + by + az \equiv r_3 \pmod h,
\end{cases}
\end{equation*}
where every $r_i$ is an integer and $r_1+r_2+r_3 \equiv 0\pmod h$.}

\begin{proof} By Lemma \ref{PRO-abc system}, the system \eqref{system of abc} has an integer solution $(x_0,y_0,z_0)$. Let $x_1,y_1,z_1$ be positive integers such that
$x_1\equiv r_1x_0-r_3z_0\pmod h$, $y_1\equiv r_1y_0-r_3x_0\pmod h$, and $z_1\equiv r_1z_0-r_3y_0\pmod h$.
We claim that $(x_1,y_1,z_1)$ is the desired solution to \eqref{system of abc1}. We verify only the second congruence, since the other two can be checked similarly:
\begin{align*}
    bx_1+ay_1+cz_1&\equiv b(r_1x_0-r_3z_0)+a(r_1y_0-r_3x_0)+c(r_1z_0-r_3y_0)\\
    &\equiv r_1(bx_0+ay_0+cz_0)-r_3(ax_0+cy_0+bz_0)\\
    & \equiv r_1\times(-1)-r_3\times 1\\
    &\equiv r_2\pmod h,
\end{align*}
where the last congruence follows from $r_1+r_2+r_3\equiv 0\pmod h$.
\end{proof}

\section{Proofs of Propositions \ref{PROP:manyF},  \ref{Prop:Cl2-lowerbound}, \ref{PRO-abcexamplegraph}, and \ref{PRO-1bcexamplegraph}} \label{APP:B}

  An $n$-vertex digraph is \emph{round} if we can label its vertices $v_1, v_2, \ldots, v_n$ so that for each $i$, we have $N^+(v_i) = \{v_{i+1}, \ldots, v_{i+d^+(v_i)}\}$ and $N^-(v_i) = \{v_{i-d^-(v_i)}, \ldots, v_{i-1}\}$, where all subscripts are taken modulo $n$. All (semi)-regular tournaments in this section can be regarded as (semi)-regular round tournaments. The verification of constructions lacking specific structure relies primarily on meticulous calculation.

\medskip
\noindent\textbf{Proposition \ref{PROP:manyF}.} \textit{There exists an infinite family of Tur\'anable oriented graphs, none of which is a subgraph of $D_s$ for any $s\in\mathbb{N}$.}

\begin{proof}
Given $a, b, c\in\mathbb{N}$, let $H^\prime$ be an oriented graph on $a+b+c$ vertices which is obtained from a blow-up of $C_3$ by replacing its three vertex classes $V_1,V_2,V_3$ with three \emph{acyclic} oriented graphs on $a$, $b$, and $c$ vertices, respectively.  Suppose that $P$ is a (directed) path in $H^\prime[V_3]$ starting from $v$ and ending at $u$. Let $H$ be the oriented graph obtained from $H^\prime$ by adding a new vertex $w$ and two new edges  $uw, wv$. See Figure \ref{FIG-DF} for an illustration. 

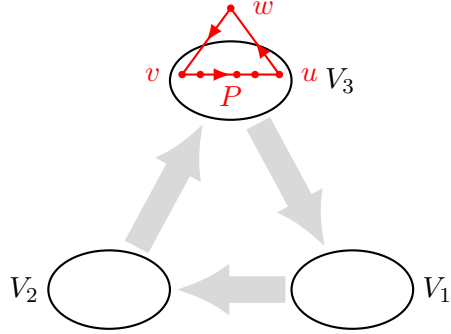
\begin{figure}[!ht]
\centering
    \begin{minipage}{0.45\textwidth}
        
        \begin{tikzpicture}[thick,scale=0.8] 

            \draw (0,0) ellipse (1 and 0.65);
            \draw (60:-4)ellipse (1 and 0.65);
            \draw (120:-4) ellipse (1 and 0.65);
            
              \filldraw[red](0,1.2) circle (1.25pt)node[label=right: $w$](w){};

              \filldraw[red](-0.5,0.1) circle (1.25pt);
              \filldraw[red](0.1,0.1) circle (1.25pt);
              \filldraw[red](0.4,0.1) circle (1.25pt);
              
            \draw[middlearrow={latex}, red] (0,1.2)-- (-0.8,0.1);
             \draw[middlearrow={latex}, red] (0.8,0.1)-- (0,1.2);
             \filldraw[red](0.8,0.1) circle (1.25pt)node[label=right: $u$](u){};
             \filldraw[red](-0.8,0.1) circle (1.25pt)node[label=left: $v$](v){};
 \node at (0,-0.25) [red] {$P$};
        \draw[middlearrow={latex}, red] (-0.8,0.1)-- (0.8,0.1);
        
            \draw[-{Latex[length=7mm,width=7mm]},line width = 10pt,gray!30](0.45135,-0.70855)--(2-0.45135,-3.4641+0.70855);
            \draw[-{Latex[length=7mm,width=7mm]},line width = 10pt,gray!30] (-2+0.45135,-3.4641+0.70855)-- (-0.45135,-0.70855);
            \draw[-{Latex[length=7mm,width=7mm]},line width = 10pt,gray!30] (0.9,-3.4641)--(-0.9,-3.4641);
 \node at (1.8,0) [black] {$V_3$};
            \node at (3.4,-3.4641) [black] {$V_1$};
            \node at (-3.4,-3.4641) [black] {$V_2$};
        \end{tikzpicture}
        \end{minipage}
        \caption{The oriented graph  $H$ in Proposition  \ref{PROP:manyF},  which is Tur\'anable but is not a subgraph of $D_s$ for any $s\geqslant 1$. }\label{FIG-DF}
    \end{figure}

We first claim that $H$ is Tur\'anable, in other words, there is a copy of $H$ in every sufficiently large regular tournament $G$. Note that $H^{\prime}$ (i.e., $H-w$) is a subgraph of $D_{a,b,c}$ and $D_{a,b,c}$ is Tur\'anable by Theorem \ref{THM:bollobasJCTB50}. Therefore,  $H^{\prime}$ is Tur\'anable. Let $H_G^{\prime}$ be a copy of $H^\prime$ in $G$.  Set $U= V(G)\backslash V(H_G^{\prime})$.  To build a copy of $H$ in $G$, we only need to find a vertex in $U$ which plays the role of $w$ in $H$.  Recall that $H[V_3]$ is an acyclic oriented graph and $P$ is a path in $H[V_3]$ starting from $v$ and ending at $u$. Then $d^+(u,V_3)+d^-(v,V_3)\leqslant |V_3|-2$. Let $u_G$ and $v_G$ be the copies of $u$ and $v$ in $G$, respectively. Note that $d^+(u_G, U)=d^+(u_G)-d^+(u_G, H_G^{\prime})=(|G|-1)/2-|V_1|-d^+(u, V_3)$. Similarly, $d^-(v_G, U)=(|G|-1)/2-|V_2|-d^-(v, V_3)$. Therefore,  we have
\begin{equation*}
\begin{aligned}
|N^+(u_G, U)\cap N^-(v_G, U)| &\geqslant d^+(u_G, U)+d^-(v_G,U)-|U| \\
&\geqslant |G|-1-|V_1|-|V_2|-(|V_3|-2)-|U|\\
&= 1.
\end{aligned}
\end{equation*}
This implies that $w$ can be embedded into $U$ and thus $H$ is Tur\'anable.

To finish the proof of Proposition \ref{PROP:manyF}, it remains to prove that $D_s$ does not contain $H$ as subgraph for any $s\in \mathbb{N}$. Suppose to the contrary that $H$ can be embedded into $D_s$ for some $s\in \mathbb{N}$. Let $U_1, U_2$ and $U_3$ be the three vertex classes of $D_s$. Observe that every directed triangle $C_3$ of $D_s$ has exactly one vertex in $U_i$ for each $i\in[3]$.  Let $x, y$ be any two vertices of $H$ with $x\in V_2$ and $y\in V_1$. By the construction of $H$, both $xvy$ and $xuy$ are directed triangles in $H$.  This implies that $u$ and $v$ belong to the same set in $D_s$, say $U_1$. Since $P$ and $w$ form a (directed) cycle of $H$ and $D_s[U_1]$ is acyclic, the vertex $w$ must lie in $U_j$ with $j\neq 1$. Then there is an edge $uw$ from $U_1$ to $U_j$ and an edge $wv$ from $U_j$ to $U_1$ in $D_s$, a contradiction. Therefore, $H$ cannot be contained in any $D_s$.
\end{proof}

Next, we prove Proposition \ref{Prop:Cl2-lowerbound}. The argument starts with two regular oriented graphs of constant orders that exclude the relevant squares of paths and cycles. The desired constructions are then obtained by taking arbitrary blow-ups of these base graphs.

\medskip
\noindent\textbf{Proposition \ref{Prop:Cl2-lowerbound}.} 
\textit{$\kappa^0(C_l^2)\geqslant4/9$ for all $l\not\equiv 0 \pmod 6$, and $\kappa^0(P_l^2)\geqslant3/8$ for all $l\geqslant 7$.}

\textit{Moreover, there exists a $4n/9$-regular $($resp., $3n/8$-regular$)$ oriented graph $G$ containing no $C_l^2$ $($resp., $P_l^2$$)$.}

\begin{proof}

We first prove that $\kappa^0(C_l^2)\geqslant4/9$ for all $l\not\equiv 0 \pmod 6$. Let $F$ be an oriented graph which consists of three disjoint directed triangles $H_1,H_2,H_3$ such that all edges between distinct triangles go from $V(H_1)$ to $V(H_2)$, from $V(H_2)$ to $V(H_3)$ and from $V(H_3)$ to $V(H_1)$. Let $F(t)$ be the $t$-blow-up of $F$. More precisely, $V(F(t))=\{v_j:v\in V(F)\text{ and }j\in[t]\}$ and 
$E(F(t))=\{u_iv_j:uv\in E(F)\text{ and }i,j\in[t]\}$. It is not difficult to check that $F(t)$ is a $4t$-regular oriented graph on $9t$ vertices. For convenience, let $H_i(t)$ be the $t$-blow-up of $H_i$ for each $i\in[3]$. Suppose that there is a copy of $C_l^2$ in $F(t)$. Next we show that $6$ must divide $l$.

  Assume  $V(C_l^2)=\{v_1, v_2, \ldots, v_l\}$ and  let $v_1v_2\cdots v_l$ be the Hamilton cycle of $C_l^2$.  It holds that $\{v_i, v_{i+1}, v_{i+2}\}$ induces a copy of $TT_3$, where the indices are taken modulo $l$. Set $T_i=C_l^2[\{v_i, v_{i+1}, v_{i+2}\}]$. Since every $H_j(t)$ is a $t$-blow-up of a directed triangle, it has at most two vertices of $T_i$. We now claim  that 
$$\mbox{if }v_i\in H_j(t), v_{i+1}\in H_{j+1}(t),\mbox{ then }v_{i+2}\in H_{j+1}(t) \mbox{ and } v_{i+3}\in H_{j+2}(t),$$ 
where the indices $i$ and $j$ are taken modulo $l$ and $3$, respectively. 
Since $T_i=C_l^2[\{v_i, v_{i+1}, v_{i+2}\}]$ is a transitive triangle and all edges between $H_j(t)$ and $H_{j+1}(t)$ go from $H_j(t)$ to $H_{j+1}(t)$, it holds that $v_{i+2}\in H_{j+1}(t)$. 
Meanwhile, as $H_{j+1}(t)$ contains at most two vertices in $V(T_{i+1})=\{v_{i+1},v_{i+2},v_{i+3}\}$, it holds that $v_{i+3}\in H_{j+2}(t)$, which proves the claim. Observe that the claim shows that if $C_l^2\subseteq F(t)$, then $6$ must divide $l$. Thus, 
$\kappa^0(C_l^2)\geqslant4/9$ for all $l\not\equiv 0 \pmod 6$.

Next, we show that $\kappa^0(P_l^2)\geqslant 3/8$ for every $l\geqslant 7$ by constructing a $3t$-regular oriented graph on $8t$ vertices that contains no square of a path of length greater than $6$. Let $H$ be the oriented graph in Figure \ref{FIG-38}. Set $V_1 = \{x_2, x_3, x_4\} $, $V_2 = \{x_5, x_6, x_7\}$ and $V_3 = \{x_1, x_8\}$. It holds that every edge between $V_i$ and $V_{i+1}$ goes from $V_i$ to $V_{i+1}$ for each $i\in[3]$. Note that $V_3$ is an independent set.

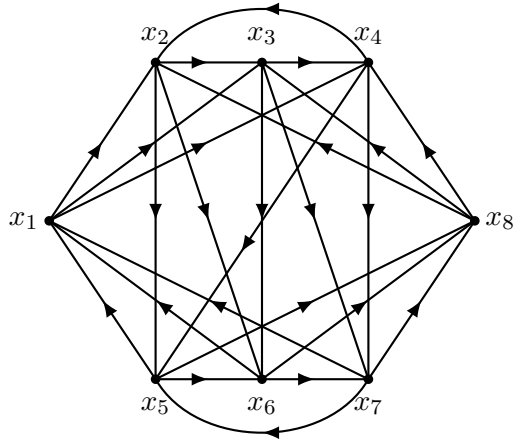
\begin{figure}[!ht]
        \centering
        \begin{tikzpicture}[thick,scale=0.7] 
            
            \coordinate [label=left:$x_1$] (x1) at (-4,0);
            \coordinate  (x2) at (-2,3);
            \coordinate  (x3) at (0,3);
            \coordinate  (x4) at (2,3);
            \coordinate  (x5) at (-2,-3);
            \coordinate  (x6) at (0,-3);
            \coordinate  (x7) at (2,-3);
            \coordinate [label=right:$x_8$] (x8) at (4,0);

            \node at (-2,3.5) [black] {$x_2$};
            \node at (2,3.5) [black] {$x_4$};
            \node at (-2,-3.5) [black] {$x_5$};
            \node at (2,-3.5) [black] {$x_7$};
            \node at (0,3.5) [black] {$x_3$};
            \node at (0,-3.5) [black] {$x_6$};
            
            \foreach \i/\j in {x1/x2,x1/x3,x1/x4,x2/x3,x3/x4,x8/x2,x8/x3,x8/x4,x5/x1,x6/x1,x7/x1,x5/x8,x6/x8,x7/x8,x5/x6,x6/x7}{\draw[middlearrow={latex}, black] (\i)-- (\j);}

            \foreach \i/\j in {x2/x5,x2/x6,x3/x6,x3/x7,x4/x7}{\draw[middlearrow={latex}] (\i)-- (\j);}
                     
            \draw[midddlearrow={latex}] (x4)-- (x5);

            \draw [middlearrow={latex}, black] (x4) [out=120,in=60] to (x2);
            \draw [middlearrow={latex}, black] (x7) [out=-120,in=-60] to (x5); 

            \foreach \i in {x1,x2,x3,x4,x5,x6,x7,x8}{\filldraw[black](\i) circle (2pt);}
        \end{tikzpicture}
    \caption{The oriented graph $H$ in Proposition \ref{Prop:Cl2-lowerbound}.}\label{FIG-38}
    \end{figure}

    We first show that $H$ contains no square of a cycle. Suppose for a contradiction that $C_l^2\subseteq H$ for some $l\in\mathbb{N}$. Since $H$ is oriented, we must have $l\geqslant 5$. Let $V(C_l^2)=\{y_1,y_2,\ldots,y_l\}$, and let $y_1y_2\cdots y_l$ be the Hamilton cycle of $C_l^2$. Since every edge between $V_i$ and $V_{i+1}$ goes from $V_i$ to $V_{i+1}$ for each $i\in[3]$, we have $V(C_l^2)\cap V_i\neq\emptyset$ for each $i\in[3]$. We may w.l.o.g.  assume that $y_1\in V_3$.  Since $V_3$ is independent and all edges between $V_i$ and $V_{i+1}$ are directed from $V_i$ to $V_{i+1}$ for each $i\in[3]$, it follows that $y_2\in V_1$ and $y_l\in V_2$. However, $C_l^2[\{y_l,y_1,y_2\}]$ is a copy of $TT_3$, so in particular $y_ly_2\in E(H)$, giving an edge from $V_2$ to $V_1$, a contradiction. Therefore, $H$ contains no square of a cycle.

    Let $P_s^2$ be a longest square of a path in $H$. Write $V(P_s^2)=\{z_1,z_2,\ldots,z_s\}$, where $z_1z_2\cdots z_s$ is the Hamilton path of $P_s^2$. Since $H[\{x_1,x_2,x_3,x_6,x_7,x_8\}]$ contains a square of a path of length $6$, we have $s\geqslant 6$, and hence $|V(P_s^2)\cap V_1|\geqslant 1$ and $|V(P_s^2)\cap V_2|\geqslant 1$. We now show that $|V(P_s^2)\cap V_i|\leqslant 2$ for each $i\in[2]$. Suppose otherwise. By symmetry, we may assume that $V_1\subseteq V(P_s^2)$.  By the definition of the square of a path, these three vertices cannot appear consecutively on $P_s^2$. Let $z_i,z_j,z_k$ be the three vertices in $V(P_s^2)\cap V_1$. By reversing all edges of $P_s^2$ if necessary, we may assume that $i<j\leqslant k-2$. Recall that $V_3$ is an independent set and that all edges between $V_i$ and $V_{i+1}$ are directed from $V_i$ to $V_{i+1}$ for each $i\in[3]$. Since $z_{k-1}\notin V_1$, it follows that $z_{k-1}\in V_3$ and hence $z_{k-2}\in V_2$. As $P_s^2[\{z_{k-2},z_{k-1},z_k\}]$ is a copy of $TT_3$, we have $z_{k-2}z_k\in E(H)$, giving an edge from $V_2$ to $V_1$, a contradiction. Therefore, $|V(P_s^2)\cap V_i|\leqslant 2$ for each $i\in[2]$, and consequently $s\leqslant 6$. 
     
     Therefore, $H$ contains neither a square of a cycle nor a square of a path of length greater than $6$. Let $G$ be the $t$-blow-up of $H$. Then $G$ is a $3t$-regular oriented graph on $8t$ vertices and still contains no square of a path of length greater than $6$. Hence, $\kappa^0(P_l^2)\geqslant 3/8$ for all $l\geqslant 7$.
\end{proof}

Next, by proving Propositions \ref{PRO-abcexamplegraph} and \ref{PRO-1bcexamplegraph}, we show that the conditions in Theorems \ref{THM-abc} and \ref{THM-1bc} are, in some sense, best possible.

\medskip
\noindent \textbf{Proposition \ref{PRO-abcexamplegraph}.} 
\textit{Let $a,b,c\in\mathbb{N}$ with $2\leqslant a\leqslant b\leqslant c$. If $\gcd(a+b+c,c^2-ab)> 1$, then there exists a semi-regular tournament on $n\in(a+b+c)\mathbb{N}$ vertices  without $D_{a,b,c}$-factors.}

\begin{proof}
Given $a,b,c\in \mathbb{N}$ with $2\leqslant a\leqslant b\leqslant c$, set $h=a+b+c$. Recall that Araujo and Xiang \cite{araujoARXIV2025} showed that for every $s\geqslant 2$, there exists a semi-regular tournament $T$ on $n\in s\mathbb{N}$ vertices with no $D_s$-factor. Hence, we may assume that $c\geqslant 3$.

Let $k\in \mathbb{N}$, and let $T^\prime$ be the blow-up of $C_3$ obtained by replacing its vertices with three semi-regular tournaments $T_1,T_2,T_3$ on $2hk+1$, $2hk-1$, and $2hk$ vertices, respectively. As mentioned at the beginning of this section, we may regard $T_1,T_2$, and $T_3$ as round tournaments. Set $V_i=V(T_i)$ for each $i\in[3]$. Since $T_3$ is a semi-regular tournament on $2hk\in 2\mathbb{N}$ vertices, there is a partition $V_3^+,V_3^-$ of $V_3$ with $|V_3^+|=|V_3^-|=hk$ such that every vertex in $V_3^+$ (resp., $V_3^-$) has out-degree (resp., in-degree) exactly $hk$ in $T_3$. Let $S\subseteq V_1$ with $|S|=hk$, and let $T$ be the oriented graph obtained from $T^\prime$ by reversing a perfect matching from $V_3^+$ to $S$, see Figure~\ref{FIG-prop16}. Let $M$ denote the resulting matching from $V_1$ to $V_3$. Note that $T$ is a semi-regular tournament on $6hk$ vertices.

\begin{figure}[!ht]
\centering
        \begin{tikzpicture}[thick,scale=0.7] 
        \path (-5,0) (5,-5);

            \draw (0,0.3) ellipse (1.6 and 1);
            \draw [rotate=70] (60:-5)ellipse (1.5 and 1);
            \draw [rotate=110](120:5) ellipse (1.5 and 1);

            \draw [rotate=70](-2.5,-4.33+1)--(-2.5,-4.33-1);
            \draw [rotate=110] (-2.5,4.33-1)--(-2.5,4.33+1);

            \draw[-{Latex[length=2mm,width=1.67mm]},line width = 1pt,red] (-3,-4.1)--(3,-4.1);
            \draw[-{Latex[length=2mm,width=1.67mm]},line width = 1pt,red] (-3,-4.5)--(3,-4.5);
            \draw[-{Latex[length=2mm,width=1.67mm]},line width = 1pt,red] (-3,-4.9)--(3,-4.9);
            \node at (0,-5.3) [red] {$M$};
          
            \draw[-{Latex[length=7mm,width=7mm]},line width = 10pt,gray!30](0.75135,-0.90855)--(2.5-0.55135,-4.33+1.00855);
            \draw[-{Latex[length=7mm,width=7mm]},line width = 10pt,gray!30] (-2.5+0.55135,-4.33+1.00855)-- (-0.75135,-0.90855);
            \draw[-{Latex[length=7mm,width=7mm]},line width = 10pt,gray!30] (1.8,-4.33+0.80855)--(-1.8,-4.33+0.80855);
 \node at (2.1,0.3) [black] {$T_2$};
            \node at (4.9,-4) [black] {$T_3$};
            \node at (-4.9,-4) [black] {$T_1$};
        \end{tikzpicture}
        \caption{The semi-regular tournament $T$ in Proposition \ref{PRO-abcexamplegraph}.}
        \label{FIG-prop16}
    \end{figure}
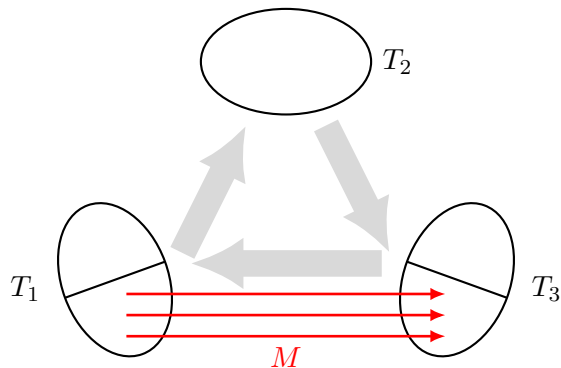

Let $D^\ast$ be any fixed copy of $D_{a,b,c}$ in $T$, and let $A,B,C$ be its three parts, where $|A|=a$, $|B|=b$, and $|C|=c\geqslant 3$. We claim that one of the following holds:

\begin{itemize}
    \item[(i)] $D^\ast$ is of type-$(i)$ for some $i\in[3]$;
    \item[(ii)] $D^\ast$ is of type-$(i,i+1,i+2)$ for some $i\in[3]$.
\end{itemize}

We first show that all vertices of $A$ lie in the same set. Suppose for a contradiction that there exist $a_1\in A\cap V_i$ and $a_2\in A\cap V_{i+1}$ for some $i\in[3]$. Recall that $(V_1,V_2,V_3)$ is obtained from a blow-up of $C_3$ by reversing all edges of a matching, and let $M^+(u)$ (resp., $M^-(u)$) denote the out-neighbor (resp., in-neighbor) of $u$ in $M$, whenever it exists. Since $(A,B,C)$ is a blow-up of $C_3$, every vertex of $C$ must lie in the common in-neighborhood of $a_1$ and $a_2$ in $T$. Hence
$$
C\subseteq N^-(\{a_1,a_2\},T)\subseteq V_i\cup (M^-(a_1)\cap V_{i+1})\cup (M^-(a_2)\cap V_{i+2}).
$$
Since $M$ is a matching between $V_1$ and $V_3$, at least one of $M^-(a_1)\cap V_{i+1}$ and $M^-(a_2)\cap V_{i+2}$ is empty. As $|C|=c\geqslant 3$, it follows that there exist two vertices $c_1,c_2\in C\cap V_i$.

Similarly,
$$
B\subseteq V_{i+1}\cup (M^+(a_1)\cap V_{i+2})\cup (M^+(a_2)\cap V_i).
$$
Since $|B|=b\geqslant 2$ and $M$ is a matching, there exists a vertex $b_0\in B\cap V_{i+1}$. Then $b_0c_1,b_0c_2\in E(T)$, giving two edges from the same vertex of $V_{i+1}$ to $V_i$. This is impossible by the construction of $T$. Therefore, $A\subseteq V_i$ for some $i\in[3]$.

An analogous argument shows that $B\subseteq V_j$ for some $j\in[3]$. It is then straightforward to check that one of (i) and (ii) must hold.

Therefore, every copy of $D_{a,b,c}$ in $T$ is either of type-$(i)$ for some $i\in[3]$, or of type-$(i,i+1,i+2)$ for some $i\in[3]$. Suppose now that $T$ contains a $D_{a,b,c}$-factor $\mathcal H$. Let $x,y,z$ be the numbers of copies in $\mathcal H$ of type-$(1,2,3)$, type-$(2,3,1)$, and type-$(3,1,2)$, respectively, and let $t_i$ be the number of copies of type-$(i)$ for each $i\in[3]$. Counting vertices in $V_1,V_2,V_3$, we obtain
\begin{align*}
\begin{cases}
ax+cy+bz+ht_1=2hk+1,\\
bx+ay+cz+ht_2=2hk-1,\\
cx+by+az+ht_3=2hk.
\end{cases}
\end{align*}
Reducing these equations modulo $h$, we see that $(x,y,z)$ is an integer solution to \eqref{system of abc}. However, by Lemma~\ref{PRO-abc system} and the assumption $\gcd(a+b+c,c^2-ab)> 1$, the system \eqref{system of abc} has no integer solution. This contradiction shows that $T$ contains no $D_{a,b,c}$-factor, which completes the proof.
\end{proof}

\medskip
\noindent \textbf{Proposition \ref{PRO-1bcexamplegraph}.} 
\textit{Let $b,c\in \mathbb{N}$ with $1\leqslant b\leqslant c$.  If $3\nmid (1+b+c)$ and $\gcd(1+b+c,c^2-b)> 1$, then there exists a semi-regular tournament on $n\in(1+b+c)\mathbb{N}$ vertices without $D_{1,b,c}$-factors.}

\begin{proof}
Given $b,c\in \mathbb{N}$ with $1\leqslant b\leqslant c$, set $h=1+b+c$ and
$r_1=\lceil h/3\rceil$, $r_2=\lfloor h/3\rfloor$, and
$r_3=h-\lceil h/3\rceil-\lfloor h/3\rfloor$.
Let $k \in \mathbb{N}$, and let $T$ be a blow-up of directed triangle $C_3$ obtained by replacing its vertices with three semi-regular tournaments $T_1, T_2, T_3$ on $hk+r_1$, $hk+r_2$, and $hk+r_3$ vertices, respectively. Set $V_i=V(T_i)$. It is not difficult to check that $T$ is a semi-regular tournament on $(3k+1)h$ vertices. Moreover, by an argument similar to that used in the proof of Proposition \ref{PRO-abcexamplegraph}, every copy of $D_{1,b,c}$ in $T$ is either of type-$(i)$ or of type-$(i,i+1,i+2)$ with respect to $(V_1,V_2,V_3)$ for some $i\in[3]$. Thus, the existence of a $D_{1,b,c}$-factor $\mathcal{H}$ in $T$ implies the existence of an integer solution to the following system of congruences:
\begin{align}\label{system of 1bc}
\begin{cases}
x+cy+bz\equiv r_1 \pmod h\\
bx+y+cz\equiv r_2 \pmod h\\
cx+by+z\equiv r_3 \pmod h,
\end{cases}
\end{align}
where $x,y,$ and $z$ are the numbers of copies of $D_{1,b,c}$ in $\mathcal{H}$ that have type-$(1,2,3)$, type-$(2,3,1)$, and type-$(3,1,2)$, respectively.

Set $\Delta=c^2-b$. Assume that $3\nmid h$ and $\gcd(h, \Delta)> 1$. Let $p$ be a prime divisor of $\gcd(h,\Delta)$. Clearly, $p>1$, 
and $h,\Delta\equiv 0\pmod p$. Since $h=1+b+c$, we have $c\equiv-1-b\pmod{p}$, and hence
$\Delta=c^2-b\equiv 1+b+b^2\equiv0\pmod{p}$.
This implies that $b^2\equiv-1-b\equiv c\pmod{p}$, and thus $b^3 \equiv bc \pmod{p}$. Moreover, as $b^3=(b-1)(1+b+b^2)+1$, we have
$bc\equiv b^3 \equiv  1 \pmod{p}$. Since $1+b+b^2\equiv 0\pmod{p}$, we must have $b\not\equiv 0,1\pmod{p}$. Indeed, if $b\equiv 0\pmod{p}$, then $1\equiv 0\pmod{p}$, forcing $p=1$, contrary to $p>1$. If $b\equiv 1\pmod{p}$, then $3\equiv 0\pmod{p}$, so $p=3$, which is impossible since $3\nmid h$ while $p\mid h$.

Next we claim that the semi-regular tournament $T$ has no $D_{1,b,c}$-factor by showing that \eqref{system of 1bc} has no integer solution. Suppose to the contrary that $(x_0,y_0,z_0)$ is an integer solution to \eqref{system of 1bc}. Since the congruences in \eqref{system of 1bc} hold modulo $h$ and $p\mid h$, reducing them modulo $p$ yields
\begin{align*}
\begin{cases}
    x_0+cy_0+bz_0\equiv r_1 \pmod p\\
bx_0+y_0+cz_0\equiv r_2 \pmod p\\
cx_0+by_0+z_0\equiv r_3 \pmod p.
\end{cases}
\end{align*}
Since $bc \equiv 1 \pmod p$ and $b^2 \equiv c \pmod p$, we have
$b(x_0+cy_0+bz_0)\equiv bx_0+y_0+cz_0\pmod{p}$,
and thus
$br_1\equiv r_2 \pmod{p}$.
Similarly, we have $br_2\equiv r_3 \pmod{p}$ and $br_3\equiv r_1 \pmod{p}$.

Recall that $r_1=\lceil h/3\rceil$, $r_2=\lfloor h/3 \rfloor$, and
$r_3=h-\lceil h/3\rceil-\lfloor h/3\rfloor$.
Since $3\nmid h$, it follows that $h\equiv1$ or $2\pmod{3}$.
In the former case, we have $r_1-1=r_2=r_3$. As
$br_2\equiv r_3 =r_2 \pmod{p}$,
we obtain
$(b-1)r_2\equiv0\pmod{p}$.
Then $r_2\equiv0\pmod{p}$, since $b\not\equiv1\pmod{p}$ and $p$ is prime. Similarly, from
$br_1\equiv r_2\equiv0\pmod{p}$ and $b\not\equiv0\pmod{p}$,
we obtain
$r_1\equiv0\pmod{p}$.
It follows from $r_1=r_2+1$ that
$1\equiv 0\pmod p$,
a contradiction.
Thus it remains to consider the case $h\equiv2\pmod{3}$.
In this case, we have $r_1=r_2+1=r_3$. As
$br_3\equiv r_1 =r_3 \pmod{p}$,
we have
$(b-1)r_3\equiv0\pmod{p}$.
Then $r_3\equiv0\pmod{p}$, since $b\not\equiv1\pmod{p}$ and $p$ is prime. Again, from
$br_2\equiv r_3\equiv0\pmod{p}$ and $b\not\equiv0\pmod{p}$,
we obtain
$r_2\equiv0\pmod{p}$.
Then $r_3=r_2+1$ implies that
$1\equiv 0\pmod p$,
a contradiction.
Therefore, \eqref{system of 1bc} has no integer solution, and thus $T$ has no $D_{1,b,c}$-factor. This completes the proof.
\end{proof}


\end{document}